\documentclass[a4paper,reqno]{amsart}

\usepackage[a4paper,left=3cm,right=3cm,top=3cm,bottom=3cm]{geometry}
\usepackage{graphicx} 
\usepackage{tikz}
\usetikzlibrary{patterns,shapes,decorations.pathmorphing,decorations.pathreplacing,calc,arrows,cd}
\usetikzlibrary{tikzmark, fit, matrix, decorations.markings}
\usepackage{mathdots}
\usepackage{amssymb}
\usepackage{amstext}
\usepackage{amsmath}
\usepackage{amscd}
\usepackage{amsthm}
\usepackage{amsfonts}

\usepackage{array}
\usepackage{enumerate}
\usepackage{latexsym}
\usepackage{mathrsfs}
\usepackage{hyperref}

\usepackage[all,color]{xy}
\usepackage{mathtools}

\usepackage{pgfplots}
\pgfplotsset{compat=1.16}

\usepackage{booktabs}
\usepackage{multirow}
\xyoption{all}
\usepackage{pstricks}
\usepackage{booktabs}
\usepackage{multirow}
\usepackage{mathtools}
\mathtoolsset{showonlyrefs=true}
\usepackage{todonotes}
\usepackage{appendix}

\numberwithin{equation}{section}
\newtheorem{theorem}{Theorem}[section]

\newtheorem{corollary}[theorem]{Corollary}
\newtheorem{lemma}[theorem]{Lemma}
\newtheorem{proposition}[theorem]{Proposition}

\theoremstyle{definition}
\newtheorem{definition}[theorem]{Definition}

\newtheorem{remark}[theorem]{Remark}
\newtheorem{example}[theorem]{Example}

\newtheorem{convention}[theorem]{Convention}

\usepackage{comment}

\newcommand{\Hom}{\operatorname{Hom}\nolimits}

\newcommand{\End}{\operatorname{End}\nolimits}

\newcommand{\im}{\operatorname{im}\nolimits}

\renewcommand{\ker}{\operatorname{ker}\nolimits}

\DeclareMathOperator{\Rep}{\mathsf{Rep}}

\DeclareMathOperator{\Ker}{{\rm Ker}}
\renewcommand{\Im}{{\rm Im}}

\DeclareMathOperator{\Fun}{Fun}

\DeclareMathOperator{\Vect}{\mathsf{Vect}}

\DeclareMathOperator{\colim}{\mathsf{colim}}

\DeclareMathOperator{\Lan}{\mathrm{Lan}}
\DeclareMathOperator{\Ran}{\mathrm{Ran}}

\DeclareMathOperator{\R}{\mathsf{R}}
\renewcommand{\L}{\operatorname{\mathsf{L}}\nolimits}
\DeclareMathOperator{\E}{\mathsf{E}}
\DeclareMathOperator{\EN}{{\rm EN}}
\DeclareMathOperator{\T}{\mathsf{T}}

\renewcommand{\lim}{\operatorname{\mathsf{lim}}\nolimits}

\DeclareMathOperator{\id}{\mathrm{id}}

\usepackage{float}
\newcommand{\qand}{\operatorname{\quad \text{and} \quad}\nolimits}

\begin{document}
\title[Interleaving distances from height-difference functions on posets]{Interleaving distances from height-difference functions on posets}

\author[Toshitaka Aoki]{Toshitaka Aoki}
\address{Toshitaka Aoki, Graduate School of Human Development and Environment, Kobe University, 3-11 Tsurukabuto, Nada-ku, Kobe 657-8501 Japan}
\email{toshitaka.aoki@people.kobe-u.ac.jp}

\begin{abstract}
Interleaving distances provide a fundamental tool for comparing persistence modules and have been widely used in topological data analysis. 
Their definitions are typically based on translation structures (shift operations) on the indexing poset, but on general posets such structures can be scarce, making this framework restrictive.

In this paper, we introduce a new interleaving-type distance for functor categories over arbitrary posets, induced by a \emph{height-difference function} $\rho$. 
The key idea is to associate to $\rho$ an $\mathbb{R}_{\geq 0}$-indexed family of adjoint endofunctors on $\Fun(P,\mathcal{C})$, which play the role of generalized translations and allow us to formulate interleavings in a purely categorical manner and define the distance $d_{\rho}$, called the height-interleaving distance. 
In particular, any height function (i.e., a real-valued order-preserving map) canonically induces such a height-difference function, so our framework remains useful on finite posets.
Moreover, when $P=\mathbb{R}^d$ and $\rho=\rho_{\mathrm{diag}}$, the resulting distance coincides with the classical interleaving distance for multiparameter persistence modules.
However, in general, $d_{\rho}$ need not satisfy the triangle inequality. Under suitable hypotheses 
(e.g. CIP for $(P,\rho)$; this condition is automatic when $P$ is a tree poset), we prove a triangle inequality up to an additive defect bounded by a constant $c(\rho)$; in particular, when $c(\rho)=0$ this yields an extended pseudo-distance.
We also establish a stability property with respect to perturbations of height-difference functions: small changes in $\rho$ induce small changes in the associated height-interleaving distance.
Finally, we study an analogue of the erosion-type constructions from classical interleavings within our framework. 
\end{abstract}


\maketitle

\tableofcontents

\section{Introduction}
\label{sec:intro}
Persistent homology \cite{ELZ02,CZ05} is a method in topological data analysis that studies the evolution of topological features (such as connected components and holes) along a nested filtration of spaces, and summarizes this information in a persistence diagram, i.e., a multiset of birth–death pairs. 
It has found applications across a range of domains, including materials science \cite{HNH+16}, evolutionary biology \cite{CCR13}, image analysis \cite{QTT+19}, 
and others \cite{HKWNU17,AAF19}.
In applications, a notable aspect of this framework is the availability of stability results: small perturbations of the input lead to controlled changes in the output, 
supporting the use of persistent homology on real-world data \cite{CSEH07,CCSGGO09}.
From an algebraic viewpoint, persistent homology gives rise to persistence modules, that is, functors from a poset to the category of vector spaces.
The interleaving distance formulates such comparisons at the level of persistence modules and plays a key role in establishing these stability results \cite{CCSGGO09,Lesnick15}.
With the growing interest in multiparameter persistent homology, quantitative comparison has become increasingly relevant.

\medskip
\noindent {\bf Background: classical interleaving distance.}
The notion of interleaving provides a concept of proximity between objects in a category, typically defined in terms of a family of shift operations parameterized by $\epsilon\ge 0$.
An $\epsilon$-interleaving can be viewed as an approximate isomorphism, and the infimum of such $\epsilon$ defines the associated interleaving distance.
In particular, a $0$-interleaving is exactly an isomorphism, while two non-isomorphic objects may be $\epsilon$-interleaved for arbitrary small $\epsilon>0$, and hence can have interleaving distance zero.
This viewpoint has been formulated in several settings, including persistence modules in the sense of \cite{BdSS15,BCM20}, categories with a flow \cite{dSMS18}, and more general 2-categorical settings \cite{MN26}.

We focus on the category of persistence modules, and more generally on the category $\Fun(P,\mathcal{C})$ of functors from a poset $P$ to a category $\mathcal{C}$.
In this setting, the interleaving distance on $\Fun(P,\mathcal{C})$ is built from a superlinear family $\Omega$ of translations on $P$, 
consisting of translations $\Omega_\epsilon\colon P\to P$ (order-preserving endomaps with $p \leq \Omega_\epsilon(p)$) for $\epsilon\ge 0$, together with coherent structure $\Omega_\epsilon\circ\Omega_\delta \le \Omega_{\epsilon+\delta}$.
In the case $P=\mathbb{R}^d$ (with the product order), the standard choice is the diagonal shift $\Omega_\epsilon(x)=x+(\epsilon,\ldots,\epsilon)$, and the resulting interleaving distance has been studied extensively, 
including 
fundamental aspects \cite{Lesnick15,GM22,BL20,BS25}, 
stability results \cite{BL24,Bjerkevik21}, 
computational complexity \cite{BBK20}, 
and connections to other distances and invariants \cite{MP20,CKM24,KS25}. 
In addition, we refer the reader to \cite{BCM20,EMY23,Bjerkevik25,MM17} for recent work on interleavings for persistence modules over posets.

This translation-based interleaving framework depends on the internal order structure of $P$, and on general posets 
suitable translations may be scarce. 
In particular, when a translation has a fixed point (e.g., on finite posets), an interleaving condition imposes an isomorphism at that point; otherwise, one cannot have an interleaving, which may lead to an infinite interleaving distance.
Thus, the availability of suitable translations plays a key role in this framework.
This naturally raises the question of whether one can develop a more flexible framework for quantitative comparison that remains useful even on finite posets.

\medskip
\noindent {\bf Our framework and results.}
In this paper, we introduce a new interleaving-type distance on $\Fun(P,\mathcal{C})$ built from an $\mathbb{R}_{\ge 0}$-indexed adjoint family of endofunctors.
The input is a height-difference function $\rho$ on $P$, i.e.\ a nonnegative function on comparable pairs satisfying a superadditivity condition along chains 
(see Definition~\ref{def:height-difference}).
For example, any height function $\phi \colon P\to\mathbb{R}$ (i.e., a real-valued order-preserving map) canonically induces a height-difference function.
Given $\rho$, we associate an adjoint pair 
$\L_r^{\rho}\dashv \R_r^{\rho}$ for each $r\ge 0$ (Proposition~\ref{prop:adjoint LrRr}),
\[
\L_r^{\rho}\colon \Fun(P,\mathcal{C}) \rightleftarrows \Fun(P,\mathcal{C}) \colon \R_r^{\rho}.
\]

These functors are inspired by latching and matching constructions for diagrams: 
$\L_r^{\rho}$ aggregates information over the $\rho$-controlled lower $r$-neighborhood of each point via a colimit, while $\R_r^{\rho}$ aggregates over the corresponding upper $r$-neighborhood via a limit.
In this sense, they play the role of generalized shifts at scale $r$.
This allows us to define an interleaving-type diagram (with respect to $\rho$) between $M$ and $N$:
\begin{equation}\label{eq:intro-interleaving}
\xymatrix@C=36pt@R=26pt{
\L_r^\rho M \ar[r] \ar[dr]^(0.4){p^{\sharp}} & M \ar[r] \ar[dr]^(0.4){p} & \R_r^\rho M \\ 
\L_r^\rho N \ar[r] \ar[ur]^(0.3){q^{\sharp}} & N \ar[r] \ar[ur]^(0.3){q} & \R_r^\rho N, 
}
\end{equation}
where $p^{\sharp}$ and $q^{\sharp}$ denote the corresponding morphisms of $p$ and $q$ under the adjunction $\L_r^\rho\dashv \R_r^\rho$, 
and the horizontal arrows are the canonical maps arising from the (co)limit constructions.
Then $d_\rho(M,N)$ is defined as the infimum of $r\ge 0$ for which such a diagram exists.
We refer to $d_\rho$ as the height-interleaving distance with respect to $\rho$.
As we will discuss later, it need not be a genuine distance in general, since the triangle inequality may fail.

The aim of this paper is to develop height-interleaving distances as tools for analyzing functor categories $\Fun(P,\mathcal{C})$ and to study their basic properties.
Firstly, we show that our distance recovers the classical interleaving distance on $\Fun(\mathbb{R}^d,\mathcal{C})$. 
Indeed, when $P=\mathbb{R}^d$ and $\rho=\rho_{\mathrm{diag}}$, the functors $\R_r^{\rho}$ and $\L_r^{\rho}$ agree with the usual $r$-shift and $(-r)$-shift, so our interleaving diagrams coincide with the classical ones and hence so does the interleaving distance (Proposition~\ref{prop:recoverRd}).

Secondly, we establish two stability properties for our height-interleaving distance, allowing us to compare this quantity even across different indexing posets and different choices of height-difference functions. 
The pullback stability (Proposition~\ref{prop:pullback stability}) shows that height-interleavings are functorial with respect to pullbacks along order-preserving maps; in particular, the associated distance is non-increasing under pullback.
On the other hand, we prove the following functional stability result, asserting that small changes of the height-difference function lead to small changes of the associated height-interleaving distance. 
Here, $\delta(\rho_1,\rho_2)$ denotes the distortion between $\rho_1$ and $\rho_2$. 

\begin{theorem}[Theorem~\ref{thm:functional stability}]
\label{mainthm:functional stability}
The assignment $\rho\mapsto d_{\rho}$ is $1$-Lipschitz with respect to $\delta$.
That is, for any height-difference functions $\rho_1,\rho_2$ on $P$,
\begin{equation}
\mathrm{dist}_{\infty}(d_{\rho_1},d_{\rho_2}) \leq \delta(\rho_1,\rho_2).
\end{equation}
\end{theorem}

Thirdly, we discuss the triangle inequality for $d_\rho$. 
In our framework, the triangle inequality may fail because composing an $r$-height-interleaving with an $s$-height-interleaving does not necessarily produce an $(r+s)$-height-interleaving: the required comparison maps between the underlying (co)limit constructions need not exist canonically.
To address this, we work with additively relaxed triangle inequalities, allowing an additive defect.
We relate the existence of the required comparison maps to the connected intersection property (CIP) for $(P,\rho)$ (see Definition~\ref{def:CIP}), which provides a canonical way to construct the corresponding natural transformations.
This enables us to compose $r$- and $s$-height-interleavings to obtain an $(r+s+c)$-height-interleaving, where the additive defect $c$ is controlled by the constant $c(\rho)$.
We thus obtain the following relaxed triangle inequality for $d_\rho$.

\begin{theorem}[Theorem~\ref{thm:relaxed TI}]
\label{mainthm:CIP}
If $(P,\rho)$ has the connected intersection property, 
then $d_{\rho}$ satisfies an additively $c(\rho)$-relaxed triangle inequality.  
In particular, if $c(\rho)=0$, then $d_{\rho}$ defines an extended pseudo-distance. 
\end{theorem}

For instance, CIP holds for any height-difference function on diamond-free posets, including zigzag posets with arbitrary orientation and tree posets (Proposition~\ref{prop:tree-CIP}), while it can already fail on the four-element diamond poset for a natural choice of $\rho$.
Combining functional stability (Theorem~\ref{mainthm:functional stability}) with Theorem~\ref{mainthm:CIP} leads to the following consequence: once a relaxed triangle inequality is established via CIP for a given height-difference function, it automatically transfers to other height-difference functions up to an additive defect controlled by distortion.

Finally, we specialize to the case $\mathcal{C}=\Vect_k$, the category of vector spaces over a field $k$, and study the relationship between our construction and classical translation-based interleavings. 
In particular, we show analogues of a few basic properties from the classical setting, including a systematic construction of $r$-height-interleavings from suitable subquotients of a given persistence module via erosion neighborhoods (Theorem~\ref{thm:erosion neighborhood}), 
as well as a comparison between the height-interleaving distance $d_{\rho}$ and the erosion neighborhood distance $d_{\rho\text{-EN}}$ (Theorem~\ref{thm:EN-summary}).

Taken together, these results establish height-interleaving distances as a framework for analyzing poset-indexed functor categories.

\medskip
\noindent\emph{Contributions and perspective.}
Our framework replaces translation structures by a nonnegative input on comparable pairs: a height-difference function $\rho$.
From $\rho$ we systematically construct an $\mathbb{R}_{\geq 0}$-indexed adjoint family of endofunctors on $\Fun(P,\mathcal{C})$ 
and define the associated height-interleaving distance.
A key point is that our approach does not require a rich translation structure on $P$, in contrast to classical translation-based interleavings.
Indeed, height functions provide a standard source of such inputs; in particular, the framework applies to finite posets, which naturally arise as indexing categories in the representation theory of finite dimensional algebras.

From the viewpoint of topological data analysis, our framework provides an interleaving-type quantification for persistence modules indexed by posets.
Under suitable hypotheses on $(P,\rho)$ (e.g.\ CIP), this yields a relaxed triangle inequality and hence a distance-like comparison. In particular, it applies to persistence modules indexed by zigzag and tree posets, and more generally by diamond-free posets.

\medskip
\noindent\emph{Related work.}
Translation-based interleavings are now formulated in abstract categorical settings, for instance via categories with a flow \cite{dSMS18} and, more generally, $2$-categorical settings \cite{MN26}.
In these frameworks, the standard input is a lax monoidal $[0,\infty)$-action on a category $\mathcal{X}$ (equivalently, a lax monoidal morphism $[0,\infty)\to \End(\mathcal{X})$), which encodes shift operations and naturally induces an interleaving distance on $\mathcal{X}$ as an extended pseudo-distance.
In contrast, our framework does not assume such an action on $\mathcal{X}=\Fun(P,\mathcal{C})$; instead, it takes as input a height-difference function $\rho$ on $P$ and constructs interleaving diagrams built from the adjunction $\L_r^\rho\dashv \R_r^\rho$.
We also note that our constructions of $\L_r^\rho$ and $\R_r^\rho$ are naturally \emph{oplax} in the monoidal direction (see \eqref{eq:muLR} in Section~\ref{sec:latching_matching}).

Diamond-free posets appear naturally in TDA and related areas; this includes tree posets, zigzag posets, and the indexing posets arising in circle-valued persistence.
In \cite{BBS24}, the authors study $0$-dimensional persistent homology over rooted tree posets and establish finite-type results.
For zigzag persistence modules, \cite{BL18} proved an algebraic stability theorem by embedding them into the $2$-parameter setting.
From an algebraic viewpoint, \cite{HIY22} defines and computes distances on the bounded derived category of zigzag persistence modules via Auslander--Reiten theory, and proves an algebraic stability theorem. 
For circle-valued persistence, the continuous cyclic setting has been studied in \cite{HR24,RZ23}, together with related isometry results for continuous quivers of type $\widetilde{A}$ \cite{GZ25}.
For persistence modules over finite non-cyclic orientations of $\widetilde{A}$, \cite{BP25} defines interleavings between persistence modules using Auslander--Reiten translation and proves the corresponding isometry theorem.

Motivated by the instability of decompositions in multiparameter persistence, Bjerkevik recently introduced $\epsilon$-refinements and an operation called pruning as tools for studying approximate decompositions \cite{Bjerkevik25}.
As part of the underlying vocabulary, $\epsilon$-erosion neighborhoods of a persistence module and an associated erosion neighborhood distance were also introduced, together with comparisons to the classical interleaving distance \cite[Theorems~3.13 and 4.7]{Bjerkevik25}.
Our results in Section~\ref{sec:PM} are inspired by these ideas and develop analogous constructions for erosion neighborhoods within the height-interleaving framework.

\medskip\noindent\textbf{Organization.}
This paper is organized as follows.
In Section~\ref{sec:prelim}, we recall basic notions on poset-indexed functor categories and classical translation-based interleavings.
In Section~\ref{sec:height_interleaving}, we introduce height-difference functions and construct the adjoint endofunctors $\L_r^{\rho}\dashv \R_r^{\rho}$, which lead to height-interleavings and the associated distance $d_{\rho}$.
We also explain how this framework recovers the classical interleaving distance on $\mathbb{R}^d$ and present illustrative examples.
In Section~\ref{sec:stability}, we prove two stability properties of $d_{\rho}$, namely pullback stability and functional stability (Theorem~\ref{mainthm:functional stability}).
In Section~\ref{sec:distances}, we investigate $d_{\rho}$ from a distance-theoretic viewpoint.
In particular, we introduce the connected intersections property (CIP) and prove that it yields a relaxed triangle inequality for $d_{\rho}$ with an additive defect controlled by a constant $c(\rho)$ (Theorem~\ref{mainthm:CIP}).
Moreover, we discuss how relaxed triangle inequalities behave under restriction to full subposets via Galois insertions.
In Section~\ref{sec:PM}, we specialize to $\mathcal{C}=\Vect_k$ and develop erosion-type constructions, including erosion neighborhoods and the associated distance.
We conclude with a brief discussion and perspectives in Section~\ref{sec:discussion}.
The appendix collects basic material on adjunctions (mates correspondence) and on limits and colimits.

\medskip
\noindent{\bf Notation.}
Throughout this paper, let $\mathcal{C}$ be a complete and cocomplete category.
The reader may keep in mind the case $\mathcal{C} = \Vect_k$, the category of vector spaces over a field $k$.
For a small category $\mathcal{J}$, we write $\Fun(\mathcal{J},\mathcal{C})$ for the category of covariant functors $\mathcal{J}\to\mathcal{C}$ and natural transformations between them.
For a full subcategory $\mathcal{I}\subseteq\mathcal{J}$ and  $M\in \Fun(\mathcal{J},\mathcal{C})$, we denote its restriction by $M|_{\mathcal{I}}$.
We use standard notation for limits and colimits in $\mathcal{C}$, see Appendix~\ref{appendix:limits-colimits} for details.
In particular, we write $\Delta\colon\mathcal{C}\to\Fun(\mathcal{J},\mathcal{C})$ for the constant diagram functor and use the adjunctions
$\colim \dashv \Delta \dashv \lim$.

\section{Preliminaries}
\label{sec:prelim}

\subsection{Poset-indexed functor categories}
\label{sec:prelim posets}
We regard a poset $P$ as a category whose objects are elements of $P$ and there is a unique morphism $a\to b$ for $a,b\in P$ whenever $a\leq b$ in $P$. 
We denote by $a^{\uparrow}$ (resp., $a^{\downarrow}$) the set of all $x\in P$ such that $a \leq x$ (resp., $x \leq a$). 
An order-preserving map $f\colon Q \to P$ is a map satisfying $f(x)\leq f(y)$ for all $x\leq y$ in $Q$. 
It is called order-embedding if moreover $f(x)\leq f(y)$ implies $x\leq y$, in which case $Q$ is isomorphic to the image of $f$ as posets. Notice that any order-preserving map is a functor between posets.

We consider the functor category $\Fun(P,\mathcal{C})$. 
In this setting, a functor $M\colon P \to \mathcal{C}$ consists of the following data:  
\begin{itemize}
    \item It assigns to an element $a\in P$ an object $M(a)$ in $\mathcal{C}$.
    \item It assigns to a relation $a\leq b$ in $P$ a morphism $M(a\leq b) \colon M(a) \to M(b)$ in $\mathcal{C}$ such that $M(a\leq a) = \id_{M(a)}$ and $M(a\leq c) = M(b\leq c)\circ M(a\leq b)$ for all $a\leq b \leq c$ in $P$.  
\end{itemize}
A morphism $f\colon M\to N$ in $\Fun(P,\mathcal{C})$ is a natural transformation, that is, a family of morphisms $(f(a) \colon M(a) \to N(a))_{a\in P}$ in $\mathcal{C}$ satisfying $N(a\leq b) \circ f(a) = f(b)\circ M(a\leq b)$ for all $a\leq b$. 
%In the case of $\mathcal{C} = \Vect_k$ the category of vector spaces over a field $k$, the functor category $\Fun(P,\Vect_k)$ is known as the category of \emph{$P$-persistence modules} and denoted by $\Rep_k(P)$. 

For an order-preserving map $f\colon Q \to P$, 
we denote by $f^* \colon \Fun(P,\mathcal{C}) \to \Fun(Q,\mathcal{C})$ the functor given by precomposition with $f$.
Explicitly, for $M \in \Fun(P,\mathcal{C})$, we set
\begin{equation}
(f^*M)(x) := M(f(x)), \qand
(f^*M)(x \le y) := M(f(x) \le f(y)).
\end{equation}
The functor $f^*$ is called the pullback functor along $f$.

\subsection{Background on interleavings from translations of posets} 
\label{sec:prelim translations}

A \emph{translation} on a poset $P$ is an order-preserving map $\Lambda \colon P\to P$ equipped with a natural transformation $\id_P \Rightarrow \Lambda$, that is, such that $x\leq \Lambda(x)$ holds for all $x\in P$. 
This yields a natural transformation $\eta_{\Lambda} \colon \id \Rightarrow \Lambda^*$.
In particular, for each $M\in \Fun(P,\mathcal{C})$, we have a morphism $\eta_{\Lambda,M} \colon M\to \Lambda^*M$ given by $\eta_{\Lambda,M}(x) := M(x\leq \Lambda(x))$ for all $x \in P$.
Notice that $\Lambda^*(\eta_{\Lambda,M}) = \eta_{\Lambda,\Lambda^*M}$ follows from the definition. 

\begin{definition}
    For two $M,N\in \Fun(P,\mathcal{C})$, we say that $M$ and $N$ are \emph{$\Lambda$-interleaved} if there are morphisms $f \colon M\to \Lambda^*N$ and $g \colon N \to \Lambda^*M$ such that the following diagram commutes. 
    \begin{equation}
        \xymatrix@C=40pt@R=26pt{
        M \ar[r]^{\eta_{\Lambda,M}}
        \ar[dr]^(0.4){f} & \Lambda^*M \ar[r]^{\Lambda^*(\eta_{\Lambda,M})} \ar[dr]^(0.4){\Lambda^*f} & \Lambda^*(\Lambda^*M)\\ 
        N \ar[r]_{\eta_{\Lambda,N}} \ar[ur]^(0.3){g} 
        & \Lambda^*N \ar[r]_{\Lambda^*(\eta_{\Lambda,N})} \ar[ur]^(0.3){\Lambda^*g} & \Lambda^*(\Lambda^*N). 
        }
    \end{equation}  
\end{definition}

A \emph{superlinear family of translations} $\Omega$ on $P$ is a family of translations
$(\Omega_r \colon P \to P)_{r\geq 0}$ equipped with natural transformations
$\Omega_{s}\Omega_{r} \Rightarrow \Omega_{s+r}$ for all $s,r \geq 0$.
The interleaving distance between $M$ and $N$ with respect to $\Omega$ is defined by
\begin{equation}
    d_{\Omega}(M,N) := \inf\{r\geq 0 \mid \text{$M$ and $N$ are $\Omega_r$-interleaved} \},
\end{equation}
where we set $d_{\Omega}(M,N)=\infty$ if there exists no $r\geq 0$ for which they are $\Omega_r$-interleaved.
It is known (see \cite{dSMS18} for example) that $d_{\Omega}$ defines an extended pseudo-distance on $\Fun(P,\mathcal{C})$.
Namely, it satisfies symmetry (i.e., $d_{\Omega}(M,N) = d_{\Omega}(N,M)$), vanishing on the diagonal
(i.e., $d_{\Omega}(M,M) = 0$), and the triangle inequality
(i.e., $d_{\Omega}(M,N) \leq d_{\Omega}(M,X) + d_{\Omega}(X,N)$).

We call $\Omega$ \emph{strong} if 
each $\Omega_r$ is an isomorphism, with inverse $\Omega_{-r}$,
and the structure maps $\Omega_s\Omega_r \Rightarrow \Omega_{s+r}$ are natural isomorphisms.
In this case, we have $\Omega_{-r}(x)\leq x$ for all $x$, that is, a natural transformation
$\Omega_{-r}\Rightarrow \id_P$, which yields a natural transformation
$\eta_{\Omega_{-r}}\colon \Omega_{-r}^*\Rightarrow \id$. 
With this notation, an $\Omega_r$-interleaving between $M$ and $N$ can then be expressed as follows:
\begin{equation}
\xymatrix@C=45pt@R=26pt{
\Omega_{-r}^*M \ar[r]^-{\eta_{\Omega_{-r},M}}
\ar[dr]^(0.4){\Omega_{-r}^*f} &
M \ar[r]^-{\eta_{\Omega_{r},M}} \ar[dr]^(0.4){f} &
\Omega_r^* M \\
\Omega_{-r}^*N \ar[r]_-{\eta_{\Omega_{-r},N}} \ar[ur]^(0.3){\Omega_{-r}^*g} &
N \ar[r]_-{\eta_{\Omega_r,N}} \ar[ur]^(0.3){g} &
\Omega_r^* N.
}
\end{equation}

\bigskip
Now, we briefly recall the classical notion of interleavings on $\mathbb{R}^d$. 
We regard $\mathbb{R}^d$ as a poset equipped with the product order. 
For a given real number $r\in \mathbb{R}$, we denote by $\vec{r}=(r,\ldots,r)\in \mathbb{R}^d$ the diagonal vector whose coordinates are all equal to $r$. 
Then the assignment $a\mapsto a+\vec{r}$ defines a translation on $\mathbb{R}^d$. 
This map induces an endofunctor on $\Fun(\mathbb{R}^d,\mathcal{C})$ as the pullback, called the \emph{$r$-shift functor}. 
Explicitly, for $M\in \Fun(\mathbb{R}^d,\mathcal{C})$, its $r$-shift $M(r)$ is given by 
\begin{equation}
    (M(r))(a) := M(a + \vec{r}) \qand 
    (M(r))(a\leq b) := M(a + \vec{r} \leq b + \vec{r}).
\end{equation}

The family of $r$-shifts with $r\geq 0$ forms a strong superlinear family of translations on $\mathbb{R}^d$. 
Accordingly, we say that two functors $M,N\in \Fun(\mathbb{R}^d,\mathcal{C})$ are $r$-interleaved if they are interleaved with respect to the $r$-shift functor. 
The associated interleaving distance is denoted by 
\begin{equation} \label{eq:Rd interleaving}
    d_I(M,N) := \inf\{r\geq 0 \mid \text{$M$ and $N$ are $r$-interleaved}\}. 
\end{equation}

\section{Interleaving distances induced by height-difference functions}
\label{sec:height_interleaving}

In this section, we introduce interleaving distances induced by \emph{height-difference functions}. 
A key ingredient is an $\mathbb{R}_{\geq 0}$-indexed family of adjoint endofunctors on the functor category $\Fun(P,\mathcal{C})$, defined via latching- and matching-type operations.

\subsection{Latching and matching functors induced by height-difference functions}
\label{sec:latching_matching}

Our motivating setting is a height function on $P$.

\begin{definition}\label{def:height function}
A \emph{height function} on a poset $P$ is an order-preserving map $\phi\colon P\to\mathbb{R}$.
\end{definition}

Given such a $\phi$, we may consider the associated height-difference on comparable pairs,
\begin{equation}\label{eq:rho-phi}
\rho_\phi(a,b):=\phi(b)-\phi(a)\qquad(a\leq b).
\end{equation}

In what follows, we introduce an axiomatic notion of a height-difference function on a poset.

\begin{definition}\label{def:height-difference}
Let $P$ be a poset. A \emph{height-difference function} on $P$ is a function
\[
\rho:\{(a,b)\in P\times P\mid a\leq b\}\to [0,\infty]
\]
such that $\rho(a,a)=0$ for all $a\in P$, and satisfying the following superadditivity condition along chains:
\begin{equation}\label{eq:superadditive}
\rho(a,c)\geq \rho(a,b)+\rho(b,c)
\end{equation}
for every $a\leq b\leq c$ in $P$.
\end{definition}
Any height function $\phi$ defines a height-difference function $\rho_\phi$ as in \eqref{eq:rho-phi}. 
Hence,
\[
\{\rho_\phi\mid \phi \text{ a height function}\}\subseteq \{\rho \mid \rho \text{ a height-difference function}\}.
\]

Let $\rho$ be a height-difference function on $P$.
For $a\in P$ and $r \geq 0$, we set
\begin{equation}\label{eq:updown-rho}
a^{\downarrow_{r}^{\rho}} := \{x\in P \mid x \leq a,\ \rho(x,a)\geq r \}
\qand 
a^{\uparrow_{r}^{\rho}} := \{x\in P \mid x \geq a,\ \rho(a,x)\geq r \}.
\end{equation}
We omit $\rho$ from the notation when it is clear from the context. 

\begin{lemma}\label{lem:updownarrows}
The following hold.
\begin{enumerate}[\rm (1)]
    \item We have $a^{\downarrow_0} = a^{\downarrow}$ and $a^{\uparrow_0} = a^{\uparrow}$.
    \item For any $a \leq b$ in $P$, we have $a^{\downarrow_r} \subseteq b^{\downarrow_r}$ and 
    $b^{\uparrow_r} \subseteq a^{\uparrow_r}$.
    \item For any $s \geq r$, 
    we have $a^{\downarrow_s} \subseteq a^{\downarrow_r}$ and $a^{\uparrow_s} \subseteq a^{\uparrow_r}$.
    \item We have $b\in a^{\downarrow_r}$ if and only if $a \in b^{\uparrow_r}$.
\end{enumerate}
\end{lemma}

\begin{proof}
It is straightforward from the definition.
\end{proof}

For a given functor $M\colon P \to \mathcal{C}$ and $a\in P$, 
let
\begin{equation}
    \L_r^{\rho} M(a) := \colim M|_{a^{\downarrow_r}}.  
\end{equation}
We write $\alpha_{a,M} \colon M|_{a^{\downarrow_r}} \to \Delta(\L_r^{\rho}M(a))$ for the unit of $\colim \dashv \Delta$ (see Notation). 
In addition, for $a\leq b$ in $P$, the universal property of the colimit yields a unique morphism  
\begin{equation}
    \L_r^{\rho} M(a\leq b) := \big(\colim M|_{a^{\downarrow_r}}  \to \colim M|_{b^{\downarrow_r}} \big)
\end{equation} 
by $a^{\downarrow_r} \subseteq b^{\downarrow_r}$ as in Lemma~\ref{lem:updownarrows}(2). 
Then, they assemble into a functor $\L_r^\rho M\colon P\to\mathcal C$. 

On the other hand, for a morphism $f\colon M \to N$ in $\Fun(P,\mathcal{C})$, 
we have a morphism $\L_r^{\rho} f\colon \L_r^{\rho} M \to \L_r^{\rho} N$, where the component $(\L_r^{\rho}f)(a)$ for $a\in P$ is defined as a unique map 
$\colim M|_{a^{\downarrow_r}} \to \colim N|_{a^{\downarrow_r}}$ given by the universal property of colimits. 
Indeed, it follows from the fact that the family $\big(M(x) \xrightarrow{f(x)} N(x) \xrightarrow{\alpha_{a,N}(x)} \L_r^{\rho}N(a)\big)_{x\in a^{\downarrow_r}}$ forms a cocone over the diagram $M|_{a^{\downarrow_r}}$.  
\begin{equation}
    \xymatrix@C=44pt@R=26pt{
    \L_r^{\rho} M (a) \ar@{.>}[r]^-{(\L_r^{\rho}f)(a)} \ar@{}[rd]|{\circlearrowleft} & \L_r^{\rho}N(a) \\ 
    M(x) \ar[r]^-{f(x)} \ar[u]^{\alpha_{a,M}(x)} & N(x) \ar[u]^{\alpha_{a,N}(x)}. 
    }
\end{equation}
This construction is natural in $a$, so the maps $(\L_r^\rho f)(a)$ give rise to a morphism $\L_r^\rho f\colon \L_r^\rho M\rightarrow \L_r^\rho N$. 

Finally, since the assignment $M\mapsto \L_r^\rho M$ is functorial, it defines an endofunctor $\L_r^\rho\colon \Fun(P,\mathcal C)\to \Fun(P,\mathcal{C})$.

Dually, we obtain an endofunctor $\R_r^{\rho}\colon \Fun(P,\mathcal{C}) \to \Fun(P,\mathcal{C})$ by using limits. 
More explicitly, for $M\in \Fun(P,\mathcal{C})$ and $a\in P$, we set 
\begin{equation}
    \R_r^{\rho} M(a) := \lim M|_{a^{\uparrow_r}}  
\end{equation}
and write $\beta_{a,M} \colon \Delta(\R_r^{\rho}M(a)) \to M|_{a^{\uparrow_r}}$ for the counit of $\Delta\dashv \lim$. 
For $a\leq b$ in $P$, the morphism 
\begin{equation}
    \R_r^{\rho}M(a\leq b) := \big( \lim M|_{a^{\uparrow_r}} \to \lim M|_{b^{\uparrow_r}} \big)
\end{equation}  
is the unique morphism induced by the universal property of limits, using the inclusion $b^{\uparrow_r} \subseteq a^{\uparrow_r}$.
For a morphism $g\colon M \to N$ in $\Fun(P,\mathcal{C})$, 
we define a morphism $\R_r^{\rho} g\colon \R_r^{\rho} M \to \R_r^{\rho} N$ as follows: 
for each $a\in P$, the family of morphisms 
$\big(\R_r^{\rho}M(a) \xrightarrow{\beta_{a,M}(y)} M(y) \xrightarrow{g(y)} N(y)\big)_{y\in a^{\uparrow_r}}$ forms a cone over the diagram $N|_{a^{\uparrow_r}}$, and hence induces a unique morphism $(\R_r^{\rho}g)(a)\colon \R_r^{\rho} M(a) \to \R_r^{\rho}N(a)$. 
\begin{equation}
    \xymatrix@C=44pt@R=26pt{
    M (y) \ar@{->}[r]^-{g(y)} \ar@{}[rd]|{\circlearrowleft} & N(y) \\ 
    \R_r^{\rho} M(a) \ar@{.>}[r]^-{(\R_r^{\rho}g)(a)} \ar[u]^{\beta_{a,M}(y)} & \R_r^{\rho} N(a) \ar[u]^{\beta_{a,N}(y)}.
    }
\end{equation}

\begin{definition}\label{def:latching and matching}
For $r\geq 0$, we call $\L_r^{\rho}$ (resp., $\R_r^{\rho}$) the \emph{$r$-latching} (resp., \emph{$r$-matching}) \emph{functor} on $\Fun(P,\mathcal{C})$ with respect to $\rho$. 
We simply write $\L_r$ and $\R_r$ instead of $\L_r^{\rho}$ and $\R_r^{\rho}$ if $\rho$ is clear from the context. 
\end{definition}    

The following result is central of this paper.

\begin{proposition}\label{prop:adjoint LrRr}
For each $r \geq 0$, we have an adjunction $\L_r^{\rho}\dashv \R_{r}^{\rho}$.
\begin{equation}
\xymatrix@C=68pt@R=38pt{
\Fun(P,\mathcal{C}) 
\ar@{<-}@/^3.5mm/[r]^{\L_r^{\rho}} 
\ar@{}[r]|{\perp}
& 
\ar@{<-}@/^3.5mm/[l]^{\R_r^{\rho}} 
\Fun(P,\mathcal{C}).  
}
\end{equation}
\end{proposition}

\begin{proof}
Fix $r\geq 0$. 
To show that $\L_r^{\rho}$ is left adjoint to $\R_r^{\rho}$, we construct mutually inverse bijections, natural in $M$ and $N$,
\begin{equation}\label{eq:LRadjoint}
(-)^{\flat} \colon \Hom(\L_r^{\rho}M,N) \leftrightarrows \Hom(M,\R_r^{\rho}N) \colon (-)^{\sharp}. 
\end{equation}

For $x,a\in P$, recall that $x\in a^{\downarrow_r}$ holds if and only if $a\in x^{\uparrow_r}$ by Lemma~\ref{lem:updownarrows}. 
In this case, we have obtained the canonical morphisms 
$\alpha_{a,M}(x) \colon M(x) \to \L_r^{\rho}M(a)$ and $\beta_{x,M}(a) \colon \R_r^{\rho}M(x)\to M(a)$.

Let $g\colon M \to \R_r^{\rho}N$ be a morphism in $\Fun(P,\mathcal{C})$. 
Then, the family of morphisms $\big(M(x) \xrightarrow{g(x)} \R_r^{\rho} N(x) \xrightarrow{\beta_{x,N}(a)}N(a)\big)_{x\in a^{\downarrow_r}}$ forms a cocone over the diagram $M|_{a^{\downarrow_r}}$, and hence 
the universal property of the colimit yields a unique morphism $g^{\sharp}(a) \colon \L_r^{\rho}M(a) \to N(a)$. 
\begin{equation}
\xymatrix@C=38pt@R=26pt{
\L_r^{\rho} M (a) \ar@{.>}[r]^-{g^{\sharp}(a)} \ar@{}[rd]|{\circlearrowleft} & N(a) \\ 
M(x) \ar@{->}[r]^-{g(x)} \ar[u]^{\alpha_{a,M}(x)} & \R_r^{\rho} N(x) \ar[u]^{\beta_{x,N}(a)}.
}
\end{equation}
This construction is natural in $a$, so the maps $g^{\sharp}(a)$ give rise to a unique morphism $g^{\sharp} \colon \L_r^{\rho}M \to N$.

Conversely, we define the map $(-)^{\flat}$ by a dual argument using limits. 
Explicitly, for a given morphism $f\colon \L_r^{\rho} M \to N$, we define a morphism $f^{\flat}\colon M\to \R_r^{\rho}N$ as follows: 
For each $a\in P$, the family of morphisms $\big(M(a) \xrightarrow{\alpha_{y,M}(a)} \L_r^{\rho}M(y) \xrightarrow{f(y)} N(y)\big)_{y\in a^{\uparrow_r}}$ forms a cone over the diagram $N|_{a^{\uparrow_r}}$. 
Then, the universal property of the limit yields a unique morphism $f^{\flat}(a)\colon M(a)\to \R_r^{\rho}N(a)$. 
These morphisms define a morphism $f^{\flat} \colon M\to \R_r^{\rho}N$. 
\begin{equation}
\xymatrix@C=38pt@R=26pt{
\L_r^{\rho} M (y) \ar@{->}[r]^-{f(y)} \ar@{}[rd]|{\circlearrowleft} & N(y) \\ 
M(a) \ar@{.>}[r]^-{f^{\flat}(a)} \ar[u]^{\alpha_{y,M}(a)} & \R_r^{\rho} N(a) \ar[u]^{\beta_{a,N}(y)}.
}
\end{equation}

By the uniqueness of the morphisms induced by the universal properties of limits and colimits, 
it follows that $(-)^{\flat}$ and $(-)^{\sharp}$ are mutually inverse. 
Moreover, one can check that these correspondences are natural in $M$ and $N$. 
This completes the proof. 
\end{proof}

In the rest of this paper, we will freely use the notations $f^{\flat}$ and $g^{\sharp}$ for the bijections induced by the adjunction $\L_r^{\rho} \dashv \R_r^{\rho}$ as in \eqref{eq:LRadjoint}.

For a height-difference function $\rho$ on $P$, the family of endofunctors $\{\L_r,\R_r\}_{r\geq 0}$ satisfies the following compatibility properties. 
\begin{itemize}
\item There are natural isomorphisms $\L_0 \simeq \id \simeq \R_0$ of endofunctors on $\Fun(P,\mathcal{C})$. 
\item For $s\geq r$, there are natural transformations $\eta_{s,r}^{\L}\colon \L_s \Rightarrow \L_r$ and $\eta_{r,s}^{\R}\colon \R_r \Rightarrow \R_s$ which are compatible in the sense that 
\[
\eta_{r,t}^{\L} = (\L_t \overset{\eta_{t,s}^{\L}}{\Rightarrow} 
\L_s \overset{\eta_{s,r}^{\L}}{\Rightarrow} \L_r)
\qand 
\eta_{r,t}^{\R} = (\R_r \overset{\eta_{r,s}^{\R}}{\Rightarrow} 
\R_s \overset{\eta_{s,t}^{\R}}{\Rightarrow} \R_t)
\]
for all $t\geq s\geq r$. 
\item For any $s,r\geq 0$, there are natural transformations
\begin{equation}\label{eq:muLR}
\mu^{\L}_{s,r}\colon \L_s\L_r \Rightarrow \L_{s+r}
\qand
\mu^{\R}_{r,s}\colon \R_{s+r} \Rightarrow \R_r\R_s.
\end{equation}
\end{itemize}

In fact, the natural transformation $\eta_{s,r}^{\L}\colon \L_{s} \Rightarrow \L_{r}$ above for $s\geq r$ is induced by the inclusion $a^{\downarrow_s}\subseteq a^{\downarrow_{r}}$. 
On the other hand, let $s,r\geq 0$ and $a\in P$.
For each $x\in a^{\downarrow_s}$, we have $x^{\downarrow_r}\subseteq a^{\downarrow_{s+r}}$ by \eqref{eq:superadditive}.
Thus the universal property of colimits yields a canonical morphism
$\colim M|_{x^{\downarrow_r}} \to \colim M|_{a^{\downarrow_{s+r}}}$.
These maps are natural in $x$, hence form a cocone over the diagram
$\bigl(\colim M|_{x^{\downarrow_r}}\to \colim M|_{a^{\downarrow_{s+r}}}\bigr)_{x\in a^{\downarrow_s}}$.
Therefore, by the universal property of the colimit, they induce a canonical morphism
$(\L_s(\L_rM))(a)\to (\L_{s+r}M)(a)$.
Naturality in $a$ and $M$ yields a natural transformation
$\mu^{\L}_{s,r}\colon \L_s\L_r\Rightarrow \L_{s+r}$. 
The construction of $\eta^{\R}_{r,s}$ and $\mu^{\R}_{r,s}$ is obtained dually, using limits instead of colimits.

Moreover, for $s\geq r$, by $\L_r\dashv \R_r$ we have the adjoint correspondence
\[
\mathrm{Nat}(\L_s,\L_r)\cong \mathrm{Nat}(\R_r,\R_s)
\]
(see Appendix~\ref{appendix:two-adjunction-mates} for details),
under which $\eta^{\R}_{r,s}$ agrees with the mate of $\eta^{\L}_{s,r}$.

On the other hand, for $s,r\geq 0$, using the composite adjunction $\L_s\L_r\dashv \R_r\R_s$ we have
\[
\mathrm{Nat}(\L_s\L_r,\L_{s+r})\cong \mathrm{Nat}(\R_{s+r},\R_r\R_s),
\]
under which $\mu^{\R}_{r,s}$ agrees with the mate of $\mu^{\L}_{s,r}$.

We will often write $\eta_r^{\L}$ and $\eta_r^{\R}$ for $\eta_{r,0}^{\L}$ and $\eta_{0,r}^{\R}$, respectively.
In addition, we define a natural transformation
\[
e_r := \eta_r^{\R}\circ \eta_r^{\L}\colon \L_r \Rightarrow \R_r
\]
as the vertical composite. 
Namely, its component at $M$ is
\begin{equation}\label{eq:erosion}
e_{r,M}:\ \L_rM \xrightarrow{\eta_{r,M}^{\L}} M \xrightarrow{\eta_{r,M}^{\R}} \R_rM .
\end{equation}

\begin{remark}\label{rem:Reedy-type} 
Our latching/matching functors are inspired by Reedy-type constructions, but they differ from the classical latching/matching functors. 
Here, we recall from \cite{Hovey99} that the \emph{latching space} of a functor $M\colon P \to \mathcal{C}$ at $a\in P$ is defined as the object $\colim M|_{a^{\downarrow} \setminus \{a\}}$ in $\mathcal{C}$. 
Similar to Definition~\ref{def:latching and matching}, one can define an endofunctor ${\rm L}$ on $\Fun(P,\mathcal{C})$ by assigning to each $M$ the functor ${\rm L}M\colon P \to \mathcal{C}$ that sends $a\in P$ to the latching space of $M$ at $a$. 
We have a natural transformation ${\rm L} \Rightarrow \id$ of endofunctors on $\Fun(P,\mathcal{C})$. 
By construction, this yields a factorization $\L_r^{\rho} \Rightarrow {\rm L} \Rightarrow \id$ for any height-difference function $\rho$ on $P$ and $r>0$.

Dually, there is an endofunctor ${\rm R}$ on $\Fun(P,\mathcal{C})$,
defined using limits over $a^{\uparrow}\setminus\{a\}$,
and one similarly obtains a factorization 
$\id \Rightarrow {\rm R} \Rightarrow \R_r^{\rho}$.
\end{remark}

For latter use, we record a pointwise description of $\L_r$ and $\R_r$ in terms of Kan extensions (We refer to \cite{MacLane} for the basics of Kan extensions). 
Fix $r\geq 0$ and $a \in P$.
We write $i_a\colon a^{\downarrow_r}\hookrightarrow P$ and $j_a\colon a^{\uparrow_r}\hookrightarrow P$ for the canonical inclusions. 
Let $\Lan_{i_a}$ (resp., $\Ran_{j_a}$) be the left (resp., right) Kan extension along $i_a$ (resp., $j_a$). 
Then, for any $M\in \Fun(P,\mathcal{C})$, we have canonical identifications 
\begin{equation}\label{eq:LR as Kanextension}
    \L_rM(a) \cong (\Lan_{i_a}(M|_{a^{\downarrow_r}}))(a) \qand 
    \R_rM(a) \cong (\Ran_{j_a}(M|_{a^{\uparrow_r}}))(a). 
\end{equation}
In other words, the values $\L_rM(a)$ and $\R_rM(a)$ are given as Kan extension values at $a$. 
Note, however, that the indexing inclusions $i_a,j_a$ depend on $a$, so this gives only a pointwise description; it does not identify the endofunctors $\L_r,\R_r\colon \Fun(P,\mathcal{C})\to \Fun(P,\mathcal{C})$
with a single Kan extension along a fixed inclusion independent of $a$. 

Instead of \eqref{eq:LR as Kanextension}, we may work with the full subposet 
\begin{equation}\label{eq:Ua}
    U_a := a^{\downarrow_r}\cup a^{\uparrow_r}\subseteq P
\end{equation}
and the canonical inclusion $u_a\colon U_a\hookrightarrow P$. 
Since $U_a \cap a^{\downarrow} = a^{\downarrow_r}$ and $U_a \cap a^{\uparrow} = a^{\uparrow_r}$, 
we also obtain pointwise descriptions 
\begin{equation}\label{eq:LR as Ua}
    \L_rM(a) \cong (\Lan_{u_a}(M|_{U_a}))(a) 
    \qand 
    \R_rM(a) \cong (\Ran_{u_a}(M|_{U_a}))(a).  
\end{equation}

\subsection{Height-interleaving distances}
In this section, we introduce a notion of interleavings using the adjoint endofunctors $\L_r \dashv\R_r$ from the previous section. 

Let $\rho$ be a height-difference function on a poset $P$. 

\begin{definition}\label{def:height-interleaving}
Let $M,N \in \Fun(P,\mathcal{C})$. 
For $r\geq 0$, an \emph{$r$-height-interleaving} between $M$ and $N$ with respect to $\rho$ is a pair of morphisms $p\colon M \to \R_r N$ and $q\colon N\to \R_r M$ satisfying 
\[
    e_{r,M} = q\circ p^{\sharp} 
    \qand 
    e_{r,N} = p\circ q^{\sharp},
\]
where $p^{\sharp}\colon \L_rM\to N$ and $q^{\sharp}\colon \L_rN\to M$ denote the morphisms corresponding to $p$ and $q$ under the adjunction $\L_r \dashv \R_r$, respectively. 
Equivalently, the following diagram commutes, where we note that the commutativity of the two adjunction squares is automatic by $\L_r\dashv \R_r$. 
\begin{equation}
    \xymatrix@C=36pt@R=26pt{
    \L_r M \ar[r]^{\eta_{r,M}^{\L}} 
    \ar[dr]^(0.4){p^{\sharp}} & M \ar[r]^{\eta_{r,M}^{\R}} \ar[dr]^(0.4){p} & \R_rM \\ 
    \L_r N \ar[r]_{\eta_{r,N}^{\L}} \ar[ur]^(0.3){q^{\sharp}} 
    & N \ar[r]_{\eta_{r,N}^{\R}} \ar[ur]^(0.3){q} & \R_rN. 
    }
\end{equation} 

In this case, we say that $M$ and $N$ are \emph{$r$-height-interleaved} (with respect to $\rho$) and write $M\overset{\rho}{\underset{r}{\sim}} N$.

We define the \emph{$\rho$-interleaving distance} between $M$ and $N$ by 
\begin{equation}
    d_{\rho}(M,N) := \inf\{r\geq 0\mid \text{$M$ and $N$ are $r$-height-interleaved with respect to $\rho$}\},
\end{equation}
where we set $d_{\rho}(M,N)= \infty$ if no such $r$ exists.
\end{definition}

\begin{convention}\label{conv:phi-notation}
Let $\phi\colon P\to\mathbb{R}$ be a height function, and let $\rho_\phi$ be the associated height-difference function as in \eqref{eq:rho-phi}.
Whenever $\rho=\rho_\phi$, we write
\[
\L_r^{\phi}:=\L_r^{\rho_\phi},
\qquad
\R_r^{\phi}:=\R_r^{\rho_\phi},
\qquad
e_r^{\phi}:=e_r^{\rho_\phi},
\qquad
d_{\phi}:=d_{\rho_\phi}.
\]
In this case, we also say \emph{$r$-height-interleaving with respect to $\phi$} to mean $r$-height-interleaving with respect to $\rho_\phi$.
\end{convention}

While $d_{\rho}$ enjoys symmetry and vanishing on the diagonal, it fails to satisfy the triangle inequality in general, even when $\rho=\rho_\phi$ is induced by a height function $\phi$ (see Example~(2) in Subsection~\ref{sec:examples}). 
In Section~\ref{sec:distances}, we will therefore discuss a relaxed version of the triangle inequality, allowing an additive defect.

\medskip
As expected for interleaving-type distances, our notion of height-interleaving has a monotonicity property as follows. 
\begin{lemma}\label{lem:monotonicity-height-interleaving}
If $M$ and $N$ are $r$-height-interleaved with respect to $\rho$, 
then they are $s$-height-interleaved for all $s\geq r$. 
\end{lemma}

\begin{proof}
It is immediate from the natural transformations $\eta_{s,r}^{\L}\colon \L_{s}\Rightarrow \L_r$ and $\eta_{r,s}^{\R}\colon \R_r\Rightarrow \R_s$, which allow us to compose an $r$-interleaving into an $s$-interleaving.
\end{proof}

\begin{lemma}\label{lem:id-LR_interleaving}
Let $M\in\Fun(P,\mathcal{C})$ and $r\geq 0$. 
Then, $M$ is $r$-height-interleaved with each of $\L_rM$ and $\R_rM$, with respect to $\rho$.
\end{lemma}

\begin{proof}
We show the claim for $\R_rM$. The statement for $\L_rM$ follows by duality. 
Define 
\[
p := (e_{r,M})^{\flat} \colon M\to \R_r(\R_rM)
\qand 
q := \id_{\R_rM}\colon \R_rM\to \R_rM.
\]
Then $(p,q)$ forms an $r$-height-interleaving between $M$ and $\R_rM$, using the naturality of $\eta_r^{\R}$ and the adjunction $\L_r\dashv \R_r$.
\end{proof}

\subsection{Recovering the classical interleaving distance on $\mathbb{R}^d$}

Recall that the classical interleaving distance $d_I$ on $\Fun(\mathbb{R}^d,\mathcal{C})$ is defined via the $r$-shift functors (see Section~\ref{sec:prelim translations}).
We now recover $d_I$ from our height-interleaving distance by specializing to the poset $P=\mathbb{R}^d$ and the following height-difference function:
\begin{equation} \label{eq:rho-diag}
    \rho_{\mathrm{diag}}(a,b):=\min_{1\leq i\leq d}(b_i-a_i)\qquad(a\leq b).
\end{equation}

\begin{proposition}\label{prop:recoverRd}
Let $\rho:=\rho_{\mathrm{diag}}$ be defined above. 
Then the following statements hold.
\begin{enumerate}[\rm (1)]
    \item There are natural isomorphisms $\L_r^{\rho} \simeq (-r)$ and $\R_r^{\rho} \simeq (r)$ for all $r\geq 0$.
    \item Under the identifications in \textnormal{(1)}, the notion of $r$-interleaving induced by the $r$-shift functors coincides with the notion of $r$-height-interleaving with respect to $\rho$ on $\Fun(\mathbb{R}^d,\mathcal{C})$.
    \item The distance $d_{\rho}$ coincides with the interleaving distance $d_{I}$ on $\Fun(\mathbb{R}^d,\mathcal{C})$.
\end{enumerate}
\end{proposition}

\begin{proof}
(1) Let $r\geq 0$ and $a\in \mathbb{R}^d$. 
For $\rho=\rho_{\mathrm{diag}}$, the subsets $a^{\downarrow_r}$ and $a^{\uparrow_r}$ are given by
\begin{equation}\label{eq:vsRshiftRd}
a^{\downarrow_r} = (a-\vec{r})^{\downarrow}
\qand
a^{\uparrow_r} = (a+\vec{r})^{\uparrow},
\end{equation}
where $\vec{r}=(r,\ldots,r)\in\mathbb{R}^d$.
Indeed, for $x\leq a$ we have $\rho(x,a)=\min_i(a_i-x_i)$, so $\rho(x,a)\geq r$ if and only if $x\leq a-\vec{r}$. 
The second identity is similar.

Let $M\in \Fun(\mathbb{R}^d,\mathcal{C})$. 
Since $(a-\vec{r})$ is the unique maximal element of $(a-\vec{r})^{\downarrow}$, we obtain
\[
\L_r^{\rho}M(a)
=\colim M|_{(a-\vec{r})^{\downarrow}}
\cong M(a-\vec{r})
=(M(-r))(a).
\]
Similarly, since $(a+\vec{r})$ is the unique minimal element of $(a+\vec{r})^{\uparrow}$, we have
\[
\R_r^{\rho}M(a)
=\lim M|_{(a+\vec{r})^{\uparrow}}
\cong M(a+\vec{r})
=(M(r))(a).
\]
These identifications are natural in $M$, hence yield natural isomorphisms $\L_r^{\rho}\simeq (-r)$ and $\R_r^{\rho}\simeq (r)$.

The assertions \textnormal{(2)} and \textnormal{(3)} follow from \textnormal{(1)}.
\end{proof}

When $d=1$, $\rho_{\mathrm{diag}}$ coincides with the height-difference induced by the height function $\id_{\mathbb{R}}$.
For $d\geq 2$, however, $\rho_{\mathrm{diag}}$ is not induced by any height function $\phi\colon \mathbb{R}^d\to\mathbb{R}$.

\subsection{Examples} \label{sec:examples}
In this subsection, we illustrate the functors $\L_r$ and $\R_r$, together with height-interleaving distances, 
through small examples in $\Fun(P,\mathcal{C})$ with 
$\mathcal{C}=\Vect_k$, 
where the height-difference function is induced by an explicitly chosen height function. 
In particular, Example~(2) exhibits functors $M,X,N$ and a choice of $\phi$ for which the associated distance $d_{\phi}$ fails to satisfy the triangle inequality. 

The following notation will be used later.
A subset $J\subseteq P$ is \emph{convex} if $x\leq y\leq z$ in $P$ and $x,z\in J$ imply $y\in J$.
For such a subset $J$, we write $k_J\in \Fun(P,\Vect_k)$ for the functor given by $k_J(x)=k$ for $x\in J$ and $k_J(x)=0$ otherwise, with all structure maps induced by the identity.

\bigskip
(1) Let $P := \{0,1,2,3\}\times\{0,1,2\}$ be the $4\times 3$ grid poset with the product order, i.e.,
\[
(i,j)\leq (i',j') \quad\Longleftrightarrow\quad i\leq i'\ \text{and}\ j\leq j'.
\]
We write $v_{ij} = (i,j)$. 
Define a height function $\phi$ on this poset by $\phi(v_{ij}) := i+j$.

Consider a functor $M\in \Fun(P,\Vect_k)$ given by the middle diagram in \eqref{eq:example-LMR}, which is indecomposable.
Then, a direct computation shows that $\L_1M$ and $\R_1M$ are given by the left and right diagrams in \eqref{eq:example-LMR}, respectively. 
Note that, by our choice of $\phi(v_{ij})=i+j$, the index sets $(v_{ij})^{\downarrow_1}$ and $(v_{ij})^{\uparrow_1}$ are exactly the strict principal downset and upset of $v_{ij}$, respectively.

\begin{equation}\label{eq:example-LMR} 
\begin{tabular}{cccc}
$\L_1M$ & $M$ & $\R_1M$  \\
$\xymatrix@R=1.8em@C=2em{
k \ar[r]^{\text{\tiny$\left[\begin{smallmatrix}1 \\ 0\end{smallmatrix}\right]$}} & k^2 \ar[r] & 0 \ar[r] & 0 \\
0 \ar[r] \ar[u] & k \ar[r]^{\text{\tiny$\left[\begin{smallmatrix}1 \\ 0\end{smallmatrix}\right]$}} \ar[u]^{\text{\tiny$\left[\begin{smallmatrix}1 \\ 0\end{smallmatrix}\right]$}} & k^2 \ar[r]^{[1\ 0]} \ar[u] & k \ar[u] \\
0 \ar[r] \ar[u] & 0 \ar[r] \ar[u] & 0 \ar[r] \ar[u] & 0 \ar[u]
}$
& 
$\xymatrix@R=1.8em@C=2em{
k \ar[r]^{1} & k \ar[r] & 0 \ar[r] & 0 \\
k \ar[r]^{\text{\tiny$\left[\begin{smallmatrix}1 \\ 0\end{smallmatrix}\right]$}} \ar[u]^{1} & k^2 \ar[r]^{\text{\tiny$[1\ 0]$}} \ar[u]^{\text{\tiny$[1\ 1]$}} & k \ar[r]^{1} \ar[u] & k \ar[u] \\
0 \ar[r] \ar[u] & 0 \ar[r] \ar[u] & 0 \ar[r] \ar[u] & 0 \ar[u]
}$
&
$
\xymatrix@R=1.8em@C=2em{
k \ar[r]  &
0 \ar[r]  &
0 \ar[r]  &
0  \\
k^2 \ar[r]^{\text{\tiny$\left[\begin{smallmatrix}1&0\\ 1&1\end{smallmatrix}\right]$}} 
\ar[u]^{\text{\tiny$\left[\begin{smallmatrix}1&1\end{smallmatrix}\right]$}} &
k^2 \ar[r]^{\text{\tiny$\left[\begin{smallmatrix}1&0\end{smallmatrix}\right]$}} \ar[u] &
k \ar[r] \ar[u] &
0 \ar[u] \\
0 \ar[r] \ar[u] &
k \ar[r] \ar[u]^{\text{\tiny$\left[\begin{smallmatrix}0\\1\end{smallmatrix}\right]$}} &
0 \ar[r] \ar[u] &
k \ar[u]
}
$
\end{tabular}
\end{equation}

We now illustrate a computation for $(\L_1M)(v_{22})$ in detail. 
By definition,
\[
(\L_1M)(v_{22})
=\colim \bigl(M|_{(v_{22})^{\downarrow_1}}\bigr).
\]

Set $J:=\{v_{11},v_{21},v_{12}\}\subset (v_{22})^{\downarrow_1}$.
Since the poset $J\cap x^{\uparrow}$ has a minimum for every $x\in (v_{22})^{\downarrow_1}$, 
the above colimit is naturally isomorphic to $\colim M|_J$.
Thus
\[
(\L_1M)(v_{22}) 
\cong 
\colim
\left(
\vcenter{\hbox{$\xymatrix@R=1.3em@C=2.0em{
k \ar@{}[r]|{\phantom{k}} & \phantom{k} \\
k^2 \ar[u]^{\text{\tiny$[1\ 1]$}} \ar[r]^{\text{\tiny$[1\ 0]$}} & k
}$}}
\right)
\cong
\colim
\left(
\vcenter{\hbox{$\xymatrix@R=1.3em@C=2.0em{
k \ar@{}[r]|{\phantom{k}} & \phantom{k} \\
k \ar[u]^{1}\ar[r] & 0
}$}}
\right)
\oplus 
\colim
\left(
\vcenter{\hbox{$\xymatrix@R=1.3em@C=2.0em{
0 \ar@{}[r]|{\phantom{k}} & \phantom{k} \\
k \ar[u] \ar[r]^{1} & k
}$}}
\right)
\cong 0. 
\]
Similar computations can apply to the other vertices of $P$.

Moreover, the above $\L_1M$ and $\R_1M$ decompose as
\[
\L_1M \cong k_{J_1}
\oplus k_{\{(1,2)\}}
\oplus k_{\{(2,1)\}},
\qand 
\R_1M \cong 
k_{J_2}
\oplus 
k_{J_3}
\oplus
k_{\{(3,0)\}} 
\]
where $J_1,J_2,J_3$ are the following convex subsets: 
\[
J_1= \{(0,2),(1,2),(1,1),(2,1),(3,1)\},\  
J_2 = \{(0,1),(1,1),(2,1)\},\ 
J_3 = \{(1,0),(0,1),(1,1),(0,2)\}. 
\]

\bigskip
(2) 
Let $P=\{a<b<c<d\}$ be the four-point chain.
Fix a constant $C>1$ and define a height function $\phi\colon P\to\mathbb{R}$ by
\[
\phi(a)=0,\qquad \phi(b)=1,\qquad \phi(c)=C+1,\qquad \phi(d)=2C+1.
\]
Consider the functors $M,X,N\in \Fun(P,\Vect_k)$:
\[
M=(k\xrightarrow{1}k\xrightarrow{1}k\xrightarrow{1}k),\qquad
X=(k\xrightarrow{1}k\xrightarrow{1}k\rightarrow 0),\qquad
N =(k\xrightarrow{1}k\rightarrow 0\rightarrow 0).
\]

We claim that
\begin{equation}\label{eq:defect C}
d_{\phi}(M,X)=0,\qquad d_{\phi}(X,N)=0,\qquad d_{\phi}(M,N)=C.
\end{equation}
In particular, the triangle inequality for $d_\phi$ fails.
%, and any inequality of the form
%\[
%d_\phi(M,N) = d_\phi(M,X)+d_\phi(X,N)+c
%\]
%forces $c\ge C$.

\smallskip
Fix $0<\epsilon<1$.
First note that we have
\[
a^{\uparrow_\epsilon}=\{b,c,d\},\qquad
b^{\uparrow_\epsilon}=\{c,d\},\qquad
c^{\uparrow_\epsilon}=\{d\}, \qquad 
d^{\uparrow_\epsilon}=\emptyset. 
\]
A direct computation of limits over these upsets gives
\[
\R_{\epsilon}M =(k\rightarrow k\rightarrow k\rightarrow 0),
\quad
\R_{\epsilon}X =(k\rightarrow k\rightarrow0\rightarrow0),\quad
\R_{\epsilon}N =(k\rightarrow 0\rightarrow0\rightarrow0). 
\]
Similarly, we have that 
\[
a^{\downarrow_\epsilon}=\emptyset,\qquad
b^{\downarrow_\epsilon}=\{a\},\qquad
c^{\downarrow_\epsilon}=\{a,b\}, \qquad 
d^{\downarrow_\epsilon}=\{a,b,c\} 
\]
and that  
\[
\L_{\epsilon}M =(0 \rightarrow k \rightarrow k \rightarrow k),
\qquad
\L_{\epsilon}X =(0 \rightarrow k \rightarrow k \rightarrow k),\qquad
\L_{\epsilon}N =(0 \rightarrow k \rightarrow k \rightarrow 0). 
\]

\smallskip
We define morphisms $f\colon M\rightarrow \R_{\epsilon}X$ and $g\colon X\rightarrow \R_{\epsilon}M$ by
\[
f(a)=f(b)=1,\quad f(c)=f(d)=0,
\qquad
g(a)=g(b)=g(c)=1,\quad g(d)=0 .
\]
Let $f^{\sharp}\colon \L_{\epsilon}M\rightarrow X$ and $g^{\sharp}\colon \L_{\epsilon}X\rightarrow M$
be the left adjoints under $\L_{\epsilon}\dashv \R_{\epsilon}$.
Then $(f,g)$ gives an $\epsilon$-height-interleaving between $M$ and $X$ by the commutativity of the following diagrams:
\begin{equation}\label{eq:MX-eps-interleaving}
\vcenter{\hbox{$
\xymatrix@C=18pt@R=16pt{
\L_{\epsilon}X \ar[d]_{g^{\sharp}}  &
\llap{(\,}0 \ar[r] \ar[d]^0 &
k \ar[r] \ar[d]^{1} &
k \ar[r] \ar[d]^{1} &
k \rlap{\,)} \ar[d]^{1} \\
M \ar[d]_{f} &
\llap{(\,} k \ar[r] \ar[d]^{1} &
k \ar[r] \ar[d]^{1} &
k \ar[r] \ar[d]^{0} &
k \rlap{\,)} \ar[d]^{0}  \\
\R_{\epsilon}X &
\llap{(\,} k \ar[r] &
k \ar[r] &
0 \ar[r] &
0 \rlap{\,)} 
}
$}}
\qand  
\vcenter{\hbox{$
\xymatrix@C=18pt@R=16pt{
\L_{\epsilon}M \ar[d]_{f^{\sharp}} &
\llap{(\,} 0 \ar[r] \ar[d]^{0} &
k \ar[r] \ar[d]^{1} &
k \ar[r] \ar[d]^{1} &
k \rlap{\,)} \ar[d]^{1} \\
X \ar[d]_{g}  &
\llap{(\,} k \ar[r] \ar[d]^{1} &
k \ar[r] \ar[d]^{1} &
k \ar[r] \ar[d]^{1} &
0 \rlap{\,)} \ar[d]^{0} \\
\R_{\epsilon}M  &
\llap{(\,} k \ar[r] &
k \ar[r] &
k \ar[r] &
0 \rlap{\,)}
}
$}}
\end{equation}
%Hence $(f,g)$ gives an $\epsilon$-height-interleaving between $M$ and $X$.
Similarly, one obtains an $\epsilon$-height-interleaving between $X$ and $N$; we omit the details.
Since $0<\epsilon<1$ was arbitrary, we conclude $d_\phi(M,X)=d_{\phi}(X,N)=0$.

\smallskip
Next, we compute $d_\phi(M,N)$.
For $s:=C+\epsilon$ we have
\[
a^{\uparrow_s}=\{c,d\},\qquad b^{\uparrow_s} = \{d\}, \qquad c^{\uparrow_s}= d^{\uparrow_s}=\emptyset
\]
and 
\[
a^{\downarrow_s}= b^{\downarrow_s} = \emptyset, \qquad c^{\downarrow_s}=\{a\},\qquad d^{\downarrow_s}=\{a,b\}
\]
In this situation, there is an $s$-height-interleaving pair $(u,v)$ between $M$ and $N$ as follows. 
\begin{equation}\label{eq:MN-r-interleaving}
\vcenter{\hbox{$
\xymatrix@C=18pt@R=16pt{
\L_{s}N \ar[d]_{v^{\sharp}}  &
\llap{(\,}0 \ar[r] \ar[d]^0 &
0 \ar[r] \ar[d]^{0} &
k \ar[r] \ar[d]^{1} &
k \rlap{\,)} \ar[d]^{1} \\
M \ar[d]_{u} &
\llap{(\,} k \ar[r] \ar[d]^{1} &
k \ar[r] \ar[d]^{1} &
k \ar[r] \ar[d]^{0} &
k \rlap{\,)} \ar[d]^{0}  \\
\R_{s}N &
\llap{(\,} 0 \ar[r] &
0 \ar[r] &
0 \ar[r] &
0 \rlap{\,)} 
}
$}}
\qand  
\vcenter{\hbox{$
\xymatrix@C=18pt@R=16pt{
\L_{s}M \ar[d]_{u^{\sharp}} &
\llap{(\,} 0 \ar[r] \ar[d]^{0} &
0 \ar[r] \ar[d]^{0} &
k \ar[r] \ar[d]^{1} &
k \rlap{\,)} \ar[d]^{1} \\
N \ar[d]_{v} &
\llap{(\,} k \ar[r] \ar[d]^{1} &
k \ar[r] \ar[d]^{1} &
0 \ar[r] \ar[d]^{0} &
0 \rlap{\,)} \ar[d]^{0} \\
\R_{s}M  &
\llap{(\,} k \ar[r] &
k \ar[r] &
0 \ar[r] &
0 \rlap{\,)}
}
$}}
\end{equation}
Hence, $d_\phi(M,N)\leq s=C+\epsilon$. 
Since $0<\epsilon<1$ was arbitrary, this yields $d_\phi(M,N)\le C$. 

On the other hand, we show that $M$ and $N$ are \emph{not} $C$-height-interleaved.
Since $c^{\uparrow_C}=\{d\}$ and $c^{\downarrow_C}= \{a,b\}$, we have
\[
(\L_C M)(c)\cong k,\qquad M(c)\cong k,\qquad (\R_C M)(c)\cong k,
\]
and the canonical morphism $e_{C,M}(c)\colon (\L_C M)(c)\to (\R_C M)(c)$ is non-zero.
However, $N(c)=0$ forces any composite
$(\L_C M)(c)\to N(c)\to (\R_C M)(c)$
to be zero. 
Therefore, $M$ and $N$ are not $C$-height-interleaved.

By Lemma~\ref{lem:monotonicity-height-interleaving}, this implies that they are not $r$-height-interleaved for any $r\leq C$.
Hence, by definition of $d_\phi$ as an infimum, we have $d_\phi(M,N)\geq C$.

Combining these yields $d_\phi(M,N)=C$, as claimed.

\section{Stability properties}
\label{sec:stability}
In this section, we discuss stability properties of height-interleavings and the associated distances $d_{\rho}$.
These results provide further motivation for using height-interleaving distances to analyze functor categories.

\subsection{Pullback stability}
We first show that height-interleavings are functorial with respect to pullbacks along order-preserving maps; 
in particular, the associated distance is stable under pullback.

Let $\rho$ be a height-difference function on a poset $P$.
For an order-preserving map $f\colon Q\to P$, we write $f^*\rho$ for the induced height-difference function on $Q$, given by
\[
(f^*\rho)(q,q'):=\rho(f(q),f(q'))\qquad(q\leq q'). 
\]

\begin{proposition}\label{prop:pullback stability}
Let $\rho$ be a height-difference function on a poset $P$, and let $f\colon Q\to P$ be an order-preserving map.
Then for any $M,N \in \Fun(P, \mathcal{C})$, we have
\[
d_{f^*\rho}(f^*M,f^*N) \leq d_{\rho}(M,N).
\]
\end{proposition}

To prove this, we need the following observation. 

\begin{lemma}\label{lem:pullbackLR}
Let $\rho$ be a height-difference function on a poset $P$, and let $f\colon Q\to P$ be an order-preserving map.
Then the following statements hold for $r\geq 0$.
\begin{enumerate}[\rm (1)]
\item For any $q\in Q$, we have
\[
f\bigl(q^{\downarrow_r^{f^*\rho}}\bigr) \subseteq (f(q))^{\downarrow_r^{\rho}}
\qand
f\bigl(q^{\uparrow_r^{f^*\rho}}\bigr) \subseteq (f(q))^{\uparrow_r^{\rho}}.
\]
\item There are natural transformations  
\[
\xi_r^{\L}\colon \L_r^{f^*\rho}f^* \Rightarrow f^*\L_r^{\rho}
\qand
\xi_r^{\R} \colon f^*\R_r^{\rho} \Rightarrow \R_r^{f^*\rho}f^* .
\]
\end{enumerate}
\end{lemma}

\begin{proof}
(1) By definition,
\[
f\bigl(q^{\downarrow_r^{f^*\rho}}\bigr)
=\{f(y)\mid y\leq q,\ \rho(f(y),f(q))\geq r\},
\]
and
\[
(f(q))^{\downarrow_r^{\rho}}
=\{x\in P\mid x\leq f(q),\ \rho(x,f(q))\geq r\}.
\]
Let $x\in f\bigl(q^{\downarrow_r^{f^*\rho}}\bigr)$. Then $x=f(y)$ for some $y\leq q$ with $\rho(f(y),f(q))\geq r$.
On the one hand, since $f$ is order-preserving we have $x=f(y)\leq f(q)$.
In addition, we have $\rho(x,f(q))=\rho(f(y),f(q))\geq r$.
Thus $x\in (f(q))^{\downarrow_r^{\rho}}$, proving $f(q^{\downarrow_r^{f^*\rho}})\subseteq (f(q))^{\downarrow_r^{\rho}}$.
The upset inclusion is similar.

(2) Let $M\in \Fun(P,\mathcal{C})$ and $q\in Q$. We compute
\[
(\L_r^{f^*\rho}(f^*M))(q)
=
\underset{y\in q^{\downarrow_r^{f^*\rho}}}{\colim} M(f(y))
\qand
(f^*(\L_r^{\rho}M))(q)
=(\L_r^{\rho}M)(f(q))
=\underset{x\in (f(q))^{\downarrow_r^{\rho}}}{\colim} M(x).
\]
By \textnormal{(1)}, the assignment $y\mapsto f(y)$ defines an order-preserving map
$q^{\downarrow_r^{f^*\rho}}\to (f(q))^{\downarrow_r^{\rho}}$.
Thus, by the universal property of colimits it induces a canonical morphism
\[
(\xi_{r,M}^{\L})(q)\colon (\L_r^{f^*\rho}(f^*M))(q)\to (f^*(\L_r^{\rho}M))(q).
\]
These maps are natural in $q$ and $M$, yielding a natural transformation
$\xi_r^{\L}\colon \L_r^{f^*\rho}f^* \Rightarrow f^*\L_r^{\rho}$.

Dually, using limits and the inclusion in \textnormal{(1)}, 
we obtain a natural transformation
$\xi_r^{\R}\colon f^*\R_r^{\rho} \Rightarrow \R_r^{f^*\rho}f^*$.
\end{proof}

\begin{proof}[Proof of Proposition~\ref{prop:pullback stability}]
Suppose that $M$ and $N$ are $r$-height-interleaved with respect to $\rho$, with interleaving morphisms
$p\colon M\to \R_r^{\rho} N$ and $q\colon N\to \R_r^{\rho} M$. 

Let $\xi_r^{\L}\colon \L_r^{f^*\rho}f^* \Rightarrow f^*\L_r^{\rho}$ and 
$\xi_r^{\R}\colon f^*\R_r^{\rho} \Rightarrow \R_r^{f^*\rho}f^*$ be as in Lemma~\ref{lem:pullbackLR}(2). 
Define
\[
u := \xi^{\R}_{r,N} \circ f^*(p)\colon f^*M\to \R_r^{f^*\rho}(f^*N),
\qquad
v := \xi^{\R}_{r,M}\circ f^*(q)\colon f^*N\to \R_r^{f^*\rho}(f^*M).
\]    

Then, by naturality of $\xi_r^{\L},\xi_r^{\R}$ and the definition of $u,v$, we obtain
\[
u^{\sharp}= f^*(p^{\sharp})\circ \xi^{\L}_{r,M},
\qquad
v^{\sharp}= f^*(q^{\sharp})\circ \xi^{\L}_{r,N}.
\]

This implies that 
\[
e_{r,f^*M}^{f^*\rho} = 
v\circ u^{\sharp}
\quad\text{and}\quad
e_{r,f^*N}^{f^*\rho} = 
u\circ v^{\sharp}.
\]
Hence $(u,v)$ forms an $r$-height-interleaving between $f^*M$ and $f^*N$ with respect to $f^*\rho$.
\end{proof}

\subsection{Functional stability}
\label{sec:functional_stability}
Next, we prove that height-interleaving distances are stable under perturbations of height-difference functions (Theorem~\ref{thm:functional stability}).
To state this result, we first record a basic monotonicity property of the assignment $\rho\mapsto d_{\rho}$.

\begin{proposition}\label{prop:monotonicity-rho}
Let $\rho_1,\rho_2$ be height-difference functions on $P$ such that $\rho_1\leq \rho_2$.
Then for any $M,N\in \Fun(P,\mathcal{C})$,
\[
d_{\rho_1}(M,N)\leq d_{\rho_2}(M,N).
\]
\end{proposition}

\begin{proof}
The assumption $\rho_1\leq \rho_2$ implies
\[
a^{\downarrow_r^{\rho_1}}\subseteq a^{\downarrow_r^{\rho_2}}
\qand
a^{\uparrow_r^{\rho_1}}\subseteq a^{\uparrow_r^{\rho_2}}
\]
for all $a\in P$ and $r\geq 0$.
Hence restriction along these inclusions induces natural transformations
\[
\L_r^{\rho_1}\Rightarrow \L_r^{\rho_2}
\qand 
\R_r^{\rho_2}\Rightarrow \R_r^{\rho_1}. 
\]
If $M$ and $N$ are $r$-height-interleaved with respect to $\rho_2$, composing the interleaving morphisms with $\R_r^{\rho_2}\Rightarrow \R_r^{\rho_1}$ gives an $r$-height-interleaving between $M$ and $N$ with respect to $\rho_1$.
Therefore $d_{\rho_1}(M,N)\leq d_{\rho_2}(M,N)$.
\end{proof}

For $a,b\in[0,\infty]$, we set
\[
|a-b|_{\infty}:=
\begin{cases}
|a-b| & a,b<\infty,\\
0 & a=b=\infty,\\
\infty & \text{otherwise}.
\end{cases}
\]
Note that for $r\geq 0$, the inequality $|a-b|_{\infty}\leq r$ is equivalent to say that
\[
a\leq b+r \qand b\leq a+r. 
\] 

For two height-difference functions $\rho_1,\rho_2$ on $P$, 
we define the \emph{distortion} between $\rho_1$ and $\rho_2$ by
\begin{equation}\label{eq:delta_rhos}
\delta(\rho_1,\rho_2)
:=\inf\Bigl\{r\geq 0 \Bigm| 
|\rho_1(a,b)-\rho_2(a,b)|_{\infty}\leq r\ \text{for all }a\leq b\Bigr\},
\end{equation}
and the distortion between the associated height-interleaving distances $d_{\rho_1}$ and $d_{\rho_2}$ by
\[
\mathrm{dist}_{\infty}(d_{\rho_1},d_{\rho_2})
:=
\inf\Bigl\{r\geq 0 \Bigm| 
|d_{\rho_1}(M,N)-d_{\rho_2}(M,N)|_{\infty}\leq r\ \text{for all }M,N \in \Fun(P,\mathcal{C}) \Bigr\}.
\]

Our result is the following.

\begin{theorem}\label{thm:functional stability}
The assignment $\rho\mapsto d_{\rho}$ is $1$-Lipschitz with respect to $\delta$.
That is, for any height-difference functions $\rho_1,\rho_2$ on $P$,
\begin{equation}\label{eq:functional stability}
\mathrm{dist}_{\infty}(d_{\rho_1},d_{\rho_2}) \leq \delta(\rho_1,\rho_2).
\end{equation}
\end{theorem}

\begin{proof}   
If $\delta(\rho_1,\rho_2)=\infty$, then the claim is clear. 
Assume $\delta(\rho_1,\rho_2)<\infty$. In this case, we may write \eqref{eq:delta_rhos} as 
\[
\delta(\rho_1,\rho_2)=\sup_{a\leq b}\bigl|\rho_1(a,b)-\rho_2(a,b)\bigr|_\infty.
\]
Set $r:=\delta(\rho_1,\rho_2)$. Then, for all $a\leq b$,
\[
\rho_2(a,b)\leq \rho_1(a,b)+r
\qand
\rho_1(a,b)\leq \rho_2(a,b)+r
\]

Fix $M,N\in\Fun(P,\mathcal C)$.
Suppose that $M$ and $N$ are $s$-height-interleaved with respect to $\rho_1$ for some $s\geq 0$.
Since $\rho_2(a,b)\le \rho_1(a,b)+r$ for all $a\le b$, 
the same construction as in Proposition~\ref{prop:monotonicity-rho} yields an $(s+r)$-height-interleaving between $M$ and $N$ with respect to $\rho_2$.
Hence $d_{\rho_2}(M,N)\leq s+r$.
Taking the infimum over such $s$ shows that
$d_{\rho_2}(M,N)\leq d_{\rho_1}(M,N)+r$.
By symmetry, we also have $d_{\rho_1}(M,N)\leq d_{\rho_2}(M,N)+r$.
Therefore,
\[
|d_{\rho_1}(M,N)-d_{\rho_2}(M,N)|_{\infty}\leq r. 
\]
Since this holds for all $M,N\in\Fun(P,\mathcal C)$, we obtain 
\[
\mathrm{dist}_{\infty}(d_{\rho_1},d_{\rho_2})\leq r = \delta(\rho_1,\rho_2),
\]
as desired. 
\end{proof}

\section{Distances with relaxed triangle inequalities}
\label{sec:distances}
Recall that the height-interleaving distance $d_{\rho}$ is symmetric and vanishes on the diagonal, but it may fail to satisfy the triangle inequality. 
This motivates us to introduce the following terminology.
We say that $d_{\rho}$ satisfies a  (additively) \emph{$c$-relaxed triangle inequality} if
\begin{equation}\label{eq:relaxed TI}
    d_{\rho}(M,N) \leq d_{\rho}(M,X) + d_{\rho}(X,N) + c \quad \text{for all $M,N,X \in \Fun(P,\mathcal{C})$}. 
\end{equation}
We refer to such a constant $c$ as a \emph{defect} (of the triangle inequality). 

When $c=0$, the $c$-relaxed triangle inequality reduces to the usual triangle inequality, 
yielding that $d_{\rho}$ with $c=0$ is a genuine extended pseudo-distance. 
\bigskip

The $c$-relaxed triangle inequality typically arises when composing an $s$-height-interleaving and an $r$-height-interleaving produces an $(s+r+c)$-height-interleaving. 
To formalize this, assume that for all $s,r\geq 0$ there are natural transformations
\begin{equation}\label{eq:sigmaL-family}
    \sigma^{\L}_{s,r}\colon \L_{s+r+c}\Longrightarrow \L_s\L_r
    \qquad (s,r\ge 0)
\end{equation}
that provide factorizations 
\begin{equation} \label{eq:sigmaL-factor}
\eta^{\L}_{s+r+c,s+r}\;=\;\left(
\L_{s+r+c} \overset{{\sigma^{\L}_{s,r}}}{\Longrightarrow} \L_s\L_r
\overset{{\mu^{\L}_{s,r}}}{\Longrightarrow} \L_{s+r}
\right).
\end{equation}
By the mates correspondence for two adjunctions 
(Appendix~\ref{appendix:two-adjunction-mates})
\begin{equation}\label{eq:mate-bijection}
\mathrm{Nat}(\L_{s+r+c},\L_s\L_r)
\;\cong\;
\mathrm{Nat}(\R_r\R_s,\R_{s+r+c}),
\end{equation} 
each natural transformation
$\sigma^{\L}_{s,r}\colon \L_{s+r+c}\Rightarrow \L_s\L_r$
uniquely determines a corresponding natural transformation
$\sigma^{\R}_{r,s}\colon \R_r\R_s\Rightarrow \R_{s+r+c}$.
Applying \eqref{eq:mate-bijection} with $s$ and $r$ exchanged, we also write $\sigma^{\R}_{s,r}\colon \R_s\R_r\Rightarrow \R_{s+r+c}$ for the mate of $\sigma^{\L}_{r,s}$.
With this choice, it provides a factorization 
\begin{equation}\label{eq:sigmaR-factor}
\eta^{\R}_{r+s,r+s+c} = 
\Big(
\R_{s+r}
\overset{\mu^{\R}_{r,s}}{\Longrightarrow}
\R_r\R_s
\overset{\sigma^{\R}_{r,s}}{\Longrightarrow}
\R_{s+r+c}
\Big).
\end{equation}

Now, we suppose that $M$ and $X$ are $s$-height-interleaved via $p\colon M\to \R_s X$ and $q\colon X\to \R_s M$, 
and that $X$ and $N$ are $r$-height-interleaved via $u\colon X\to \R_r N$ and $v\colon N\to \R_r X$.
Then, we define morphisms
\begin{eqnarray}    
f  &:=&
M \xrightarrow{p} \R_s X \xrightarrow{\R_s(u)} \R_s\R_r N
\xrightarrow{\;\sigma_{s,r,N}^{\R}\;} \R_{s+r+c}N, 
\qquad \\ 
g & :=&
N \xrightarrow{v} \R_r X \xrightarrow{\R_r(q)} \R_r\R_s M
\xrightarrow{\;\sigma_{r,s,M}^{\R}\;} \R_{s+r+c}M,
\end{eqnarray}
and we let $f^{\sharp}\colon \L_{s+r+c}M\to N$ and $g^{\sharp}\colon \L_{s+r+c}N\to M$ be their left adjoints.
Using the factorizations \eqref{eq:sigmaL-factor} and \eqref{eq:sigmaR-factor}, 
one checks that
\[
e_{s+r+c,M} = g\circ f^{\sharp},
\qquad
e_{s+r+c,N} = f\circ g^{\sharp},
\]
so the pair $(f,g)$ forms an $(s+r+c)$-height-interleaving between $M$ and $N$. 
Hence \eqref{eq:relaxed TI} holds for all $M,N,X\in\Fun(P,\mathcal C)$, i.e. $d_{\rho}$ satisfies a $c$-relaxed triangle inequality.

\subsection{Relaxed triangle inequalities under CIP}

In this subsection, we show that the connected intersections property (CIP) for $(P,\rho)$ (Definition~\ref{def:CIP}) provides a canonical way to construct a family of natural transformations \eqref{eq:sigmaL-family}, where the resulting defect $c$ can be bounded by a constant determined by $\rho$ (Theorem~\ref{thm:relaxed TI}). 

Let $P$ be a poset with a height-difference function $\rho$.
Fix $a\in P$ and $s,r\geq 0$. We set
\begin{equation}\label{eq:iterated-updown}
    a^{\downarrow_s\downarrow_r}
    := \bigcup_{x\in a^{\downarrow_s}} x^{\downarrow_r}
    \subseteq a^{\downarrow_{s+r}}
    \qand
    a^{\uparrow_r\uparrow_s}
    := \bigcup_{x\in a^{\uparrow_r}} x^{\uparrow_s}
    \subseteq a^{\uparrow_{s+r}} .
\end{equation}
The inclusions in \eqref{eq:iterated-updown} follow from \eqref{eq:superadditive}.

These subsets enjoy properties analogous to Lemma~\ref{lem:updownarrows}:
\begin{itemize}
    \item For any $a\leq b$ in $P$, we have
    $a^{\downarrow_s\downarrow_r}\subseteq b^{\downarrow_s\downarrow_r}$
    and
    $b^{\uparrow_r\uparrow_s}\subseteq a^{\uparrow_r\uparrow_s}$.
    \item For any $s'\geq s$ and $r'\geq r$, we have
    $a^{\downarrow_{s'}\downarrow_{r'}} \subseteq a^{\downarrow_s\downarrow_r}$
    and
    $a^{\uparrow_{r'}\uparrow_{s'}} \subseteq a^{\uparrow_{r}\uparrow_{s}}$.
    \item We have $b\in a^{\downarrow_s\downarrow_r}$ if and only if
    $a\in b^{\uparrow_r\uparrow_s}$.
\end{itemize}

Next, for $M\in \Fun(P,\mathcal C)$ and $a\in P$, we define
\begin{equation}\label{eq:Trs-def}
    \T_{s,r}^{\L}M(a)
    := \colim M|_{a^{\downarrow_s\downarrow_r}}
    \qand
    \T_{r,s}^{\R}M(a)
    := \lim M|_{a^{\uparrow_r\uparrow_s}} .
\end{equation}
By the first two items above, the assignment \eqref{eq:Trs-def}
defines endofunctors
\[
    \T_{s,r}^{\L},\ \T_{r,s}^{\R}\colon \Fun(P,\mathcal C)\to \Fun(P,\mathcal C).
\]
We have an adjunction
$\T_{s,r}^{\L}\dashv \T_{r,s}^{\R}$, as shown in the same way as $\L_r\dashv \R_r$, using the equivalence 
$b\in a^{\downarrow_s\downarrow_r}$ if and only if $a\in b^{\uparrow_r\uparrow_s}$.

Moreover, the natural transformations $\mu_{s,r}^{\L}$ and $\mu_{r,s}^{\R}$ factor through these functors as 
\begin{equation}\label{eq:mu-factor-through-T}
    \mu_{s,r}^{\L} =
    \Big(
        \L_s\L_r
        \overset{\kappa_{s,r}^{\L}}{\Longrightarrow}
        \T_{s,r}^{\L}
        \overset{\tau_{s,r}^{\L}}{\Longrightarrow}
        \L_{s+r}
    \Big)
    \qand
    \mu_{r,s}^{\R} =
    \Big(
        \R_{r+s}
        \overset{\tau_{r,s}^{\R}}{\Longrightarrow}
        \T_{r,s}^{\R}
        \overset{\kappa_{r,s}^{\R}}{\Longrightarrow}
        \R_r\R_s
    \Big).
\end{equation}
In fact, the natural transformation 
$\tau_{s,r}^{\L}\colon \T_{s,r}^{\L} \Rightarrow \L_{s+r}$
is induced by the inclusion $a^{\downarrow_s\downarrow_r}\subseteq a^{\downarrow_{s+r}}$.
The natural transformation $\kappa^{\L}_{s,r}\colon \L_s\L_r\Rightarrow \T^{\L}_{s,r}$ is obtained by the same argument as for $\mu^{\L}_{s,r}$, using the inclusions 
$x^{\downarrow_r}\subseteq a^{\downarrow_s\downarrow_r}$ for $x\in a^{\downarrow_s}$. 
The constructions of $\tau_{r,s}^{\R}$ and $\kappa^{\R}_{r,s}$ are dual.

\subsubsection{A Fubini-type finality criterion for $\kappa$}
Fix $a\in P$ and $s,r\ge 0$. Set
\[
I:=a^{\downarrow_s},\qquad F(x):=x^{\downarrow_r}\subseteq P\quad (x\in I),
\qquad Q:=\bigcup_{x\in I}F(x)=a^{\downarrow_s\downarrow_r}.
\]
For a functor $D\in \Fun(Q,\mathcal{C})$, 
consider the canonical morphism
\begin{equation}\label{eq:Fubini-map}
\colim_{x\in I}\,\colim D|_{F(x)}\ \longrightarrow\ \colim D,
\end{equation}
natural in $D$ (see Appendix~\ref{appendix:iterated-colimits}). 
Applied to $D=M|_{Q}$, this is precisely the component
\begin{equation}\label{eq:Fubini-map_kappa}
\kappa^{\L}_{s,r,M}(a)\colon (\L_s\L_r M)(a)\longrightarrow (\T^{\L}_{s,r}M)(a).
\end{equation}

\begin{proposition}\label{prop:cofinality-kappaL}
In the above setting, 
if for every $q\in Q$ the subposet
\[
I_{s,r}(a,q):= a^{\downarrow_s}\cap q^{\uparrow_r}
\]
is connected, then the natural morphism \eqref{eq:Fubini-map} is an isomorphism for any $D\in \Fun(Q,\mathcal C)$.
\end{proposition}
\begin{proof}
Let $D\in \Fun(Q,\mathcal{C})$ be a functor.
With $I,F,Q$ as above, the morphism \eqref{eq:Fubini-map} agrees with the canonical map
\eqref{eq:iterated-colim-to-Q} in Appendix~\ref{appendix:iterated-colimits}.
Moreover, for $q\in Q$ the corresponding index subposet is
\[
I_q=\{x\in I\mid q\in F(x)\}
=\{x\in a^{\downarrow_s}\mid q\in x^{\downarrow_r}\}
=a^{\downarrow_s}\cap q^{\uparrow_r}
=I_{s,r}(a,q).
\]
Hence, by Proposition~\ref{prop:finality-int-to-Q} (the finality criterion in Appendix~\ref{appendix:iterated-colimits}),
if $I_{s,r}(a,q)$ is connected for every $q\in Q$, then
\eqref{eq:Fubini-map} is an isomorphism.
\end{proof}

Now, we introduce the following terminology.

\begin{definition}\label{def:CIP}
    We say that $(P,\rho)$ has the \emph{connected intersections property} (CIP) if for any $a,q\in P$ and $s,r\geq0$, the set $I_{s,r}(a,q)$ is empty or connected. 
    Note that $I_{s,r}(a,q)$ is non-empty if and only if $q\in a^{\downarrow_s\downarrow_r}$, if and only if $a\in q^{\uparrow_r\uparrow_s}$.
\end{definition}

\begin{corollary}\label{cor:CIP-kappa-iso}
If $(P,\rho)$ satisfies CIP, then for all $s,r\ge 0$ the natural transformations
$\kappa^{\L}_{s,r}\colon \L_s\L_r\Rightarrow \T^{\L}_{s,r}$ and 
$\kappa^{\R}_{r,s}\colon \T^{\R}_{r,s}\Rightarrow \R_r\R_s$ in \eqref{eq:mu-factor-through-T} are isomorphisms. 
\end{corollary}

\begin{proof}
Fix $s,r\ge 0$ and $M\in \Fun(P,\mathcal C)$.
For each $a\in P$, by CIP we may apply Proposition~\ref{prop:cofinality-kappaL}
to the canonical morphism \eqref{eq:Fubini-map} for the functor $D=M|_{Q}$,
where $Q=a^{\downarrow_s\downarrow_r}$.
Hence the component $\kappa^{\L}_{s,r,M}(a)$ in \eqref{eq:Fubini-map_kappa}
is an isomorphism. 
Since \eqref{eq:Fubini-map} is natural in $D$, these isomorphisms are natural in $M$
(and in $a$), so $\kappa^{\L}_{s,r}$ is a natural isomorphism.
The statement for $\kappa^{\R}_{r,s}$ follows dually.
\end{proof}

\subsubsection{The defect constant and comparison morphisms}

Thanks to Corollary~\ref{cor:CIP-kappa-iso}, under CIP we may replace
$\L_s\L_r$ and $\R_r\R_s$ by $\T^{\L}_{s,r}$ and $\T^{\R}_{r,s}$ via the isomorphisms $\kappa^{\L}_{s,r}$ and $\kappa^{\R}_{r,s}$, respectively.
Hence, in order to obtain \eqref{eq:sigmaL-family} 
it suffices to compare $\L_{s+r+c}$ with $\T^{\L}_{s,r}$, and dually $\T^{\R}_{r,s}$ with $\R_{s+r+c}$. 

For this, we consider the following approximate intermediate-value property of $\rho$.

\begin{description}
\item[(IV$_c$)]
For any $a\leq b$ in $P$ and any $t\in[0,\infty)$ with 
$t\leq \rho(a,b)$, there exists $z\in P$ with $a\leq z\leq b$ such that
\[
\bigl|\rho(a,z)-t\bigr|_{\infty}\leq \frac{c}{2}
\qand
\bigl|\rho(z,b)-\bigl(\rho(a,b)-t\bigr)\bigr|_{\infty}\leq \frac{c}{2}.
\]
\end{description}

Note that (IV$_c$) is monotone in $c$: if (IV$_c$) holds, then so does (IV$_{c'}$) for every $c'\geq c$.
We define the constant associated with $\rho$ by 
\begin{equation}\label{eq:c-rho}
c(\rho):=\inf\{\,c\ge 0 \mid (P,\rho)\text{ satisfies (IV$_c$)}\,\}.
\end{equation}

The next lemma shows that (IV$_c$) indeed provides the required comparison morphisms. 

\begin{lemma}\label{lem:theta_plusc}
Assume that $(P,\rho)$ satisfies {\rm (IV$_{c}$)} for $c\geq 0$. 
Then, for any $s,r\geq 0$ there are natural transformations 
\[
\theta_{s,r}^{\L} \colon \L_{s+r+c}\Rightarrow \T_{s,r}^{\L}
\qand
\theta_{r,s}^{\R} \colon \T_{r,s}^{\R} \Rightarrow \R_{s+r+c}
\]
satisfying  
\[
\eta_{s+r+c,s+r}^{\L} = \big( 
\L_{s+r+c}
\overset{\theta_{s,r}^{\L}}{\Longrightarrow} 
\T_{s,r}^{\L} 
\overset{\tau_{s,r}^{\L}}{\Longrightarrow} 
\L_{s+r} 
\big) 
\qand 
\eta_{s+r,s+r+c}^{\R} = \big( 
\R_{s+r} 
\overset{\tau_{r,s}^{\R}}{\Longrightarrow} 
\T_{r,s}^{\R} 
\overset{\theta_{r,s}^{\R}}{\Longrightarrow} 
\R_{s+r+c} 
\big).  
\]
\end{lemma}

\begin{proof}
We construct $\theta^{\L}_{s,r}$ and prove the claimed identity for $\eta^{\L}_{s+r+c,s+r}$. 
The statement for $\theta^{\R}_{r,s}$ and $\eta^{\R}_{s+r,s+r+c}$ follows dually.

Fix $s,r\geq 0$, $M\in\Fun(P,\mathcal C)$, and $a\in P$.
We first show the inclusion
\begin{equation}\label{eq:IVc-down-inclusion}
a^{\downarrow_{s+r+c}}\subseteq a^{\downarrow_s\downarrow_r}.
\end{equation}
Take $y\in a^{\downarrow_{s+r+c}}$. By definition, we have $y\leq a$ and $\rho(y,a)\geq s+r+c$.
If $\rho(y,a)<\infty$, set $t:=\rho(y,a)-s-c/2\in[0,\rho(y,a)]$. 
Applying {\rm (IV$_c$)} to $y\leq a$ and this $t$, there exists $z$ with $y\leq z\leq a$ such that
\[
\rho(y,z)\ge t-\frac{c}{2}
\qand
\rho(z,a)\ge (\rho(y,a)-t)-\frac{c}{2}.
\]
Since $\rho(y,a)-t=s+c/2$, we have $\rho(z,a)\ge s$, hence $z\in a^{\downarrow_s}$.
Moreover,
\[
\rho(y,z)\ge t-\frac{c}{2}=\rho(y,a)-s-c\ge r,
\]
so $y\in z^{\downarrow_r}$.
Thus $y\in a^{\downarrow_s\downarrow_r}$.
Otherwise, apply {\rm (IV$_c$)} to $y\le a$ with $t:=r+c$. 
Then there exists $y\leq z\leq a$ such that $\rho(y,z)\geq t-c/2\geq r$.
Since $\rho(y,a)=\infty$, the second inequality in {\rm (IV$_c$)} forces $\rho(z,a)=\infty$, and in particular $\rho(z,a)\geq s$.
Hence $y\in z^{\downarrow_r}$ and $z\in a^{\downarrow_s}$, so $y\in a^{\downarrow_s\downarrow_r}$.
This proves the inclusion \eqref{eq:IVc-down-inclusion}.

By restriction along the inclusion \eqref{eq:IVc-down-inclusion}, we obtain a canonical morphism
\[
(\L_{s+r+c}M)(a)=\colim M|_{a^{\downarrow_{s+r+c}}}\ \longrightarrow\
\colim M|_{a^{\downarrow_s\downarrow_r}}=(\T^{\L}_{s,r}M)(a).
\]
These maps are natural in $a$ and $M$, hence define a natural transformation
$\theta^{\L}_{s,r}\colon \L_{s+r+c} \Rightarrow \T^{\L}_{s,r}$.

Finally, for each $a$ the composite
\[
(\L_{s+r+c}M)(a)\xrightarrow{\theta^{\L}_{s,r,M}(a)}
(\T^{\L}_{s,r}M)(a)\xrightarrow{\tau^{\L}_{s,r,M}(a)}
(\L_{s+r}M)(a)
\]
is precisely the morphism induced by the chain of inclusions
$a^{\downarrow_{s+r+c}}\subseteq a^{\downarrow_s\downarrow_r}\subseteq a^{\downarrow_{s+r}}$,
and therefore coincides with $\eta^{\L}_{s+r+c,s+r,M}(a)$.

The construction of $\theta^{\R}_{r,s}$ and the identity for $\eta^{\R}_{s+r,s+r+c}$ are obtained dually, using the inclusion
$a^{\uparrow_{s+r+c}}\subseteq a^{\uparrow_r\uparrow_s}$ induced by {\rm (IV$_c$)}.
\end{proof}

Combining this lemma with Corollary~\ref{cor:CIP-kappa-iso}, we obtain the following result.

\begin{theorem}\label{thm:relaxed TI}
If $(P,\rho)$ has CIP, then $d_{\rho}$ satisfies a $c(\rho)$-relaxed triangle inequality. 
In particular, if $c(\rho)=0$, then $d_{\rho}$ is an extended pseudo-distance. 
\end{theorem}

\begin{proof}
Fix $c\geq 0$ and assume that $(P,\rho)$ satisfies {\rm (IV$_c$)}.

Let $s,r\geq 0$.
By Corollary~\ref{cor:CIP-kappa-iso}, the natural transformation
$\kappa^{\L}_{s,r}\colon \L_s\L_r\Rightarrow \T^{\L}_{s,r}$ is an isomorphism.
By Lemma~\ref{lem:theta_plusc}, we also have a natural transformation
$\theta^{\L}_{s,r}\colon \L_{s+r+c}\Rightarrow \T^{\L}_{s,r}$
satisfying $\eta^{\L}_{s+r+c,s+r}=\tau^{\L}_{s,r}\circ\theta^{\L}_{s,r}$.

Define a natural transformation $\sigma^{\L}_{s,r}$ by 
\begin{equation}\label{eq:sigmaL_under_CIP}
\sigma^{\L}_{s,r}:=(\kappa^{\L}_{s,r})^{-1}\circ \theta^{\L}_{s,r}
\colon \L_{s+r+c}\Longrightarrow \L_s\L_r.
\end{equation}
Then, by construction, we have the factorization
\[
\eta^{\L}_{s+r+c,s+r}
=
\Big(
\L_{s+r+c}
\xRightarrow{\sigma^{\L}_{s,r}}
\L_s\L_r
\xRightarrow{\mu^{\L}_{s,r}}
\L_{s+r}
\Big),
\]
i.e.\ $\sigma^{\L}_{s,r}$ satisfies \eqref{eq:sigmaL-factor}.
Dually, we obtain natural transformations
$\sigma^{\R}_{r,s} \colon \R_r\R_s\Rightarrow \R_{r+s+c}$. 
By construction, they agree with the mates of $\sigma^{\L}_{s,r}$ under the bijection
\eqref{eq:mate-bijection}.

Therefore, the discussion at the beginning of this section applies (to the transformations $\sigma^{\L}_{s,r}$ and their mates $\sigma^{\R}_{r,s}$) and shows that,
for any $M,X,N\in \Fun(P,\mathcal C)$,
composing an $s$-height-interleaving between $M$ and $X$ with an $r$-height-interleaving between $X$ and $N$ yields an $(s+r+c)$-height-interleaving between $M$ and $N$.
Hence $d_{\rho}$ satisfies a $c$-relaxed triangle inequality.

Finally, since {\rm (IV$_c$)} is monotone in $c$, the same conclusion holds for every
$c>c(\rho)$. Taking the infimum over such $c$ yields that $d_{\rho}$ satisfies a
$c(\rho)$-relaxed triangle inequality. In particular, if $c(\rho)=0$, then $d_{\rho}$
is an extended pseudo-distance.
\end{proof}

The following corollary is a formal consequence of Theorem~\ref{thm:relaxed TI} and functional stability: 
once a relaxed triangle inequality is established via CIP for a given height-difference function, it automatically transfers to any other height-difference function, up to an additive defect controlled by their distortion.

\begin{corollary}\label{cor:RTI-transport}
Let $\rho_1,\rho_2$ be height-difference functions on a poset $P$.
If $(P,\rho_1)$ has CIP, then 
$d_{\rho_2}$ satisfies a $(c(\rho_1)+3\delta(\rho_1,\rho_2))$-relaxed triangle inequality. 
\end{corollary}
\begin{proof}
By Theorem~\ref{thm:relaxed TI}, $d_{\rho_1}$ satisfies the $c(\rho_1)$-relaxed triangle inequality.
If $\delta(\rho_1,\rho_2)=\infty$, there is nothing to prove; assume $\delta(\rho_1,\rho_2)<\infty$.

Set $r:=\mathrm{dist}_{\infty}(d_{\rho_1},d_{\rho_2})$. By Theorem~\ref{thm:functional stability}, $r\le \delta(\rho_1,\rho_2)$.
Let $\epsilon>0$. By the definition of $r$, for all $A,B\in\Fun(P,\mathcal{C})$ we have
\[
d_{\rho_2}(A,B)\le d_{\rho_1}(A,B)+r+\epsilon
\qand
d_{\rho_1}(A,B)\le d_{\rho_2}(A,B)+r+\epsilon.
\]
Hence for $M,X,N\in\Fun(P,\mathcal{C})$,
\[
\begin{aligned}
d_{\rho_2}(M,N)
&\le d_{\rho_1}(M,N)+r+\epsilon\\
&\le d_{\rho_1}(M,X)+d_{\rho_1}(X,N)+c(\rho_1)+r+\epsilon\\
&\le \bigl(d_{\rho_2}(M,X)+r+\epsilon\bigr)+\bigl(d_{\rho_2}(X,N)+r+\epsilon\bigr)+c(\rho_1)+r+\epsilon\\
&= d_{\rho_2}(M,X)+d_{\rho_2}(X,N)+c(\rho_1)+3(r+\epsilon).
\end{aligned}
\]
Since $\epsilon>0$ is arbitrary, letting $\epsilon\to0$ yields
\[
d_{\rho_2}(M,N)\le d_{\rho_2}(M,X)+d_{\rho_2}(X,N)+c(\rho_1)+3r
\le d_{\rho_2}(M,X)+d_{\rho_2}(X,N)+ c(\rho_1)+3\delta(\rho_1,\rho_2),
\]
as desired.
\end{proof}

We conclude this subsection with a few remarks and conventions that will be used throughout the sequel.

\begin{convention}\label{conv:phi-c}
We continue the notation from Convention~\ref{conv:phi-notation}.
Throughout, whenever $\rho=\rho_\phi$ is induced by a height function $\phi$, we will write the associated notions and constructions using $\phi$.
For instance, we say that $(P,\phi)$ satisfies CIP if $(P,\rho_{\phi})$ does. 
In addition, the condition {\rm (IV$_c$)} for $\rho_\phi$ is equivalent to the following:
\begin{itemize}
\item For any $a\leq b$ in $P$ and any $p\in[\phi(a),\phi(b)]$,
there exists $z\in P$ with $a\leq z\leq b$ such that $|\phi(z)-p|\leq c/2$.
\end{itemize}
Accordingly, $c(\phi):=c(\rho_\phi)$ is precisely the infimum of $c\geq 0$ for which $\phi$ satisfies the above condition.

Moreover, if $P$ is locally finite (i.e., every interval $[a,b]$ is finite), then
\[
c(\phi)=\sup_{a\lessdot b}\,(\phi(b)-\phi(a)),
\]
where the supremum is taken over all cover relations $a\lessdot b$ in $P$.
In particular, this holds for any finite poset.
\end{convention}

The additive defect in Theorem~\ref{thm:relaxed TI} is best possible in general, as explained in the following remark.

\begin{remark}\label{rem:sharpness-cphi}
(1) The additive defect $c(\rho)$ in Theorem~\ref{thm:relaxed TI} is sharp
in the sense that it cannot be improved uniformly under CIP.

Indeed, let $(P,\phi)$ and $M,X,N\in \Fun(P,\Vect_k)$ be as in Example~(2) of Subsection \ref{sec:examples}.
In this case $(P,\phi)$ satisfies CIP, and one has $c(\phi)=C$.
Moreover, \eqref{eq:defect C} shows that $d_{\phi}(M,X)=d_{\phi}(X,N)=0$ and $d_{\phi}(M,N)=C$. 
Thus any inequality of the form
\[
d_{\phi}(M,N)\leq d_{\phi}(M,X)+d_{\phi}(X,N)+c
\]
forces $c\geq C$. 

(2) Moreover, without the assumption of CIP, the distance $d_{\rho}$ may fail to satisfy a $c(\rho)$-relaxed triangle inequality.
Indeed, for the pair $(B,\phi)$ in Example~\ref{ex:S1}, which does not satisfy CIP, there exist $M,M_1,N\in \Fun(B,\Vect_k)$ such that
\[
d_{\phi}(M,N) > d_{\phi}(M,M_1) + d_{\phi}(M_1,N) + c(\phi).
\]
\end{remark}

We next discuss situations in which CIP can be verified, thereby making Theorem~\ref{thm:relaxed TI} applicable. 

\subsubsection{The case of diamond-free posets}
\label{sec:tree-posets}

A poset $P$ is called \emph{diamond-free} if it has no full subposet isomorphic to the four-element diamond poset, namely $\{a<b,\ a<c,\ b<d,\ c<d\}$.
Note that this is equivalent to requiring that every interval $[a,b]$ in $P$ be totally ordered. 
This class of posets naturally includes zigzag posets (with arbitrary orientations), as well as finite posets whose Hasse diagrams are either trees or of affine $\widetilde{A}$-type with no commutative paths.

\begin{proposition}\label{prop:tree-CIP}
If $P$ is diamond-free, then $(P,\rho)$ has CIP for any height-difference function $\rho$ on $P$.
In particular, the distance $d_{\rho}$ satisfies a $c(\rho)$-relaxed triangle inequality.
\end{proposition}

\begin{proof}
Fix $s,r\geq 0$ and $a,q\in P$. 
By definition, the set $I_{s,r}(a,q):=a^{\downarrow_s}\cap q^{\uparrow_r}$ is a full subposet of the interval $[q,a]$ whenever it is non-empty. 
By assumption, the interval $[q,a]$ is totally ordered. Hence its full subposet $I_{s,r}(a,q)$ is also totally ordered, and therefore connected. This verifies CIP for $(P,\rho)$.
\end{proof}

The following example gives a continuous zigzag poset $Z$ to which Proposition~\ref{prop:tree-CIP} applies.

\begin{example}
Let $Z$ be the set $\mathbb{R}$ equipped with the partial order $\leq_Z$ defined as follows.
If $x,y\in[n,n+1]$ for some $n\in\mathbb{Z}$, then
\[
x\leq_Z y
\quad\Longleftrightarrow\quad
\begin{cases}
x\leq y & \text{if }n \text{ is even},\\
x\geq y & \text{if }n \text{ is odd}.
\end{cases}
\]
Otherwise, $x$ and $y$ are incomparable.
Then every even integer $2k$ is a minimal element and every odd integer $2k+1$ is a maximal element of $Z$.
Clearly, $Z$ is diamond-free. 

Define a height function $\phi\colon Z\to\mathbb{R}$ as follow: 
if $x\in[n,n+1]$ for some $n\in\mathbb{Z}$, then
\[
\phi(x)=
\begin{cases}
x-n & \text{if }n \text{ is even},\\
n+1-x & \text{if }n \text{ is odd}.
\end{cases}
\]
In particular, $\phi(2k)=0$ and $\phi(2k+1)=1$ for all $k\in\mathbb{Z}$.
Since $c(\phi)=0$, Proposition~\ref{prop:tree-CIP} implies that $d_{\phi}$ is an extended pseudo-distance on $\Fun(Z,\mathcal{C})$.

Set $M:=k_{[-1,2]}$ and $N:=k_{[-1,0)}\oplus k_{(0,2]}$. 
Then, one checks that $d_{\phi}(M,N)=0$ with respect to the above $\phi$.

On the other hand, $M$ and $N$ are not interleaved with respect to any translation $\Lambda\colon Z\to Z$.
Indeed, every translation $\Lambda$ fixes each minimal element, so in particular $\Lambda(0)=0$.
Thus, the canonical map
$M(0)\to (\Lambda^*\Lambda^*M)(0)=M(0)$ is the identity, while
$(\Lambda^*N)(0)=N(\Lambda(0))=N(0)=0$ forces any composite
$M(0)\to (\Lambda^*N)(0)\to (\Lambda^*\Lambda^*M)(0)$ to be zero.
This shows that $M$ and $N$ are not $\Lambda$-interleaved. 

Consequently, for any superlinear family of translations $\Omega$, the modules $M$ and $N$ are not $\Omega_r$-interleaved for any $r\geq 0$ by the argument above.
Hence $d_{\Omega}(M,N)=\infty$.
\end{example}

We next give an example of a poset together with a choice of height function that does not satisfy CIP.

\begin{example}\label{ex:S1}
Let $S^1$ be the circle and fix two distinguished points $s,t\in S^1$.
Removing $\{s,t\}$ decomposes $S^1$ into two open arcs, whose closures we denote by $Q^1$ and $Q^2$.
Then $Q^1\cap Q^2=\{s,t\}$ and $Q^1\cup Q^2=S^1$.
Let $B$ be $S^1$ endowed with the partial order in which each $Q^i$ is a chain from $s$ to $t$, and there are no other comparabilities. This poset is, up to poset isomorphism, the one considered in \cite{Tada25} (called the bipath poset in their terminology).

Fix $G>4$. 
We identify each chain $Q^i$ with the interval $[0,G]$ via a poset isomorphism.
Accordingly, we write
\[
Q^1=\{c_z\}_{z\in[0,G]}\qand Q^2=\{d_z\}_{z\in[0,G]},
\]
with $c_0=d_0=s$ and $c_G=d_G=t$.
Define a height function 
$\phi\colon B\to\mathbb{R}$ by $\phi(c_z)=\phi(d_z)=z$ for all $z\in[0,G]$. 
It is straightforward to check that $(B,\phi)$ does not satisfy CIP, while $c(\phi)=0$.

Set $M:=k_B$, and define $M_1:=\L_1M$ and $N:=\L_1M_1=\L_1(\L_1M)$.
We claim that
\[
d_{\phi}(M,M_1)=1,\qquad d_{\phi}(M_1,N)=1,\qand d_{\phi}(M,N)=G/2.
\]
In particular, this shows that the $c(\phi)$-relaxed triangle inequality fails for the triple $(M,M_1,N)$.

A direct computation yields
\[
M_1 \cong k_{[c_1,t]\cup[d_1,t]}
\qand
N\cong k_{[c_2,t]}\oplus k_{[d_2,t]}.
\]
By Lemma~\ref{lem:id-LR_interleaving}, $M$ and $M_1$ (resp.\ $M_1$ and $N$) are $1$-height-interleaved with respect to $\phi$.
Since they are not $r$-height-interleaved for any $r<1$, it follows that
$d_{\phi}(M,M_1)=1$ and $d_{\phi}(M_1,N)=1$.

To compute $d_{\phi}(M,N)$, we record that
\[
\L_rM \cong 
\begin{cases}
k_{[c_r,t]\cup[d_r,t]} & 0\leq r \leq G,\\
0 & \textrm{else}.
\end{cases}
\]
In particular, $\Hom(\L_rM,N)=0$ for all $0\leq r<G$.
On the other hand, one has $e_{r,M}\neq 0$ for $r\leq G/2$, whereas $e_{r,M}=0$ for $r>G/2$.
Hence $M$ and $N$ are not $r$-height-interleaved for $r\leq G/2$, while they are $r$-height-interleaved for every $r>G/2$.
Therefore $d_{\phi}(M,N)=G/2$.
\end{example}

\subsubsection{Multiparameter case}

We now turn to the multiparameter setting $P=\mathbb{R}^d$.
Let $\rho_{\mathrm{diag}}$ be the height-difference function on $\mathbb{R}^d$ defined by \eqref{eq:rho-diag}. 
As shown in Proposition~\ref{prop:recoverRd}, the height-interleaving distance $d_{\rho_{\mathrm{diag}}}$ coincides with the classical interleaving distance $d_I$ on $\Fun(\mathbb{R}^d,\mathcal{C})$.
For completeness, we record that this is consistent with Theorem~\ref{thm:relaxed TI}.

\begin{proposition}\label{prop:Rd-CIP-IV0}
The pair $(\mathbb{R}^d,\rho_{\mathrm{diag}})$ satisfies CIP and {\rm (IV$_0$)}.
In particular, $c(\rho_{\mathrm{diag}})=0$ and $d_{\rho_{\mathrm{diag}}}$ is an extended pseudo-distance.
\end{proposition}

\begin{proof}
Let $a,q\in\mathbb{R}^d$ and $s,r\geq 0$.
We have
\[
a^{\downarrow_s}= (a-\vec{s})^{\downarrow} = \prod_{i=1}^d(-\infty,a_i-s]
\qand
q^{\uparrow_r}=(q+\vec{r})^{\uparrow} = \prod_{i=1}^d[q_i+r,\infty),
\]
hence
\[
I_{s,r}(a,q)=a^{\downarrow_s}\cap q^{\uparrow_r}
=\prod_{i=1}^d [q_i+r,\ a_i-s],
\]
which is empty or has a minimum element. Thus $(\mathbb{R}^d,\rho_{\mathrm{diag}})$ satisfies CIP.

Next, let $a\leq b$ and set $m:=\rho_{\mathrm{diag}}(a,b)=\min_i(b_i-a_i)$.
For any $t\in[0,m]$, put $z:=a+(t,\ldots,t)$.
Then $a\leq z\leq b$, and
\[
\rho_{\mathrm{diag}}(a,z)=t
\qand
\rho_{\mathrm{diag}}(z,b)=m-t=\rho_{\mathrm{diag}}(a,b)-t.
\]
This verifies {\rm (IV$_0$)}.
\end{proof}

Proposition~\ref{prop:Rd-CIP-IV0} also allows us to phrase comparisons with $d_I$ in terms of height-difference functions, which leads to the following terminology.

\begin{definition}\label{def:diagonal-dominating}
Let $P=\mathbb{R}^d$ with the product order, and let $\rho$ be a height-difference function on $P$.
We say that $\rho$ \emph{dominates the diagonal shifts} if for all $a\in P$ and $r\geq 0$,
\[
\rho(a,a+\vec r)\geq r
\qand
\rho(a-\vec r,a)\geq r. 
\]
\end{definition}

\begin{proposition}\label{prop:dI-le-drho}
Assume that $\rho$ dominates the diagonal shifts.
Then for any $M,N\in\Fun(\mathbb{R}^d,\mathcal C)$,
\[
d_I(M,N)\leq d_{\rho}(M,N).
\]
\end{proposition}

\begin{proof}
We first show that $\rho_{\mathrm{diag}}\leq \rho$.
Fix $a\leq b$ and set $r:=\rho_{\mathrm{diag}}(a,b)=\min_i(b_i-a_i)$.
Then $a+\vec r\leq b$.
Since $\rho$ dominates the diagonal shifts, we have $\rho(a,a+\vec r)\geq r$.
By \eqref{eq:superadditive}, we obtain
\[
\rho(a,b)\geq \rho(a,a+\vec r)+\rho(a+\vec r,b)\geq r=\rho_{\mathrm{diag}}(a,b).
\]
Thus $\rho_{\mathrm{diag}}\leq \rho$.

By Proposition~\ref{prop:monotonicity-rho}, it follows that $d_{\rho_{\mathrm{diag}}}\leq d_{\rho}$.
Since $d_{\rho_{\mathrm{diag}}}=d_I$ by Proposition~\ref{prop:recoverRd}, the claim follows.
\end{proof}

\subsection{Inherited relaxed triangle inequalities via Galois insertions}
\label{sec:RTI via GI}
Let $P,P'$ be posets. A pair of order-preserving maps
$\iota\colon P\to P'$ and $\pi\colon P'\to P$
is called a \emph{Galois insertion} if $\iota\dashv\pi$ and $\iota$ is an order-embedding. 
We identify $P$ with its image $\iota(P)$ via $\iota$, and regard it as a full subposet of $P'$.

\begin{theorem}\label{thm:RTI via GI}
Let $\rho'$ be a height-difference function on a poset $P'$.
Let $(\iota,\pi)$ be a Galois insertion $P\rightleftarrows P'$ and  $\rho:=\iota^{*}\rho'$ the restriction of $\rho$ to $P$.
Then the following statements hold. 
\begin{enumerate}[\rm (1)]
\item For any $M,N\in \Fun(P,\mathcal{C})$, we have
\begin{equation}\label{eq:completion ineq rho}
d_{\rho}(M,N) \leq 
d_{\rho'}(\pi^*M, \pi^*N) \leq d_{\rho}(M,N) + \delta(\rho',\pi^*\rho).
\end{equation}
%\item If $d_{\rho'}$ satisfies a $c$-relaxed triangle inequality, then $d_{\rho}$ satisfies a $(c+2\delta(\rho',\pi^*\rho))$-relaxed triangle inequality.

\item If $(P',\rho')$ has CIP, then $d_{\rho}$ satisfies a
$\bigl(c(\rho')+2\delta(\rho',\pi^*\rho)\bigr)$-relaxed triangle inequality.
\end{enumerate}
\end{theorem}

\begin{proof}
Since $\iota\dashv\pi$, for $a\in P$ and $x\in P'$ we have $\iota(a)\leq x$ if and only if $a\leq \pi(x)$.
In particular, $ \iota(\pi(x))\leq x$ for all $x\in P'$.

(1) Using $\pi\iota=\id_P$, we have $\iota^*\pi^*\simeq \id_{\Fun(P,\mathcal{C})}$, hence by pullback stability,
\[
d_{\rho}(M,N)=d_{\iota^*\rho'}(\iota^*\pi^*M,\iota^*\pi^*N)\leq d_{\rho'}(\pi^*M,\pi^*N),
\]
which gives the left-hand inequality in \eqref{eq:completion ineq rho}.

On the other hand, since $\rho=(\pi\iota)^*\rho = \iota^*(\pi^*\rho)$ by $\pi\iota=\id_P$, 
applying pullback stability twice gives
\[
d_{\rho}(M,N)=d_{\iota^*(\pi^*\rho)}(\iota^*\pi^*M,\iota^*\pi^*N)\leq d_{\pi^*\rho}(\pi^*M,\pi^*N)
\leq d_{\rho}(M,N), 
\]
hence
\[
d_{\pi^*\rho}(\pi^*M,\pi^*N)=d_{\rho}(M,N).
\]

Now, by functional stability (Theorem~\ref{thm:functional stability}) applied on $P'$,
\[
|d_{\rho'}(\pi^*M,\pi^*N)-d_{\pi^*\rho}(\pi^*M,\pi^*N)|_{\infty}\leq \delta(\rho',\pi^*\rho).
\]
In particular,
\[
d_{\rho'}(\pi^*M,\pi^*N)\leq d_{\pi^*\rho}(\pi^*M,\pi^*N)+\delta(\rho',\pi^*\rho)
= d_{\rho}(M,N)+\delta(\rho',\pi^*\rho),
\]
which is the right-hand inequality in \eqref{eq:completion ineq rho}.

(2) Assume $(P',\rho')$ has CIP. 
By Theorem~\ref{thm:relaxed TI}, $d_{\rho'}$ satisfies a $c(\rho')$-relaxed triangle inequality.
For $M,X,N\in \Fun(P,\mathcal{C})$, by (1) we have
\[
d_{\rho}(M,N)\leq d_{\rho'}(\pi^*M,\pi^*N)
\leq d_{\rho'}(\pi^*M,\pi^*X)+d_{\rho'}(\pi^*X,\pi^*N)+c(\rho').
\]
Applying (1) again to each term, we obtain
\[
d_{\rho'}(\pi^*M,\pi^*X)\leq d_{\rho}(M,X)+\delta(\rho',\pi^*\rho)
\qand
d_{\rho'}(\pi^*X,\pi^*N)\leq d_{\rho}(X,N)+\delta(\rho',\pi^*\rho).
\]
Combining these inequalities yields
\[
d_{\rho}(M,N)\leq d_{\rho}(M,X)+
d_{\rho}(X,N)+\big(c(\rho')+2\delta(\rho',\pi^*\rho)\big),
\]
as desired.
\end{proof}

\section{Height-interleavings for persistence modules over $\Vect_k$}
\label{sec:PM}

In this section, we restrict our attention to the case $\mathcal{C}=\Vect_k$ and focus on the category of persistence modules.
Accordingly, we work in the abelian category $\Rep_k(P):=\Fun(P,\Vect_k)$ and regard its objects as $P$-persistence modules.
Recall that kernels, cokernels and images in $\Rep_k(P)$ are computed pointwise.
In particular, a morphism in $\Rep_k(P)$ is mono/epi if and only if it is so at each $a\in P$.

Let $F,G\colon \Rep_k(P) \to \Rep_k(P)$ be endofunctors. 
We say that $F$ is a \emph{subfunctor} (resp., \emph{quotient functor}) of $G$ if there is a natural transformation $F\Rightarrow G$ such that its component $F(M) \to G(M)$ is a monomorphism (resp., an epimorphism) for every $M$.
In addition, we say that $F$ is a \emph{subquotient} of $G$ if 
there exists an endofunctor $H\colon \Rep_k(P)\to \Rep_k(P)$ such that  
$H$ is a subfunctor of $G$ and $F$ is a quotient functor of $H$. 
In this case, $F(M)$ is a subquotient of $G(M)$ functorially in $M$. 
Finally, for a given natural transformation $\alpha\colon F\Rightarrow G$, we write $\im(\alpha)$ (resp. $\ker(\alpha)$) for the endofunctor obtained by taking the image (resp. kernel) pointwise.

\subsection{Erosion functors}

Fix a height-difference function $\rho$ on $P$ and $r\geq 0$.
We have defined endofunctors $\L_r$ and $\R_r$ on $\Rep_k(P)$ together with natural transformations
$\eta_{r}^{\L}\colon \L_r\Rightarrow \id$ and $\eta_{r}^{\R}\colon \id \Rightarrow \R_{r}$.
Notice that both $\L_r$ and $\R_r$ are additive.

\begin{proposition}
$\L_r$ is right exact and $\R_r$ is left exact.
\end{proposition}

\begin{proof}
Since $\Rep_k(P)$ is abelian, any left adjoint functor is right exact and any right adjoint functor is left exact.
Then, the assertion follows from Proposition~\ref{prop:adjoint LrRr}.
\end{proof}

We now introduce a third endofunctor, called the $r$-erosion functor, which sits between $\L_r$ and $\R_r$.
See Remark~\ref{rem:erosion} for the motivation and background of this terminology.

\begin{definition}
For $r\geq 0$, we define an endofunctor
\[
\E_r:=\im\bigl(e_r\colon \L_r\Rightarrow \R_r\bigr)\colon \Rep_k(P)\to \Rep_k(P)
\]
and call it the \emph{$r$-erosion functor} (with respect to $\rho$).
\end{definition}

Explicitly, for $M\in \Rep_k(P)$, we set
\[
\E_rM := \im \bigl(e_{r,M} \colon \L_rM \to \R_rM\bigr) \subseteq \R_rM.
\]
For a morphism $f\colon M\to N$, the naturality of $e_r$ induces the commutative diagram
\begin{equation*}
    \xymatrix@C45pt@R34pt{
    \L_rM \ar[r]^{e_{r,M}} \ar[d]_{\L_r(f)} \ar@{}[rd]|{\circlearrowleft}& \R_rM \ar[d]^{\R_r(f)} \\
    \L_rN \ar[r]^{e_{r,N}} & \R_rN
    }
\end{equation*}
and the restriction of $\R_r(f)$ to $\E_rM$ defines a morphism $\E_r(f)\colon \E_rM\to \E_rN$.

Moreover, in the same spirit as $\L_r$ and $\R_r$, which admit pointwise descriptions in terms of left and right Kan extensions,
the erosion functor $\E_r$ relates to the \emph{intermediate extension} \cite{BBD,Kuhn94} (see also \cite{AET25} for its application to persistence modules).
More precisely, for each $a\in P$, consider the full subposet
$U_a:= a^{\downarrow_r} \cup a^{\uparrow_r}\subseteq P$ with the canonical inclusion $u_a\colon U_a\to P$.
We recall from \eqref{eq:LR as Ua} that, for $M\in \Rep_k(P)$,
the objects $\L_rM(a)$ and $\R_rM(a)$ can be written as pointwise Kan extensions along $u_a$, respectively.
With these identifications, $\E_rM(a)$ is given as the image of the canonical map
\[
(\Lan_{u_a}(M|_{U_a}))(a)\longrightarrow (\Ran_{u_a}(M|_{U_a}))(a)
\]
induced by the adjoint triple $\Lan_{u_a}\dashv u_a^{*}\dashv \Ran_{u_a}$.
This realizes $\E_r$ as a pointwise intermediate extension between $\L_r$ and $\R_r$.

\begin{proposition}\label{prop:erosion functor}
For $r,s\geq 0$, we have the following.
\begin{enumerate}[\rm (1)]
\item $\E_r$ preserves monomorphisms and epimorphisms.
\item For $s\geq r$, $\E_s$ is a subquotient of $\E_r$.
\item For all $s,r\geq0$, $\E_s\E_r$ is a subquotient of $\E_{s+r}$.
\end{enumerate}
\end{proposition}

\begin{proof}
(1) We recall a fundamental fact that, for a natural transformation $\alpha\colon F\Rightarrow G$ between endofunctors on an abelian category, the image of $\alpha$ preserves epimorphisms whenever $F$ does, and preserves monomorphisms whenever $G$ does.

By Proposition~\ref{prop:adjoint LrRr}, $\L_r$ is right exact and $\R_r$ is left exact.
In particular, $\L_r$ preserves epimorphisms and $\R_r$ preserves monomorphisms.
Thus, the claim follows from the above fact applied to the natural transformation $e_r\colon \L_r\Rightarrow \R_r$.

(2) Let $s\ge r$.
Recall the canonical factorization
\[
\L_sM \xrightarrow{\eta^{\L}_{s,r,M}} \L_rM
\xrightarrow{\eta^{\L}_{r,M}} M
\xrightarrow{\eta^{\R}_{r,M}} \R_rM
\xrightarrow{\eta^{\R}_{r,s,M}} \R_sM .
\]
Let
\[
\gamma_{s,r}\colon \L_s \Rightarrow \R_r
\qand
H_{s,r}:=\im(\gamma_{s,r}).
\]
Then, for each $M$, the above factorization shows that $H_{s,r}M$ is a subobject of $\E_rM$,
and that $\E_sM$ is a quotient of $H_{s,r}M$.
Since these constructions are natural in $M$, we conclude that $\E_s$ is a subquotient of $\E_r$.

(3) Let $s,r\geq 0$ and $M\in \Rep_k(P)$.
We write
\begin{equation}\label{eq:Er factorization}
\L_rM \overset{\pi}{\twoheadrightarrow} \E_rM \overset{\iota}{\hookrightarrow} \R_rM
\end{equation}
for the canonical surjection and inclusion.
Applying $\L_s$ and $\R_s$ gives the following commutative diagram:
\begin{equation*}
\xymatrix@C38pt@R30pt{
\L_s(\L_rM) \ar[r]^-{\mu_{s,r,M}^{\L}} \ar[d]^{\L_s(\pi)} &
\L_{s+r}M \ar[r]^-{\eta_{s+r,M}^{\L}} &
M \ar@{}[d]|{\circlearrowleft} \ar[r]^-{\eta_{s+r,M}^{\R}} &
\R_{s+r}M \ar[r]^-{\mu_{s,r,M}^{\R}} &
\R_s(\R_rM) \\
\L_s(\E_rM) \ar[rr]^-{\eta_{s,\E_rM}^{\L}} &&
\E_rM \ar[rr]^-{\eta_{s,\E_rM}^{\R}} &&
\R_s(\E_rM) \ar[u]_{\R_s(\iota)}.
}
\end{equation*}
%Here we use the natural transformation $\mu^{\R}_{s,r}\colon \R_{s+r}\Rightarrow \R_s\R_r$, constructed in the same way as $\mu^{\L}_{s,r}$ using \eqref{eq:superadditive}.

Since $\L_s(\pi)$ is an epimorphism and $\R_s(\iota)$ is a monomorphism, we obtain an isomorphism
\[
\E_s(\E_rM) := \im \bigl(\L_s(\E_rM) \to \R_s(\E_rM)\bigr)
\cong
\im\bigl(\L_s(\L_rM)\to \R_s(\R_rM)\bigr).
\]
By the same argument as in (2), we conclude that $\E_s(\E_rM)$ is a subquotient of $\E_{s+r}M$.
Since the construction is natural in $M$, the endofunctor $\E_s\E_r$ is a subquotient of $\E_{s+r}$.
\end{proof}

\begin{remark}\label{rem:erosion}
Our notion of $r$-erosion is analogous to the one defined in \cite{Bjerkevik25} for a superlinear family of translations. 
Let $\Omega=(\Omega_r)_{r\geq 0}$ be a strong superlinear family of translations on $P$. 
For a $P$-persistence module $M$, the $r$-erosion of $M$ with respect to $\Omega$ is defined as the image of the canonical map $(\Omega_{-r}^*M \to M \to \Omega_{r}^*M)$.

In particular, when $P=\mathbb{R}^d$ and $\rho=\rho_{\mathrm{diag}}$, 
our $r$-erosion agrees with the $r$-erosion associated with the $r$-shift functor on $\Rep_k(\mathbb{R}^d)$ by Proposition~\ref{prop:recoverRd}(1). 
\end{remark}

\subsection{Erosion neighborhoods}

In the setting of classical translation-based interleavings, the notion of $\epsilon$-erosion neighborhoods and the associated erosion neighborhood distance were introduced in \cite{Bjerkevik25}.
Motivated by this, we introduce the height-interleaving analogue (Definition~\ref{def:EN})
and establish the analogous basic properties, including a construction of $r$-height-interleavings from erosion neighborhoods (Theorem~\ref{thm:erosion neighborhood}), as well as
a comparison result between the height-interleaving distance $d_{\rho}$ and the erosion neighborhood distance
$d_{\rho\text{-EN}}$ (Theorem~\ref{thm:EN-summary}).

\medskip
To formulate erosion neighborhoods, we first introduce the image and kernel functors associated with the canonical maps $\eta_r^{\L}$ and $\eta_r^{\R}$.
We denote by $\Im_r,\Ker_r \colon \Rep_k(P)\to \Rep_k(P)$ the endofunctors defined as the image and kernel of the natural transformations
$\eta_{r}^{\L}\colon \L_r \Rightarrow \id$ and $\eta_{r}^{\R} \colon \id\Rightarrow \R_r$, respectively.
Explicitly, for $M\in \Rep_k(P)$, we set
\[
\Im_rM := \im (\eta_{r,M}^{\L}\colon \L_rM \to M) \subseteq M,
\qand
\Ker_rM := \ker (\eta_{r,M}^{\R}\colon M\to \R_rM) \subseteq M
\]
regarded as submodules of $M$.

We give basic properties of these functors.

\begin{proposition}\label{prop:KerIm}
The following statements hold.
\begin{enumerate}[\rm (1)]
\item $\Ker_r$ is left exact.
\item $\Im_r$ preserves monomorphisms and epimorphisms.
\item For $s\geq r$, we have $\Ker_r\subseteq \Ker_s$ and $\Im_s\subseteq \Im_r$ as subfunctors.
\item For all $s,r\geq 0$, $\Im_{s}\Im_{r}$ is a subfunctor of $\Im_{s+r}$.
\item For any $M\in \Rep_k(P)$, we have $\Ker_rM=M$ if and only if $\Im_rM=0$.
\end{enumerate}
\end{proposition}

\begin{proof}
(1) The assertion follows from the facts that $\Ker_r$ is defined as the kernel of $\eta_{r}^{\R}\colon \id\Rightarrow \R_r$ and
the kernel of a natural transformation between left exact functors is left exact.

(2) Recall that $\Im_r$ is defined as the image of the natural transformation $\eta_{r}^{\L}\colon \L_r\Rightarrow \id$.
Then, the assertion follows from a similar argument to the proof of Proposition~\ref{prop:erosion functor}(1).

(3) For $s\ge r$, we have factorizations of natural transformations
$\eta_s^{\L} = \eta_r^{\L}\circ \eta_{s,r}^{\L}$ and
$\eta_s^{\R} = \eta_{r,s}^{\R}\circ \eta_r^{\R}$.
Taking images and kernels, we obtain $\Im_s M \subseteq \Im_r M$ and $\Ker_r M \subseteq \Ker_s M$ for all $M\in \Rep_k(P)$, as claimed.

(4) Let $M\in \Rep_k(P)$.
We write
\[
\L_rM \xrightarrow{\;\pi\;} \Im_r M \xrightarrow{\;\iota\;} M
\]
for the canonical surjection and inclusion.
Then, we obtain the following commutative diagram
\begin{equation}\label{eq:ImIm_to_Im}
\xymatrix@C42pt@R32pt{
\L_s(\L_rM) \ar@{->}[r]^{\L_s(\pi)} \ar[d]_{\mu_{s,r,M}^{\L}} \ar@{}[rd]|{\circlearrowleft} &
\L_s(\Im_rM) \ar[r]^-{\eta_{s,\Im_rM}^{\L}} \ar[d]^{\L_s(\iota)} \ar@{}[rd]|{\circlearrowleft} &
\Im_rM \ar@{->}[d]^{\iota} \\
\L_{s+r}M \ar[r]^-{\eta_{s+r,s,M}^{\L}} &
\L_sM \ar[r]^{\eta_{s,M}^{\L}} &
M.
}
\end{equation}
Since $\L_s$ is right exact, $\L_s(\pi)$ is an epimorphism.
Thus, we identify $\Im_s(\Im_rM)$ with $\im(\L_s(\L_rM) \to M)$.
By commutativity, the morphism $\L_s(\L_rM)\to M$ factors through $\L_{s+r}M \to M$, and therefore,
\[
\Im_s(\Im_rM) = \im(\L_s(\L_rM) \to M) \subseteq \im(\L_{s+r}M \to M) = \Im_{s+r}M.
\]
This inclusion is natural in $M$, hence defines a natural transformation $\Im_s\Im_r\Rightarrow \Im_{s+r}$ whose components are monomorphisms.

(5) By definition, $M=\Ker_rM$ holds if and only if $\eta^{\R}_{r,M}\colon M\to \R_rM$ is the zero morphism.
Also, $\Im_rM=0$ holds if and only if $\eta^{\L}_{r,M}\colon \L_rM\to M$ is the zero morphism.
Under the adjunction $\L_r\dashv \R_r$, the morphisms $\eta^{\R}_{r,M}$ and $\eta^{\L}_{r,M}$ are mates of each other.
Hence $\eta^{\R}_{r,M}=0$ if and only if $\eta^{\L}_{r,M}=0$.
\end{proof}

By definition, the $r$-erosion $\E_rM$ is given as the image of the composite \eqref{eq:erosion} and hence a subquotient of $M$.
Explicitly, there is an isomorphism
\begin{equation}\label{eq:ErImKer}
\E_rM \cong \Im_rM/(\Im_rM\cap \Ker_rM).
\end{equation}

In view of \eqref{eq:ErImKer}, it is natural to consider more general subquotients of $M$ obtained by
enlarging $\Im_rM$ and/or shrinking $\Ker_rM$. 
This leads to the following definition.

\begin{definition}\label{def:EN}
Let $M\in \Rep_k(P)$ and $r \geq0$. 
An \emph{$r$-erosion neighborhood} of $M$ is a subquotient of $M$ of the form $M_1/M_2$, where $M_2 \subseteq M_1\subseteq M$ are submodules such that $\Im_rM \subseteq M_1$ and $M_2 \subseteq \Ker_rM$. 
We denote by $\EN_r(M)$ the set of isomorphism classes of $r$-erosion neighborhoods of $M$. 
\end{definition}

Our result is the following.

\begin{theorem}\label{thm:erosion neighborhood}
Let $M\in \Rep_k(P)$ and $r\geq 0$.
\begin{enumerate}[\rm (1)]
\item $\E_rM$ is an $r$-erosion neighborhood of $M$.
\item Let $N\in \EN_r(M)$. Then $M$ and $N$ are $r$-height-interleaved with respect to $\rho$.
\end{enumerate}
In particular, $M$ and $\E_rM$ are $r$-height-interleaved with respect to $\rho$.
\end{theorem}

\begin{proof}
(1) This is immediate from \eqref{eq:ErImKer}.

(2) Suppose that $N\cong M_1/M_2\in \EN_r(M)$, where $M_2\subseteq M_1\subseteq M$ with
$\Im_rM\subseteq M_1$ and $M_2\subseteq \Ker_rM$.
Let $\iota \colon M_1\hookrightarrow M$ and $\pi\colon M_1\twoheadrightarrow N$, so that
\[
0\longrightarrow M_2\longrightarrow M_1 \xrightarrow{\ \pi\ } N\longrightarrow 0
\]
is exact.

Define $\alpha\colon \L_rM\to N$ by
\[
\alpha\colon \L_rM \twoheadrightarrow \Im_rM \hookrightarrow M_1 \xrightarrow{\ \pi\ } N.
\]
Then we have $\alpha\circ \L_r(\iota)=\pi\circ \eta^{\L}_{r,M_1}$.
Moreover, since $M_2\subseteq \Ker_rM$, the composite $\eta^{\R}_{r,M}\circ \iota\colon M_1\to \R_rM$ vanishes on $M_2$ and hence factors uniquely through $\pi$.
Let $\beta\colon N\to \R_rM$ be the induced morphism, so that $\beta\circ \pi=\eta^{\R}_{r,M}\circ \iota$.

We claim that $(\alpha^{\flat},\beta)$ form an $r$-height-interleaving between $M$ and $N$. 
By construction,
\[
\beta\circ \alpha=\eta^{\R}_{r,M}\circ \eta^{\L}_{r,M}=e_{r,M}.
\]
It remains to show $\alpha^{\flat} \circ \beta^{\sharp} = e_{r,N}$. 
Since $\L_r$ is right exact, applying $\L_r$ to $\pi$ yields an epimorphism
$\L_r(\pi)\colon \L_rM_1\twoheadrightarrow \L_rN$.
Hence it suffices to prove
\[
(\alpha^{\flat}\circ \beta^{\sharp})\circ \L_r(\pi)=e_{r,N}\circ \L_r(\pi).
\]

By definition of $\beta^{\sharp}$ (as the left adjoint of $\beta$) and the relation $\beta\circ \pi=\eta^{\R}_{r,M}\circ \iota$,
we have
\begin{equation}\label{eq:beta-sharp}
\beta^{\sharp}\circ \L_r(\pi)=
\eta^{\L}_{r,M}\circ \L_r(\iota) = 
\iota \circ \eta^{\L}_{r,M_1}
\colon \L_rM_1\to M.
\end{equation}
Moreover, since $\alpha^{\flat}$ is the right adjoint of $\alpha$, the relation
$\alpha\circ \L_r(\iota)=\pi\circ \eta^{\L}_{r,M_1}$ 
implies
\begin{equation}\label{eq:alpha-flat}
\alpha^{\flat}\circ \iota = \eta^{\R}_{r,N}\circ \pi\colon M_1\to \R_rN.
\end{equation}

Combining these gives
\[
(\alpha^{\flat}\circ \beta^{\sharp})\circ \L_r(\pi)
\overset{\eqref{eq:beta-sharp}}{=}
(\alpha^{\flat}\circ \iota)\circ \eta^{\L}_{r,M_1}
\overset{\eqref{eq:alpha-flat}}{=}
\eta^{\R}_{r,N}\circ \pi\circ \eta^{\L}_{r,M_1}.
\]
By naturality of $\eta_r^{\L}$, we have
\[
\pi\circ \eta^{\L}_{r,M_1}=\eta^{\L}_{r,N}\circ \L_r(\pi).
\]
Therefore,
\[
(\alpha^{\flat}\circ \beta^{\sharp})\circ \L_r(\pi)
=\eta^{\R}_{r,N}\circ \eta^{\L}_{r,N}\circ \L_r(\pi)
=e_{r,N}\circ \L_r(\pi),
\]
as desired.
\end{proof}

Theorem~\ref{thm:erosion neighborhood} thus provides a systematic way to construct $r$-height-interleavings from suitable subquotients of $M$.

We end this subsection by defining a distance on $\Rep_k(P)$ based on erosion neighborhoods.
For $M,N\in \Rep_k(P)$, set
\begin{equation}\label{eq:def-drhoEN}
d_{\rho\text{-EN}}(M,N)
:=\inf\{\,r\ge 0 \mid \EN_r(M)\cap \EN_r(N)\neq \emptyset\,\}.
\end{equation}

The following result relates $d_{\rho\text{-EN}}$ to the height-interleaving distance $d_\rho$.

\begin{theorem}\label{thm:EN-summary}
Let $P$ be a poset equipped with a height-difference function $\rho$.
\begin{enumerate}[\rm (1)]
\item If $M$ and $N$ are $r$-height-interleaved with respect to $\rho$, then there exists
$Q\in \EN_r(M)\cap \EN_r(N)$.
In particular,
\begin{equation}\label{eq:dEN_le_drho}
d_{\rho\text{-EN}}(M,N)\le d_\rho(M,N).
\end{equation}

\item Assume that $(P,\rho)$ has CIP.
Then $d_{\rho\text{-EN}}$ satisfies a $c(\rho)$-relaxed triangle inequality, and for all
$M,N\in \Rep_k(P)$ we have
\begin{equation}\label{eq:ENvsI-under-CIP}
d_{\rho\text{-EN}}(M,N)\le d_\rho(M,N)\le 2\,d_{\rho\text{-EN}}(M,N)+c(\rho).
\end{equation}
In particular, if $c(\rho)=0$, then $d_{\rho\text{-EN}}$ is an extended pseudo-distance and
\[
d_{\rho\text{-EN}}(M,N)\le d_\rho(M,N)\le 2\,d_{\rho\text{-EN}}(M,N).
\]
\end{enumerate}
\end{theorem}

The proof of Theorem~\ref{thm:EN-summary} relies on three lemmas
(Lemmas~\ref{lem:char-EN}--\ref{lem:EN-triangle-inequality}) stated below.
Their proofs are given in Subsections~\ref{sec:EN-proof1}--\ref{sec:EN-proof3}.

\begin{lemma}\label{lem:char-EN}
Let $M\in \Rep_k(P)$ and $r\geq 0$.
Every module $N$ of the form
\begin{equation}\label{eq:char_EN}
N \cong \im\Big(
\L_rM \oplus X
\xrightarrow{
\footnotesize \begin{bmatrix} \eta_{r,M}^{\L} & f \end{bmatrix}}
M
\xrightarrow{
\footnotesize \begin{bmatrix} \eta_{r,M}^{\R} \\ g \end{bmatrix}}
\R_rM \oplus Y
\Big)
\end{equation}
belongs to $\EN_r(M)$.
Moreover, any $r$-erosion neighborhood of $M$ can be written in this form.
\end{lemma}

\begin{lemma}\label{lem:EN-composition}
Assume that $(P,\rho)$ has CIP and satisfies {\rm (IV$_c$)} for some $c\ge 0$.
Let $s,r\ge 0$.
If $N\in \EN_r(M)$ and $Q\in \EN_s(N)$, then $Q\in \EN_{s+r+c}(M)$.
\end{lemma}

\begin{lemma}\label{lem:EN-triangle-inequality}
Assume that $(P,\rho)$ has CIP and satisfies {\rm (IV$_c$)} for some $c\ge 0$.
Let $s,r\ge 0$.
If $Q_1 \in \EN_r(M)\cap \EN_r(X)$ and $Q_2\in \EN_s(X)\cap \EN_s(N)$, then there exists $Q_3 \in \EN_{s}(Q_1)\cap \EN_{r}(Q_2)$.
\end{lemma}

With the above lemmas in hand, we now complete the proof.

\begin{proof}[Proof of Theorem~\ref{thm:EN-summary}]
(1) Let $M$ and $N$ be $r$-height-interleaved with respect to $\rho$ via morphisms
$p\colon M\to \R_rN$ and $q\colon N\to \R_rM$.

Define
\begin{equation}\label{eq:Q for M}
Q :=
\im\Big(
\L_rM\oplus \L_rN \xrightarrow{\footnotesize
\begin{bmatrix} \eta_{r,M}^{\L} & q^{\sharp}\end{bmatrix}}
M \xrightarrow{\footnotesize
\begin{bmatrix} \eta_{r,M}^{\R}\\ p\end{bmatrix}}
\R_rM\oplus \R_rN
\Big).
\end{equation}
By Lemma~\ref{lem:char-EN}, this shows $Q\in \EN_r(M)$.

Using the interleaving identities and the adjunction squares
$\eta_{r,M}^{\R}\circ q^{\sharp}=q\circ \eta_{r,N}^{\L}$ and
$p\circ \eta_{r,M}^{\L}=\eta_{r,N}^{\R}\circ p^{\sharp}$,
we compute
\[
\begin{bmatrix}
\eta_{r,M}^{\R}\\[2pt] p
\end{bmatrix}
\circ
\begin{bmatrix}
\eta_{r,M}^{\L} & q^{\sharp}
\end{bmatrix}
=
\begin{bmatrix}
q\\[2pt] \eta_{r,N}^{\R}
\end{bmatrix}
\circ
\begin{bmatrix}
p^{\sharp} & \eta_{r,N}^{\L}
\end{bmatrix}.
\]
Therefore the two composites define the same morphism
$\L_rM\oplus \L_rN\to \R_rM\oplus \R_rN$, so their images coincide.
In particular,
\[
Q \cong
\im\Big(
\L_rM\oplus \L_rN \xrightarrow{\footnotesize
\begin{bmatrix} p^{\sharp} & \eta_{r,N}^{\L}\end{bmatrix}}
N \xrightarrow{\footnotesize
\begin{bmatrix} q\\ \eta_{r,N}^{\R}\end{bmatrix}}
\R_rM\oplus \R_rN
\Big),
\]
and hence $Q\in \EN_r(N)$.
Consequently, $Q\in \EN_r(M)\cap \EN_r(N)$.

(2) Assume that $(P,\rho)$ has CIP. 
If $c(\rho)=\infty$, then there is nothing to prove. 
Otherwise, take $c>c(\rho)$. 
Then $(P,\rho)$ satisfies {\rm (IV$_c$)}. 
The left inequality in \eqref{eq:ENvsI-under-CIP} is \eqref{eq:dEN_le_drho}.
For the right inequality, let $r>d_{\rho\text{-EN}}(M,N)$.
Then there exists $Q\in \EN_r(M)\cap \EN_r(N)$ by definition of $d_{\rho\text{-EN}}$.
By Theorem~\ref{thm:erosion neighborhood}(2), $Q$ is $r$-height-interleaved with both $M$ and $N$.
Hence $M$ and $N$ are $(2r+c)$-height-interleaved.
Taking the infimum over such $r$ yields
\[
d_{\rho\text{-EN}}(M,N)\le d_\rho(M,N)\le 2\,d_{\rho\text{-EN}}(M,N)+c.
\]
Finally, since $c>c(\rho)$ was arbitrary, it follows that 
\[
d_{\rho\text{-EN}}(M,N)\le d_\rho(M,N)\le 2\,d_{\rho\text{-EN}}(M,N)+c(\rho),
\]
as desired.

It remains to show that $d_{\rho\text{-EN}}$ satisfies a $c$-relaxed triangle inequality for every $c>c(\rho)$.
Take $Q_1\in \EN_r(M)\cap \EN_r(X)$ and $Q_2\in \EN_s(X) \cap \EN_s(N)$.
By Lemma~\ref{lem:EN-triangle-inequality}, there exists $Q_3\in \EN_{s}(Q_1)\cap \EN_{r}(Q_2)$.
Since $Q_1\in \EN_r(M)$ and $Q_2\in \EN_s(N)$, Lemma~\ref{lem:EN-composition} yields
$Q_3\in \EN_{s+r+c}(M)\cap \EN_{s+r+c}(N)$.
This finishes the proof.
\end{proof}

In the rest of this section, we prove Lemmas~\ref{lem:char-EN}--\ref{lem:EN-triangle-inequality}. 

\subsection{Proof of Lemma \ref{lem:char-EN}}
\label{sec:EN-proof1}

For given $f,g$, set
\[
\tilde f:= {\footnotesize \begin{bmatrix} \eta_{r,M}^{\L} & f \end{bmatrix}}
\colon \L_rM\oplus X\to M,
\qquad
\tilde g:= {\footnotesize \begin{bmatrix} \eta_{r,M}^{\R} \\ g \end{bmatrix}}
\colon M\to \R_rM\oplus Y.
\]
In the abelian category $\Rep_k(P)$ we have an isomorphism
\begin{equation}\label{eq:tilde_gf}
\im(\tilde g\circ \tilde f)\ \cong\ \im \tilde f\big/(\im \tilde f \cap \ker \tilde g).
\end{equation}

Suppose that $N\cong M_1/M_2$ for submodules $M_2\subseteq M_1\subseteq M$ with
$\Im_rM\subseteq M_1$ and $M_2\subseteq \Ker_rM$.
Take $X:=M_1$ and let $f\colon X=M_1\hookrightarrow M$ be the inclusion.
Take $Y:=M/M_2$ and let $g\colon M\twoheadrightarrow M/M_2$ be the quotient map.
Then $\im \tilde f=\im(\eta_{r,M}^{\L})+\im(f)=\Im_rM+M_1=M_1$.
Moreover,
\[
\ker \tilde g=\ker(\eta_{r,M}^{\R})\cap \ker(g)=\Ker_rM\cap M_2=M_2,
\]
since $M_2\subseteq \Ker_rM$ by assumption.
Thus \eqref{eq:tilde_gf} gives
\[
\im(\tilde g\circ \tilde f)\ \cong\ M_1/(M_1\cap M_2)=M_1/M_2\ \cong\ N,
\]
which is \eqref{eq:char_EN}.

Assume \eqref{eq:char_EN}.
Set $M_1:=\im \tilde f\subseteq M$ and $M_2:=M_1\cap \ker \tilde g\subseteq M_1$.
Then \eqref{eq:tilde_gf} yields $N\cong M_1/M_2$.
Moreover, $\Im_rM=\im(\eta_{r,M}^{\L})\subseteq \im \tilde f=M_1$.
Also $M_2\subseteq \ker \tilde g\subseteq \ker(\eta_{r,M}^{\R})=\Ker_rM$. 
Therefore $N\in \EN_r(M)$.

\subsection{Proof of Lemma~\ref{lem:EN-composition}}
\label{sec:EN-proof2}
Assume that $(P,\rho)$ has CIP and satisfies {\rm (IV$_c$)} for some $c\ge 0$.
Let $s,r\ge 0$. We first claim that $\Im_{s+r+c}\subseteq \Im_s\Im_r$ as a subfunctor.
We denote the natural transformation \eqref{eq:sigmaL_under_CIP} by
$\sigma_{s,r}^{\L}\colon \L_{s+r+c}\Rightarrow \L_s\L_r$.
Let $M\in \Rep_k(P)$.
By the factorization property \eqref{eq:sigmaL-factor} of $\sigma^{\L}_{s,r}$, the structure map
$\eta^{\L}_{s+r+c,M}\colon \L_{s+r+c}M\to M$ factors as
\[
\L_{s+r+c}M \xrightarrow{\sigma^{\L}_{s,r,M}} \L_s\L_rM
\xrightarrow{\mu^{\L}_{s,r,M}} \L_{s+r}M
\xrightarrow{\eta^{\L}_{s+r,M}} M .
\]
On the other hand, by \eqref{eq:ImIm_to_Im},
the image of the composite
$\eta^{\L}_{s+r,M}\circ \mu^{\L}_{s,r,M}\colon \L_s\L_rM\to M$
identifies with $\Im_s(\Im_rM)\subseteq M$.
Therefore,
\begin{equation}\label{eq:sub_s+r+c_sr}
\Im_{s+r+c}M=\im(\eta^{\L}_{s+r+c,M})\subseteq \Im_s(\Im_rM).
\end{equation}
Naturality in $M$ yields the claimed inclusion of subfunctors.

Now, we prove the claim for fixed $M,N,Q$.
Let $N\cong M_1/M_2$ with $\Im_rM\subseteq M_1$ and $M_2\subseteq \Ker_rM$,
and let $Q\cong N_1/N_2$ with $\Im_sN\subseteq N_1$ and $N_2\subseteq \Ker_sN$.
Let $f\colon M_1\twoheadrightarrow N\cong M_1/M_2$ be the canonical surjection, and set
$M_1':=f^{-1}(N_1)$ and $M_2':=f^{-1}(N_2)$.
Then $Q\cong N_1/N_2\cong M_1'/M_2'$.
To prove $Q\in \EN_{s+r+c}(M)$, it remains to show that $\Im_{s+r+c}M\subseteq M_1'$ and $M_2'\subseteq \Ker_{s+r+c}M$.

We first show that $\Im_{s+r+c}M\subseteq M_1'$.
Since $\Im_rM\subseteq M_1$ by assumption, functoriality of $\Im_s$ yields
\[
\Im_s(\Im_rM)\subseteq \Im_s(M_1).
\]
On the other hand, from $\Im_sN\subseteq N_1$ we obtain $\Im_s(M_1)\subseteq M_1'$:
indeed, naturality of $\eta^{\L}_{s}$ gives
$f(\Im_s(M_1))\subseteq \Im_s(N)\subseteq N_1$, hence $\Im_s(M_1)\subseteq f^{-1}(N_1)=M_1'$.
Combining it with \eqref{eq:sub_s+r+c_sr}, we conclude $\Im_{s+r+c}M \subseteq M_1'$.

Next, we show that $M_2' \subseteq \Ker_{s+r+c}M$.
With $f$ and $M_2'$ as above, the restriction $f':=f|_{M_2'}\colon M_2'\twoheadrightarrow N_2$ is surjective.
Since $N_2\subseteq \Ker_s(N)$, we have $\eta^{\R}_{s,N_2}=0$, hence $\eta^{\L}_{s,N_2}=0$ by Proposition~\ref{prop:KerIm}(5).

Applying $\L_s$ to $f'$ yields an epimorphism $\L_s(f')\colon \L_s(M_2')\twoheadrightarrow \L_s(N_2)$, and naturality gives
\[
f'\circ \eta^{\L}_{s,M_2'}
=
\eta^{\L}_{s,N_2}\circ \L_s(f')
=0.
\]
In particular, $\eta^{\L}_{s,M_2'}$ has image $\Im_s(M_2')$, so
\[
\Im_s(M_2')=\im(\eta^{\L}_{s,M_2'})\subseteq \ker(f')=\ker(f)=M_2.
\]

Since $M_2\subseteq \Ker_r(M)$, it follows that $\eta^{\R}_{r,M_2'}$ vanishes on $\Im_s(M_2')$, and hence
\[
\eta^{\R}_{r,M_2'}\circ \eta^{\L}_{s,M_2'}\colon \L_s(M_2')\longrightarrow \R_r(M_2')
\]
is zero.
Under $\L_s\dashv \R_s$, the right adjoint of this morphism is
\[
M_2' \xrightarrow{\eta^{\R}_{r,M_2'}} \R_r(M_2')
\xrightarrow{\eta^{\R}_{s,\R_r(M_2')}} \R_s\R_r(M_2'),
\]
so this composite is zero.
Composing with $\sigma^{\R}_{s,r,M_2'}\colon \R_s\R_r(M_2')\to \R_{s+r+c}(M_2')$ and using
the factorization \eqref{eq:sigmaR-factor}, we obtain $\eta^{\R}_{s+r+c,M_2'}=0$.
Thus $M_2'\subseteq \Ker_{s+r+c}(M_2')\subseteq \Ker_{s+r+c}(M)$. 
This completes the proof.

\subsection{Proof of Lemma~\ref{lem:EN-triangle-inequality}}
\label{sec:EN-proof3}
Take $Q_1\in \EN_r(M)\cap \EN_r(X)$ and $Q_2\in \EN_s(X)\cap \EN_s(N)$.
We may choose submodules
\[
X_1'\subseteq X_1\subseteq X
\qand
X_2'\subseteq X_2\subseteq X
\]
such that
\[
Q_1\cong X_1/X_1',
\qquad
\Im_rX\subseteq X_1,
\qquad
X_1'\subseteq \Ker_rX,
\]
and
\[
Q_2\cong X_2/X_2',
\qquad
\Im_sX\subseteq X_2,
\qquad
X_2'\subseteq \Ker_sX.
\]

Set
\[
X_3:=X_1\cap X_2
\qand
X_3':=(X_1'+X_2')\cap X_3\subseteq X_3,
\qquad
Q_3:=X_3/X_3'.
\]
We will show that $Q_3\in \EN_s(Q_1)$.

Let $\pi_1\colon X_1\twoheadrightarrow Q_1:=X_1/X_1'$ be the canonical epimorphism.
Since $X_1'\cap X_3\subseteq X_3'$, we obtain 
\[
Q_3=X_3/X_3'\ \cong\ (X_3+X_1')/(X_3'+X_1')\ =\ \pi_1(X_3)/\pi_1(X_3').
\]
Thus, to prove $Q_3\in \EN_s(Q_1)$ it remains to show that
\[
\Im_s(Q_1)\subseteq \pi_1(X_3)
\qand
\pi_1(X_3')\subseteq \Ker_s(Q_1).
\]

By naturality of $\eta^{\L}_s$, we have
\[
\pi_1\circ \eta^{\L}_{s,X_1}
=
\eta^{\L}_{s,Q_1}\circ \L_s(\pi_1).
\]
Since $\L_s$ preserves epimorphisms, $\L_s(\pi_1)$ is an epimorphism, and hence
\[
\Im_s(Q_1)
=\im(\eta^{\L}_{s,Q_1})
=\im\big(\eta^{\L}_{s,Q_1}\circ \L_s(\pi_1)\big)
=\im\big(\pi_1\circ \eta^{\L}_{s,X_1}\big)
=\pi_1\big(\Im_s(X_1)\big).
\]
Moreover, since $\Im_s$ preserves monomorphisms, 
we obtain $\Im_s(X_1)\subseteq \Im_s(X)\subseteq X_2$.
Hence, 
\[
\Im_s(X_1)\subseteq X_1\cap \Im_s(X)\subseteq X_1\cap X_2=X_3,
\]
and therefore $\Im_s(Q_1)=\pi_1(\Im_s(X_1))\subseteq \pi_1(X_3)$.

On the other hand, by naturality of $\eta^{\R}_s$ we have
\begin{equation}
\eta^{\R}_{s,Q_1}\circ \pi_1=\R_s(\pi_1)\circ \eta^{\R}_{s,X_1}.
\end{equation}
Therefore,
\begin{equation}\label{eq:pi1-Ker}
\pi_1\big(\Ker_s(X_1)\big)\subseteq \Ker_s(Q_1).
\end{equation}

Next, since $X_3\subseteq X_1$ and $X_1'\subseteq X_1$, we obtain
\begin{equation}
X_3'=(X_1'+X_2')\cap X_3
\subseteq (X_1'+X_2')\cap X_1
\subseteq X_1'+(X_2'\cap X_1).
\end{equation}

Moreover, since $X_2'\subseteq \Ker_s(X)$ and $\Ker_s$ preserves monomorphisms, we have
\begin{equation}
X_2'\cap X_1\subseteq \Ker_s(X)\cap X_1 = \Ker_s(X_1).
\end{equation}
Combining these yields
\begin{equation}\label{eq:X3prime-in-X1prime+Ker}
X_3'\subseteq X_1'+\Ker_s(X_1).
\end{equation}
Applying $\pi_1$ to \eqref{eq:X3prime-in-X1prime+Ker} and using \eqref{eq:pi1-Ker}, we conclude that
\[
\pi_1(X_3')
\subseteq \pi_1(X_1')+\pi_1(\Ker_s(X_1))
=0+\pi_1(\Ker_s(X_1))
\subseteq \Ker_s(Q_1),
\]
as desired. 

We have shown that $Q_3\in \EN_s(Q_1)$.
By the same argument with the roles of $r$ and $s$ interchanged
we also obtain $Q_3\in \EN_r(Q_2)$.

\section{Discussion and perspective}
\label{sec:discussion}

%\medskip
\noindent\emph{Further links to classical interleavings.}
By Proposition~\ref{prop:recoverRd}, our height-interleaving distance recovers the classical interleaving distance on $\Fun(\mathbb{R}^d,\mathcal{C})$.
It is natural to ask to what extent familiar results from classical, translation-based interleaving theory can be formulated in our setting, as a direction for future work.
Our results in Section~\ref{sec:PM} provide support in this direction. 
Moreover, since our construction of $r$-latching/matching functors is inspired by ideas from model category theory, it may also be interesting to explore related directions, for instance those involving homotopy interleavings.

\medskip\noindent\emph{Representation-theoretic viewpoint.}
For finite $P$, the functor category $\Fun(P,\mathcal{C})$ can be identified with the module category of the associated incidence algebra. 
In this setting, our distance $d_{\rho}$ may provide quantitative invariants for analyzing module categories, and one may ask for a representation-theoretic interpretation of these invariants.

\medskip\noindent\emph{Applications to zigzag persistence and related settings.}
Our framework has the potential to be useful for the quantitative study of persistence modules indexed by zigzag and, more generally, by diamond-free posets.
In continuous settings, for suitable choices of $\rho$, the resulting height-interleaving distance defines an extended pseudo-distance; for more general diamond-free posets, including discrete or finite ones, the same construction provides a distance-like comparison up to an additive defect.
Developing effective computational methods and establishing stability results for data-driven persistent homology constructions are natural directions for further work.

\appendix
\section{Mates correspondence for adjunctions}\label{app:mates}

For basic notions and terminology in category theory, we refer the reader to \cite{MacLane}. 

\subsection{The basic mates correspondence}\label{appendix:mates}
Let $\mathcal{A},\mathcal{B}$ be categories and let
\[
L:\mathcal{A}\rightleftarrows \mathcal{B}:R
\]
be an adjoint pair, written $L\dashv R$.
We denote by $\eta\colon\id_{\mathcal{A}}\Rightarrow RL$ the unit and by
$\varepsilon\colon LR\Rightarrow \id_{\mathcal{B}}$ the counit.
For any $A\in\mathcal{A}$ and $B\in\mathcal{B}$, we write
\[
(-)^{\flat}:\Hom_{\mathcal{B}}(L(A),B) \cong \Hom_{\mathcal{A}}(A,R(B)) :(-)^{\sharp}
\]
for the adjunction bijection.

Let $F:\mathcal{X}\to\mathcal{A}$ and $G:\mathcal{X}\to\mathcal{B}$ be functors.
There is a canonical bijection
\begin{equation}\label{eq:mates-bijection-basic}
\mathrm{Nat}(LF,G)\ \cong\ \mathrm{Nat}(F,RG),
\end{equation}
natural in $F$ and $G$.
Given $\alpha:LF\Rightarrow G$, its \emph{right mate} $\alpha^{\flat}:F\Rightarrow RG$
is defined componentwise by
\begin{equation}\label{eq:mate-right-formula}
(\alpha^{\flat})_x
:= R(\alpha_x)\circ \eta_{F(x)}
\quad (x\in\mathcal{X}).
\end{equation}
Conversely, given $\beta:F\Rightarrow RG$, its \emph{left mate} $\beta^{\sharp}:LF\Rightarrow G$
is defined componentwise by
\begin{equation}\label{eq:mate-left-formula}
(\beta^{\sharp})_x
:= \varepsilon_{G(x)}\circ L(\beta_x)
\quad (x\in\mathcal{X}).
\end{equation}
These assignments are mutually inverse, hence yield \eqref{eq:mates-bijection-basic}.

The mates correspondence satisfies the following standard compatibilities.

\begin{lemma}\label{lem:mates-properties}
Under the bijection \eqref{eq:mates-bijection-basic}:
\begin{enumerate}[\rm (1)]
\item %\textbf{Naturality.}
Let $\alpha:LF\Rightarrow G$ and $\beta:F\Rightarrow RG$ be natural transformations. 
For any $\theta:F'\Rightarrow F$ and $\psi:G\Rightarrow G'$, we have
\[
(\psi\circ \alpha)^{\flat}=R\psi\circ \alpha^{\flat},
\qquad
(\alpha\circ L\theta)^{\flat}=\alpha^{\flat}\circ \theta,
\]
and
\[
(R\psi\circ \beta)^{\sharp}=\psi\circ \beta^{\sharp},
\qquad
(\beta\circ \theta)^{\sharp}=\beta^{\sharp}\circ L\theta.
\]

\item %\textbf{Units and counits.}
The right mate of $\id_{LF}:LF\Rightarrow LF$ is the unit
$\eta_F:F\Rightarrow RLF$.
The left mate of $\id_{RG}:RG\Rightarrow RG$ is the counit
$\varepsilon_G:LRG\Rightarrow G$.
In particular, the triangle identities for $(\eta,\varepsilon)$ are precisely the statements that
$(-)^{\flat}$ and $(-)^{\sharp}$ are mutually inverse.

\item %\textbf{Isomorphisms.}
A natural transformation is an isomorphism if and only if its mate is an isomorphism.
\end{enumerate}
\end{lemma}

\subsection{A two-adjunction variant}
\label{appendix:two-adjunction-mates}

We will also need a variant of \eqref{eq:mates-bijection-basic} relating two adjunctions.
Let $L\dashv R$ and $L'\dashv R'$ be adjunctions between categories $\mathcal A$ and $\mathcal B$.
Applying \eqref{eq:mates-bijection-basic} with $\mathcal X=\mathcal A$, $F=\id_{\mathcal A}$ and $G=L$ for the adjunction $L'\dashv R'$
identifies
\[
\mathrm{Nat}(L',L)\ \cong\ \mathrm{Nat}(\id_{\mathcal A},R'L).
\]
On the other hand, applying \eqref{eq:mates-bijection-basic} with $\mathcal X=\mathcal B$, $F=R$ and $G=\id_{\mathcal B}$ for the adjunction $L\dashv R$
identifies
\[
\mathrm{Nat}(\id_{\mathcal A},R'L)\ \cong\ \mathrm{Nat}(R,R').
\]
Composing these canonical bijections yields the following correspondence.

\begin{corollary}\label{cor:mates-two-adj}
Let $L\dashv R$ and $L'\dashv R'$ be adjunctions between $\mathcal A$ and $\mathcal B$.
Then there is a canonical bijection
\[
\mathrm{Nat}(L',L)\ \cong\ \mathrm{Nat}(R,R'),
\]
natural in $L,L',R,R'$.
\end{corollary}

\section{Limits, colimits, and finality for posets
}

\subsection{Limits and colimits}\label{appendix:limits-colimits}

Let $\mathcal C$ be a category. A diagram of shape $J$ in $\mathcal{C}$ is a functor $M\colon \mathcal J\to \mathcal C$.

A \emph{cone} to $M$ is a pair $(X,\gamma)$ consisting of an object $X\in\mathcal C$ and a family of morphisms
\[
\gamma_j\colon X\to M(j)\qquad (j\in\mathcal J)
\]
such that for every morphism $u\colon j\to j'$ in $\mathcal J$ we have
\[
M(u)\circ \gamma_j=\gamma_{j'}.
\]

A \emph{limit} of $M$ is an object $\lim M\in\mathcal C$ equipped with a cone $(\lim M,\beta_M)$,
\[
\beta_M(j)\colon \lim M\to M(j)\qquad (j\in\mathcal J),
\]
satisfying the following universal property:
for every cone $(X,\gamma)$ to $M$, there exists a unique morphism
$h\colon X\to \lim M$ such that $\beta_M(j)\circ h=\gamma_j$ for all $j\in\mathcal J$.

Dually, a \emph{cocone} from $M$ is a pair $(X,\delta)$ consisting of an object $X\in\mathcal C$ and a family of morphisms
\[
\delta_j\colon M(j)\to X\qquad (j\in\mathcal J)
\]
such that for every morphism $u\colon j\to j'$ in $\mathcal J$ we have
\[
\delta_{j'}\circ M(u)=\delta_j.
\]

A \emph{colimit} of $M$ is an object $\colim M\in\mathcal C$ equipped with a cocone $(\colim M,\alpha_M)$,
\[
\alpha_M(j)\colon M(j)\to \colim M\qquad (j\in\mathcal J),
\]
satisfying the following universal property:
for every cocone $(X,\delta)$ from $M$, there exists a unique morphism
$k\colon \colim M\to X$ such that $k\circ \alpha_M(j)=\delta_j$ for all $j\in\mathcal J$.

Write $\Fun(\mathcal J,\mathcal C)$ for the functor category and
$\Delta\colon \mathcal C\to \Fun(\mathcal J,\mathcal C)$ for the constant diagram functor.
If $\mathcal C$ is complete and cocomplete, then the assignments
$M\mapsto \lim M$ and $M\mapsto \colim M$ define functors
\[
\lim\colon \Fun(\mathcal J,\mathcal C)\to \mathcal C,
\qquad
\colim\colon \Fun(\mathcal J,\mathcal C)\to \mathcal C,
\]
and we have adjunctions
\[
\colim \dashv \Delta \dashv \lim.
\]
We denote by
$\alpha\colon \id \Rightarrow \Delta(\colim(-))$ the unit of $\colim\dashv \Delta$ and by
$\beta\colon \Delta(\lim(-)) \Rightarrow \id$ the counit of $\Delta\dashv \lim$.
Thus, for each diagram $M$ we have components
\[
\alpha_M(j)\colon M(j)\to \colim M,
\qquad
\beta_M(j)\colon \lim M\to M(j)
\qquad (j\in\mathcal J),
\]
which are the chosen colimiting cocone and limiting cone, respectively.

If $\mathcal I$ is a full subcategory of $\mathcal J$, we write $M|_{\mathcal I}$
for the restriction of $M$ along the inclusion $\mathcal I\hookrightarrow \mathcal J$, and we use the notations
\[
\lim M|_{\mathcal I},\ \ \colim M|_{\mathcal I}
\qquad\text{or equivalently}\qquad
\lim_{i\in\mathcal I} M(i),\ \ \colim_{i\in\mathcal I} M(i).
\]
In this case, the restriction of the limiting cone $(\lim M,\beta_M)$ to $\mathcal I$ is a cone to $M|_{\mathcal I}$,
hence induces a canonical morphism
\[
\lim M\longrightarrow \lim M|_{\mathcal I}.
\]
Dually, the colimiting cocone $(\colim M|_{\mathcal I},\alpha_{M|_{\mathcal I}})$ induces a canonical morphism
\[
\colim M|_{\mathcal I}\longrightarrow \colim M.
\]

\subsection{Iterated colimits and the Grothendieck construction for posets}\label{appendix:iterated-colimits}
In this section, we recall the Grothendieck construction (see, e.g., \cite{MacLane}) in the poset case, which provides the indexing poset for our comparison between iterated colimits and a single colimit. 

Let $I$ be a poset and let $F\colon I\to {\rm Poset}$ be a functor from $I$ to the category {\rm Poset} of posets, namely, for each $x\in I$ we have a poset $F(x)$, and for each $x\leq y$ an order-preserving map 
$F(x\leq y)\colon F(x)\to F(y)$.

We define a poset $\int_I F$ as follows:
its elements are pairs $(x,q)$ with $x\in I$ and $q\in F(x)$, and we set
\[
(x,q)\le (y,u)
\quad\Longleftrightarrow\quad
\text{$x\leq y$ in $I$ and $F(x\le y)(q)\le u$ in $F(y)$.}
\]
The projection $\int_I F\to I$, $(x,q)\mapsto x$, is order-preserving.

Now let $\mathcal C$ be a cocomplete category and let $D\colon \int_I F\to \mathcal C$ be a functor.
For each $x\in I$, we write
\[
C_x:=\colim_{q\in F(x)} D(x,q).
\]
For $x\le y$, the assignment $q\mapsto (x,q)\le\big(y,F(x\le y)(q)\big)$ in $\int_I F$
induces a natural transformation
\[
D(x,-)\Longrightarrow D\big(y,F(x\le y)(-)\big),
\]
and hence a canonical morphism $C_x\to C_y$ by the universal property of the colimit.
Thus $x\mapsto C_x$ defines a functor $I\to \mathcal C$, and we may form the iterated colimit
\[
\colim_{x\in I} C_x = \colim_{x\in I} \colim_{q\in F(x)} D(x,q).
\]

On the other hand, for each $(x,q)\in \int_I F$ the structure morphism
$D(x,q)\to \colim_{(y,u)\in \int_I F} D(y,u)$
forms a cocone over the diagram $q\mapsto D(x,q)$ indexed by $F(x)$.
Hence, by the universal property of $C_x=\colim_{q\in F(x)} D(x,q)$,
they induce a canonical morphism
\[
C_x\longrightarrow \colim_{(y,u)\in \int_I F} D(y,u).
\]
These morphisms are compatible with $x\leq y$, hence induce a canonical morphism
\begin{equation}\label{eq:iterated-colim-to-grothendieck}
\colim_{x\in I}\ \colim_{q\in F(x)} D(x,q)
\longrightarrow
\colim_{(y,u)\in \int_I F} D(y,u).
\end{equation}

The following result is a standard consequence of the general Grothendieck construction for colimits.

\begin{proposition}\label{prop:iterated-colim-grothendieck}
The canonical morphism \eqref{eq:iterated-colim-to-grothendieck} is an isomorphism,
natural in $D$.
\end{proposition}

\subsection{A Fubini-type finality criterion for posets}\label{appendix:fubini-posets}

Let $I$ and $P$ be posets.
Let $F\colon I \to 2^P$ be a map such that
$F(x)$ is a downset of $P$ for every $x\in I$ and that $F(x)\subseteq F(y)$ for all $x\leq y$. 
In particular, it defines a functor $F\colon I\to {\rm Poset}$.  
Set
\[
Q:=\bigcup_{x\in I} F(x)\subseteq P,
\]
viewed as a full subposet of $P$.
Then there is an order-preserving map
\[
\pi\colon \int_I F\longrightarrow Q,\qquad (x,q)\longmapsto q.
\]
Indeed, if $(x,q)\le (y,u)$ in $\int_I F$, then $x\le y$ and $F(x\le y)(q)\le u$ in $F(y)$.
Since $F(x\le y)$ is the inclusion, this means $q\le u$ in $P$, hence also in $Q$.

Let $\mathcal{C}$ be a cocomplete category and $D\colon Q\to \mathcal{C}$ a functor. 
Then, the pullback along $\pi$ induces a canonical map 
\begin{equation}\label{eq:int-to-Q}
  \colim_{(x,q)\in \int_I F} D(q)
  \longrightarrow
  \colim_{q\in Q} D(q)
\end{equation}
natural in $D$, where we use $(\pi^*D)(x,q) = D(q)$. 

\begin{proposition}\label{prop:finality-int-to-Q}
If for every $q\in Q$ the subposet
\[
I_q:=\{x\in I\mid q\in F(x)\}\subseteq I
\]
is connected, then \eqref{eq:int-to-Q} is an isomorphism. 
\end{proposition}

\begin{proof}
Fix $q\in Q$. Recall that the comma category $(q\downarrow \pi)$ is the poset whose elements are pairs
$(x,p)\in \int_I F$ equipped with a morphism $q\to\pi(x,p)=p$ in $Q$.
Equivalently,
\[
(q\downarrow \pi)=\{(x,p)\in \int_I F \mid x\in I,\ p\in F(x),\ q\le p\}. 
\]

By a standard cofinality argument for colimits (see, e.g., \cite[Chapter~IX]{MacLane}),
to show that \eqref{eq:int-to-Q} is an isomorphism it suffices to prove that
$\pi$ is final, 
i.e., that $(q\downarrow \pi)$ is non-empty and connected for every $q\in Q$. 
Note that $(q\downarrow \pi)$ is always non-empty: since $q\in Q=\bigcup_{x\in I}F(x)$, there exists $x_0\in I$ with $q\in F(x_0)$, hence $(x_0,q)\in (q\downarrow \pi)$.

For each $x\in I_q$, consider the full subposet
\[
(q\downarrow \pi)_x:=\{(x,p)\in(q\downarrow \pi)\}
\]
consisting of elements with first coordinate $x$.
Let $(x,p)\in(q\downarrow \pi)_x$. 
Then $p\in F(x)$ and $q\le p$ in $Q$.
Since $F(x)$ is a downset in $Q$, it follows that $q\in F(x)$.
Hence $(x,q)\in(q\downarrow \pi)_x$.
Moreover, $(x,q)$ is a minimum element of $(q\downarrow \pi)_x$, hence $(q\downarrow \pi)_x$ is connected.
Consequently, $(q\downarrow \pi)$ is connected if and only if $I_q$ is connected.
Since $I_q$ is connected by assumption, $(q\downarrow \pi)$ is connected.
\end{proof}

Via the isomorphism \eqref{eq:iterated-colim-to-grothendieck} applied to the pulled-back diagram
$\pi^{*}D\colon \int_I F\to \mathcal C$, 
we may identify \eqref{eq:int-to-Q} with the canonical map
\begin{equation}\label{eq:iterated-colim-to-Q}
    \colim_{x\in I}\ \colim D|_{F(x)} \longrightarrow \colim D.
\end{equation}
Consequently, under the assumption of Proposition~\ref{prop:finality-int-to-Q},
the map \eqref{eq:iterated-colim-to-Q} is an isomorphism.

\section*{Acknowledgements}
This work is supported by JSPS Grant-in-Aid for Transformative Research Areas (A) (22H05105). 
The author would like to thank Shunsuke Tada for helpful comments and discussions.

\bibliographystyle{alpha} 
\bibliography{main.bib}

@book{MacLane,
        isbn = {978-0-387-98403-2},
	pages = {318},
author={Mac Lane, Saunders},
	publisher = {Springer},
address ={New York, NY},
	series = {London Mathematical Society Student Texts},
	title = {Categories for the Working Mathematician},
	DOI = {https://doi.org/10.1007/978-1-4757-4721-8},
	year = {1978}
	}

@article {AET25,
    AUTHOR = {Aoki, Toshitaka and Escolar, Emerson G. and Tada, Shunsuke},
     TITLE = {Summand-injectivity of interval covers and monotonicity of
              interval resolution global dimensions},
   JOURNAL = {J. Appl. Comput. Topol.},
  FJOURNAL = {Journal of Applied and Computational Topology},
    VOLUME = {9},
      YEAR = {2025},
    NUMBER = {2},
     PAGES = {Paper No. 13},
      ISSN = {2367-1726,2367-1734},
   MRCLASS = {99-06},
  MRNUMBER = {4908881},
       DOI = {10.1007/s41468-025-00210-2},
       URL = {https://doi.org/10.1007/s41468-025-00210-2}
}

@article{HNH+16,
  title={Hierarchical structures of amorphous solids characterized by persistent homology},
  author={Hiraoka, Yasuaki and Nakamura, Takenobu and Hirata, Akihiko and Escolar, Emerson G. and Matsue, Kaname and Nishiura, Yasumasa},
  journal={Proceedings of the National Academy of Sciences},
  volume={113},
  number={26},
  pages={7035--7040},
  year={2016},
  publisher={National Acad Sciences}
}

@article{BBD,
  title={Faisceaux pervers},
  author={Beilinson, Alexander A and Bernstein, Joseph and Deligne, Pierre},
  journal={Ast{\'e}risque},
  volume=100,
  pages={5--171},
  year=1982,
  publisher={Springer Netherlands}
}

@article {Lesnick15,
    AUTHOR = {Lesnick, Michael},
     TITLE = {The theory of the interleaving distance on multidimensional
              persistence modules},
   JOURNAL = {Found. Comput. Math.},
  FJOURNAL = {Foundations of Computational Mathematics. The Journal of the
              Society for the Foundations of Computational Mathematics},
    VOLUME = {15},
      YEAR = {2015},
    NUMBER = {3},
     PAGES = {613--650},
      ISSN = {1615-3375,1615-3383},
   MRCLASS = {55N35},
  MRNUMBER = {3348168},
MRREVIEWER = {Peter\ Bubenik},
       DOI = {10.1007/s10208-015-9255-y},
       URL = {https://doi.org/10.1007/s10208-015-9255-y},
}

@article {CZ05,
    AUTHOR = {Zomorodian, Afra and Carlsson, Gunnar},
     TITLE = {Computing persistent homology},
   JOURNAL = {Discrete Comput. Geom.},
  FJOURNAL = {Discrete \& Computational Geometry. An International Journal
              of Mathematics and Computer Science},
    VOLUME = {33},
      YEAR = {2005},
    NUMBER = {2},
     PAGES = {249--274},
      ISSN = {0179-5376,1432-0444},
   MRCLASS = {55N99 (55-04 55U10)},
  MRNUMBER = {2121296},
MRREVIEWER = {Donald\ M.\ Davis},
       DOI = {10.1007/s00454-004-1146-y},
       URL = {https://doi.org/10.1007/s00454-004-1146-y},
}

@inproceedings{CCSGGO09,
  title={Proximity of persistence modules and their diagrams},
  author={Fr{\'e}d{\'e}ric Chazal and David Cohen-Steiner and Marc Glisse and Leonidas J. Guibas and Steve Oudot},
  booktitle={SCG '09},
  year={2009},
  url={https://api.semanticscholar.org/CorpusID:840484}
}

@article {CSEH07,
    AUTHOR = {Cohen-Steiner, David and Edelsbrunner, Herbert and Harer,
              John},
     TITLE = {Stability of persistence diagrams},
   JOURNAL = {Discrete Comput. Geom.},
  FJOURNAL = {Discrete \& Computational Geometry. An International Journal
              of Mathematics and Computer Science},
    VOLUME = {37},
      YEAR = {2007},
    NUMBER = {1},
     PAGES = {103--120},
      ISSN = {0179-5376,1432-0444},
   MRCLASS = {68U05 (55N05)},
  MRNUMBER = {2279866},
       DOI = {10.1007/s00454-006-1276-5},
       URL = {https://doi.org/10.1007/s00454-006-1276-5},
}

@article {BdSS15,
    AUTHOR = {Bubenik, Peter and de Silva, Vin and Scott, Jonathan},
     TITLE = {Metrics for generalized persistence modules},
   JOURNAL = {Found. Comput. Math.},
  FJOURNAL = {Foundations of Computational Mathematics. The Journal of the
              Society for the Foundations of Computational Mathematics},
    VOLUME = {15},
      YEAR = {2015},
    NUMBER = {6},
     PAGES = {1501--1531},
      ISSN = {1615-3375,1615-3383},
   MRCLASS = {55N35 (55U10)},
  MRNUMBER = {3413628},
MRREVIEWER = {Mikael\ Vejdemo Johansson},
       DOI = {10.1007/s10208-014-9229-5},
       URL = {https://doi.org/10.1007/s10208-014-9229-5},
}

@article {dSMS18,
    AUTHOR = {de Silva, V. and Munch, E. and Stefanou, A.},
     TITLE = {Theory of interleavings on categories with a flow},
   JOURNAL = {Theory Appl. Categ.},
  FJOURNAL = {Theory and Applications of Categories},
    VOLUME = {33},
      YEAR = {2018},
     PAGES = {Paper No. 21, 583--607},
      ISSN = {1201-561X},
   MRCLASS = {18D05 (18C10 18D10 18D20 55N35)},
  MRNUMBER = {3812461},
MRREVIEWER = {Julia\ Bergner},
}

@misc{BCM20,
      title={A Relative Theory of Interleavings}, 
      author={Magnus Bakke Botnan and Justin Curry and Elizabeth Munch},
      year={2020},
      note   = {arXiv preprint arXiv:2004.14286},
      url={https://arxiv.org/abs/2004.14286}, 
}

@article {MN26,
    AUTHOR = {McFaddin, Patrick K. and Needham, Tom},
     TITLE = {Interleaving distances, monoidal actions and 2-categories},
   JOURNAL = {Algebr. Geom. Topol.},
  FJOURNAL = {Algebraic \& Geometric Topology},
    VOLUME = {26},
      YEAR = {2026},
    NUMBER = {1},
     PAGES = {227--281},
      ISSN = {1472-2747,1472-2739},
   MRCLASS = {55N31 (18M05 18N10)},
  MRNUMBER = {5018430},
       DOI = {10.2140/agt.2026.26.227},
       URL = {https://doi.org/10.2140/agt.2026.26.227},
}

@article {EMY23,
    AUTHOR = {Escolar, Emerson G. and Meehan, Killian and Yoshiwaki, Michio},
     TITLE = {Interleavings and matchings as representations},
   JOURNAL = {Appl. Algebra Engrg. Comm. Comput.},
  FJOURNAL = {Applicable Algebra in Engineering, Communication and
              Computing},
    VOLUME = {34},
      YEAR = {2023},
    NUMBER = {6},
     PAGES = {965--993},
      ISSN = {0938-1279,1432-0622},
   MRCLASS = {16G30 (55N99)},
  MRNUMBER = {4652879},
MRREVIEWER = {Rui\ Miguel\ Saramago},
       DOI = {10.1007/s00200-021-00530-7},
       URL = {https://doi.org/10.1007/s00200-021-00530-7},
}

@misc{Tada25,
      title={Stability of Bipath Persistence Diagrams}, 
      author={Shunsuke Tada},
      year={2025},
      note   = {arXiv preprint arXiv:2503.01614},
      url={https://arxiv.org/abs/2503.01614}, 
}

@misc{Bjerkevik25,
      title={Stabilizing decomposition of multiparameter persistence modules}, 
      author={Håvard Bakke Bjerkevik},
      year={2025},
      note   = {arXiv preprint arXiv:2305.15550v3},
      url={https://arxiv.org/abs/2305.15550}, 
}

@article {Bjerkevik21,
    AUTHOR = {Bjerkevik, H\aa vard Bakke},
     TITLE = {On the stability of interval decomposable persistence modules},
   JOURNAL = {Discrete Comput. Geom.},
  FJOURNAL = {Discrete \& Computational Geometry. An International Journal
              of Mathematics and Computer Science},
    VOLUME = {66},
      YEAR = {2021},
    NUMBER = {1},
     PAGES = {92--121},
      ISSN = {0179-5376,1432-0444},
   MRCLASS = {55N31},
  MRNUMBER = {4270636},
MRREVIEWER = {Yuichi\ Ike},
       DOI = {10.1007/s00454-021-00298-0},
       URL = {https://doi.org/10.1007/s00454-021-00298-0},
}

@article {BS25,
    AUTHOR = {Bauer, Ulrich and Scoccola, Luis},
     TITLE = {Multi-parameter persistence modules are generically
              indecomposable},
   JOURNAL = {Int. Math. Res. Not. IMRN},
  FJOURNAL = {International Mathematics Research Notices. IMRN},
      YEAR = {2025},
    NUMBER = {5},
     PAGES = {Paper No. rnaf034, 31},
      ISSN = {1073-7928,1687-0247},
   MRCLASS = {55N31},
  MRNUMBER = {4870578},
MRREVIEWER = {Chengyuan\ Wu},
       DOI = {10.1093/imrn/rnaf034},
       URL = {https://doi.org/10.1093/imrn/rnaf034},
}

@article {BBK20,
    AUTHOR = {Bjerkevik, H\aa vard Bakke and Botnan, Magnus Bakke and
              Kerber, Michael},
     TITLE = {Computing the interleaving distance is {NP}-hard},
   JOURNAL = {Found. Comput. Math.},
  FJOURNAL = {Foundations of Computational Mathematics. The Journal of the
              Society for the Foundations of Computational Mathematics},
    VOLUME = {20},
      YEAR = {2020},
    NUMBER = {5},
     PAGES = {1237--1271},
      ISSN = {1615-3375,1615-3383},
   MRCLASS = {55N31 (15A83 68Q17)},
  MRNUMBER = {4156997},
MRREVIEWER = {Henry\ Hugh\ Adams},
       DOI = {10.1007/s10208-019-09442-y},
       URL = {https://doi.org/10.1007/s10208-019-09442-y},
}

@misc{GM22,
      title={Galois Connections in Persistent Homology}, 
      author={Aziz Burak Gulen and Alexander McCleary},
      year={2023},
      note   = {arXiv preprint arXiv:2201.06650v5},
      url={https://arxiv.org/abs/2201.06650}, 
}

@misc{CKM24,
      title={The Generalized Rank Invariant: M\"obius invertibility, Discriminating Power, and Connection to Other Invariants}, 
      author={Nathaniel Clause and Woojin Kim and Facundo Mémoli},
      year={2024},
      note   = {arXiv preprint arXiv:2207.11591v5},
      url={https://arxiv.org/abs/2207.11591}, 
}

@misc{KS25,
      title={Interleaving Distance as an Edit distance}, 
      author={Woojin Kim and Won Seong},
      year={2025},
      note   = {arXiv preprint arXiv:2509.24233},
      url={https://arxiv.org/abs/2509.24233}, 
}

@incollection {BL20,
    AUTHOR = {Bauer, Ulrich and Lesnick, Michael},
     TITLE = {Persistence diagrams as diagrams: a categorification of the
              stability theorem},
 BOOKTITLE = {Topological data analysis---the {A}bel {S}ymposium 2018},
    SERIES = {Abel Symp.},
    VOLUME = {15},
     PAGES = {67--96},
 PUBLISHER = {Springer, Cham},
      YEAR = {2020},
      ISBN = {978-3-030-43407-6; 978-3-030-43408-3},
   MRCLASS = {55N31},
  MRNUMBER = {4338669},
MRREVIEWER = {Haibin\ Hang},
       DOI = {10.1007/978-3-030-43408-3\_3},
       URL = {https://doi.org/10.1007/978-3-030-43408-3_3},
}

@article {MP20,
    AUTHOR = {McCleary, Alex and Patel, Amit},
     TITLE = {Bottleneck stability for generalized persistence diagrams},
   JOURNAL = {Proc. Amer. Math. Soc.},
  FJOURNAL = {Proceedings of the American Mathematical Society},
    VOLUME = {148},
      YEAR = {2020},
    NUMBER = {7},
     PAGES = {3149--3161},
      ISSN = {0002-9939,1088-6826},
   MRCLASS = {55N31 (05E99)},
  MRNUMBER = {4099800},
MRREVIEWER = {Magnus\ Bakke\ Botnan},
       DOI = {10.1090/proc/14929},
       URL = {https://doi.org/10.1090/proc/14929},
}

@article{HIY22,
  title={Algebraic stability theorem for derived categories of zigzag persistence modules},
  author={Hiraoka, Yasuaki and Ike, Yuichi and Yoshiwaki, Michio},
  journal={Journal of Topology and Analysis},
  volume={16},
  number={02},
  pages={555--585},
  year={2022},
  publisher={World Scientific Publishing Company},
  doi={10.1142/S1793525322500091}
}

@misc{BP25,
      title={An isometry theorem for persistent homology of circle-valued functions}, 
      author={Nathan Broomhead and Mariam Pirashvili},
      year={2025},
      note   = {arXiv preprint arXiv:2506.02999},
      url={https://arxiv.org/abs/2506.02999}, 
}

@article {BL18,
    AUTHOR = {Botnan, Magnus Bakke and Lesnick, Michael},
     TITLE = {Algebraic stability of zigzag persistence modules},
   JOURNAL = {Algebr. Geom. Topol.},
  FJOURNAL = {Algebraic \& Geometric Topology},
    VOLUME = {18},
      YEAR = {2018},
    NUMBER = {6},
     PAGES = {3133--3204},
      ISSN = {1472-2747,1472-2739},
   MRCLASS = {55N35 (55U99)},
  MRNUMBER = {3868218},
MRREVIEWER = {Ellen\ Gasparovic},
       DOI = {10.2140/agt.2018.18.3133},
       URL = {https://doi.org/10.2140/agt.2018.18.3133},
}

@misc{BBS24,
      title={Decomposing zero-dimensional persistent homology over rooted tree quivers}, 
      author={Riju Bindua and Thomas Brüstle and Luis Scoccola},
      year={2024},
      note   = {arXiv preprint arXiv:2411.19319},
      url={https://arxiv.org/abs/2411.19319}, 
}

@article {HR24,
    AUTHOR = {Hanson, Eric J. and Rock, Job Daisie},
     TITLE = {Decomposition of pointwise finite-dimensional {$\Bbb{S}^1$}
              persistence modules},
   JOURNAL = {J. Algebra Appl.},
  FJOURNAL = {Journal of Algebra and its Applications},
    VOLUME = {23},
      YEAR = {2024},
    NUMBER = {3},
     PAGES = {Paper No. 2450054, 24},
      ISSN = {0219-4988,1793-6829},
   MRCLASS = {16G20 (55N31)},
  MRNUMBER = {4688822},
       DOI = {10.1142/S0219498824500543},
       URL = {https://doi.org/10.1142/S0219498824500543},
}

@misc{GZ25,
  author = {Gao, Xiaowen and Zhao, Minghui},
  title  = {Isometry Theorem for Continuous Quiver of Type $\tilde{A}$},
  year   = {2025},
  note   = {arXiv preprint arXiv:2412.12462v2},
}

@article {RZ23,
    AUTHOR = {Rock, Job Daisie and Zhu, Shijie},
     TITLE = {Continuous {N}akayama representations},
   JOURNAL = {Appl. Categ. Structures},
  FJOURNAL = {Applied Categorical Structures. A Journal Devoted to
              Applications of Categorical Methods in Algebra, Analysis,
              Computer Science, Logic, Order and Topology},
    VOLUME = {31},
      YEAR = {2023},
    NUMBER = {5},
     PAGES = {Paper No. 44, 25},
      ISSN = {0927-2852,1572-9095},
   MRCLASS = {16G10 (16G20 26A48 37E05 37E10)},
  MRNUMBER = {4649409},
MRREVIEWER = {Frauke\ M.\ Bleher},
       DOI = {10.1007/s10485-023-09748-7},
       URL = {https://doi.org/10.1007/s10485-023-09748-7},
}

@incollection {ELZ02,
    AUTHOR = {Edelsbrunner, Herbert and Letscher, David and Zomorodian,
              Afra},
     TITLE = {Topological persistence and simplification},
      NOTE = {Discrete and computational geometry and graph drawing
              (Columbia, SC, 2001)},
   JOURNAL = {Discrete Comput. Geom.},
  FJOURNAL = {Discrete \& Computational Geometry. An International Journal
              of Mathematics and Computer Science},
    VOLUME = {28},
      YEAR = {2002},
    NUMBER = {4},
     PAGES = {511--533},
      ISSN = {0179-5376,1432-0444},
   MRCLASS = {52B55 (65D18)},
  MRNUMBER = {1949898},
MRREVIEWER = {H.\ W.\ Guggenheimer},
       DOI = {10.1007/s00454-002-2885-2},
       URL = {https://doi.org/10.1007/s00454-002-2885-2},
}

@article {CCR13,
    AUTHOR = {Chan, Joseph Minhow and Carlsson, Gunnar and Rabadan, Raul},
     TITLE = {Topology of viral evolution},
   JOURNAL = {Proc. Natl. Acad. Sci. USA},
  FJOURNAL = {Proceedings of the National Academy of Sciences of the United
              States of America},
    VOLUME = {110},
      YEAR = {2013},
    NUMBER = {46},
     PAGES = {18566--18571},
      ISSN = {0027-8424,1091-6490},
   MRCLASS = {92D15},
  MRNUMBER = {3153945},
       DOI = {10.1073/pnas.1313480110},
       URL = {https://doi.org/10.1073/pnas.1313480110},
}

@article{QTT+19,
	author = {Talha Qaiser and Yee-Wah Tsang and Daiki Taniyama and Naoya Sakamoto and Kazuaki Nakane and David Epstein and Nasir Rajpoot},
	doi = {https://doi.org/10.1016/j.media.2019.03.014},
	issn = {1361-8415},
	journal = {Medical Image Analysis},
	pages = {1-14},
	title = {Fast and accurate tumor segmentation of histology images using persistent homology and deep convolutional features},
	url = {https://www.sciencedirect.com/science/article/pii/S1361841518302688},
	volume = {55},
	year = {2019}
}

@inproceedings{HKWNU17,
  author    = {Hofer, Christoph and Kwitt, Roland and Niethammer, Marc and Uhl, Andreas},
  title     = {Deep Learning with Topological Signatures},
  booktitle = {Advances in Neural Information Processing Systems 30 (NeurIPS 2017)},
  pages     = {1634--1644},
  year      = {2017}
}

@article{AAF19,
  author  = {Aktas, Mehmet Emin and Akbas, Esra and El Fatmaoui, Abdoulaye},
  title   = {Persistent homology of networks: methods and applications},
  journal = {Applied Network Science},
  year    = {2019},
  volume  = {4},
  number  = {1},
  pages   = {61},
  doi     = {10.1007/s41109-019-0169-9}
}

@book {Hovey99,
    AUTHOR = {Hovey, Mark},
     TITLE = {Model categories},
    SERIES = {Mathematical Surveys and Monographs},
    VOLUME = {63},
 PUBLISHER = {American Mathematical Society, Providence, RI},
      YEAR = {1999},
     PAGES = {xii+209},
      ISBN = {0-8218-1359-5},
   MRCLASS = {55U35 (18D15 18G30 18G55)},
  MRNUMBER = {1650134},
MRREVIEWER = {Teimuraz\ Pirashvili},
}

@article {Kuhn94,
    AUTHOR = {Kuhn, Nicholas J.},
     TITLE = {Generic representations of the finite general linear groups
              and the {S}teenrod algebra. {II}},
   JOURNAL = {$K$-Theory},
  FJOURNAL = {$K$-Theory. An Interdisciplinary Journal for the Development,
              Application, and Influence of $K$-Theory in the Mathematical
              Sciences},
    VOLUME = {8},
      YEAR = {1994},
    NUMBER = {4},
     PAGES = {395--428},
      ISSN = {0920-3036,1573-0514},
   MRCLASS = {55S10 (19A49 20G05)},
  MRNUMBER = {1300547},
MRREVIEWER = {Haynes\ R.\ Miller},
       DOI = {10.1007/BF00961409},
       URL = {https://doi.org/10.1007/BF00961409},
}

@article {BL24,
    AUTHOR = {Blumberg, Andrew J. and Lesnick, Michael},
     TITLE = {Stability of 2-parameter persistent homology},
   JOURNAL = {Found. Comput. Math.},
  FJOURNAL = {Foundations of Computational Mathematics. The Journal of the
              Society for the Foundations of Computational Mathematics},
    VOLUME = {24},
      YEAR = {2024},
    NUMBER = {2},
     PAGES = {385--427},
      ISSN = {1615-3375,1615-3383},
   MRCLASS = {55N31 (62R40)},
  MRNUMBER = {4733354},
MRREVIEWER = {Mauricio\ Che},
       DOI = {10.1007/s10208-022-09576-6},
       URL = {https://doi.org/10.1007/s10208-022-09576-6},
}

@misc{MM17,
      title={An Isometry Theorem for Generalized Persistence Modules}, 
      author={Killian Meehan and David Meyer},
      year={2017},
      note   = {arXiv preprint arXiv:1710.02858},
      url={https://arxiv.org/abs/1710.02858}, 
}

\end{document}